\pdfoutput=1
\documentclass[leqno,dvipsnames,10pt]{article}
\usepackage[normal]{mystyle}
\usepackage{standard_macros}

\usepackage[margin=1.5in]{geometry}
\AtBeginDocument{
	\newgeometry{
		textheight=9in,
		textwidth=6in,
		top=1in,
		headheight=14pt,
		headsep=25pt,
		footskip=30pt
	}
}

\newcommand{\isovposets}{\catname{Isov}}
\newcommand{\Deltacat}{\catname{$\pmb \Delta$}}

\newcommand{\nreal}{\text{n}\mathbb{r}}
\newcommand{\real}{\mathbb{r}}

\title{Simplicial $C_2$-Isovariant Homotopy
}

\author{
	Santiago Toro Oquendo \\
	Université Breatgne Occidentale \\
	LMBA - Brest\\
	\href{mailto:santiago.torooquendo@univ-brest.fr}{santiago.torooquendo@univ-brest.fr} \\
}

\begin{document}
	
\maketitle

\begin{abstract}
	This article presents a novel approach to construct a model category structure designed to model the homotopy theory of spaces equipped with an action by the group $C_2$, where morphisms are considered to be isovariant. Our methodology centers on simplicial techniques. We replace the conventional simplex category $\Deltacat$ with a modified category $C_2\Deltacat$ and then delve into the study of presheaves of sets on $C_2\Deltacat$. To establish the model structure, we employ Cisinski's methodologies for model structures in categories of presheaves, in particular we use an analogous idea to the one employed by Cisinski and Moerdijk in the construction of a model category structure for Dendroidal Sets \cite{Cisinski2011}. This theory distinguishes itself from prior work, such as Yeakel's \cite{Yeakel2019IET}, which primarily focuses on a more topological context. Our approach brings a new perspective to the study of isovariant homotopy theory for $C_2$-spaces.
\end{abstract}

	
\setcounter{tocdepth}{2}


	\section{Introduction}

Let $G$ be a finite group and consider $X$ and $Y$ to be compactly generated spaces endowed with continuous left actions by $G$. Recall that an equivariant map $f:X\lto Y$ is a continuous map that preserves the $G$-action in the sense that $f(g\cdot x)=g\cdot f(x)$. The map $f$ is said to be \emph{isovariant} if it preserves isotropy groups in a strict sense, so that, $G_x = G_{f(x)}$ for all $x\in X$. 

Perhaps one of the first works in which isovariant homotopy plays an influential role is in the works of Browder and Quinn \cite{Browder1973}. They used isovariant homotopy theory to provide an approach to surgery theory in $G$-manifolds and stratified sets. By using surgery methods, they managed to obtain isovariant versions of the $s$-cobordism theorem and a surgery exact sequence. These ideas were later on developed further by Schultz \cite{Schultz1992}. In particular they addresses natural questions about the computation of isovariant structures on manifolds and the natural applications of these (see also \cite{SchultzDovermann1990}). \\
Recently, Sarah Yeakel in \cite{Yeakel2019IET} has shown that the category of $G$-spaces with isovariant maps $G\Top_{\text{isov}}^{\lhd}$, in which a formal terminal object has been attached,  admits a Quillen model structure and this latter is moreover Quillen equivalent to a model category of diagrams, thus giving rise to an isovariant Elmendorf's theorem. Yeakel's main result reads as follows:  
\begin{theorem}[{\cite[\nopp 3.1]{Yeakel2019IET}}]
	Let $G$ be a finite group. Then there is a Quillen equivalence:
	\[\adjto{\Delta_G^{\bullet}\otimes_{\catcal{L}_G}-}{\Fun[\catcal{L}_G^{\op},\Top]}{G\Top_{\emph{isov}}^{\lhd}}{\Map_{\emph{isov}}(\Delta_G^{\bullet},-)}\]
\end{theorem} 

To illustrate one difference between isovariant homotopy theory and the more classical equivariant homotopy theory one may consider Example 2.5 in \cite{Yeakel2019IET}. In such example, one considers the unit disk $\mathbb{D}^2$ with action by the cyclic group of order $2$, denoted by $C_2$. This action is given by reflection with respect to the $y$-axis. The unit disk then admits many equivariant and isovariant maps from the one-point space $\ast$ endowed with trivial $C_2$-action. However there is no isovariant map from $\mathbb{D}^2$ into $\ast$. Let us fix some $f:\ast \to \mathbb{D}^2$. 

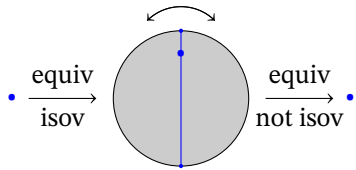
\begin{center}
	\begin{tikzpicture}[scale=.45]
			\node[blue] at (-5,0) {$\bullet$};
			\draw[->] (-4.5,0)--(-2.5,0);
			\node[above] at (-3.5,0) {equiv};
			\node[below] at (-3.5,0) {isov};
			\draw[<-] (-1,2.3) to [out=45,in=135] (1,2.3);
			\draw[->] (-1,2.3) to [out=45,in=135] (1,2.3);
			\filldraw[fill=black!20] (0,0) circle (2cm);
			\filldraw[blue] (0,2) circle (1.3pt);
			\filldraw[blue] (0,-2) circle (1.3pt);
			\node[blue] at (0,1.3) {$\bullet$};
			\draw[blue] (0,2)--(0,-2);
			\draw[->] (2.5,0)--(4.5,0);
			\node[above] at (3.5,0) {equiv};
			\node[below] at (3.5,0) {not isov};
			\node[blue] at (5,0) {$\bullet$};
		\end{tikzpicture}
\end{center}      
The classical equivariant homotopy theory sees the map $f$ as an equivariant weak equivalence since the induced maps on fixed points $f^{e}$ and $f^{C_2}$ are both weak equivalences of spaces. Nevertheless, $f$ is not a weak equivalence in $C_2\Top_{\text{isov}}^{\lhd}$ as shown in \cite[Example 2.5 ]{Yeakel2019IET}. 

In this article we propose a different approach to isovariant homotopy by using simplicial techniques. We consider only the case where $G=C_2$ and study presheaves of sets over an \enquote{isovariant simplex category} $G\Deltacat$. We aim to generalize these ideas to allow actions of any profinite group $G$ in a subsequent paper. 
The category $G\Deltacat$ is analog to the usual simplex category $\Deltacat$ and its objects may be thought of as building blocks for isovariant $C_2$-spaces. This building blocks can also be seen in the context of stratified homotopy theory \cite{douteau2020simplicial} where simplices are stratified by isotropy. 

By using techniques from \cite{Cisinski2016LPCMDTH} and analogue results from the construction of model category structures on Dendroidal Sets \cite{Cisinski2011} we show that the category $\ssetisov$ of presheaves of sets over the category $G\Deltacat$ admits a cofibrantly generated model category structure. Our main result is the following: 

\begin{theorem}[\Cref{Corollary:modelstructureabsolute}]
	The category $\ssetisov$ is endowed with a cofibrantly generated model category structure in which the cofibrations are the normal monomorphisms, the fibrant objects are the $\interval$-fibrant objects, and the fibrations between fibrant objects are the isovariant fibrations. Weak equivalences are given by the isovariant homotopy equivalences and in particular for any $X$ normal and $Y$ fibrant we have that 
	\[[X,Y]_{\emph{isov}} = \Hom_{\catname{Ho}(\ssetisov)}(X,Y).\]
\end{theorem}

\noindent
\textbf{Organization of the paper.} In this paper we use freely the language of model categories and in particular we assume familiarity with Cisinski's model structures \cite{Cisinski2016LPCMDTH}. In \Cref{section:c2isovariantsimplexcategory} we construct the category $C_2\Deltacat$ which is an isovariant analog to the simplex category $\Deltacat$. In particular we prove that $C_2 \Deltacat$ is a skeletal category in the sense of Cisinski. In \Cref{section:isovariant simplicial sets} we study presheaves of sets over the category $C_2\Deltacat$ and we call this objects \emph{isovariant simplicial sets}. We derive some of its natural properties and a canonical singular/realization adjunction between the category of isovariant simplicial sets $\catname{sIsov}_{C_2}$ and the category of $C_2$-topological spaces together with isovariant maps $C_2\Top_{\text{isov}}$. We also construct a few natural examples analog to simplicial spheres and horns and an explicit exact cylinder functor for the category $\catname{sIsov}_{C_2}$. Finally in \Cref{Section: model structure for ssetisov} we construct a cofibrantly generated model category structure on the category of isovariant simplicial sets by using similar ideas to the construction of the model category structure on Dendroidal Sets \cite{Cisinski2011}.   

%

	\section{A $C_2$-Isovariant Simplex Category}\label{section:c2isovariantsimplexcategory}
Let $P = \set{x_0<x_1<\cdots <x_n}$ be a finite non-empty totally ordered poset and denote by $C_2=\set{e,\sigma}$ the cyclic group of order $2$. We will consider $C_2$ with the trivial order where $e$ and $\sigma$ are not comparable. Then the cartesian product $P\times G$ has a natural product order induced by that of $P$: 
	\begin{equation}\label{eq:diagram poset Pk}
		\begin{tikzcd}[column sep=small]
			{(x_0,e)} & {(x_1,e)} & \cdots & {(x_n,e)} \\
			{(x_0,\sigma)} & {(x_1,\sigma)} & \cdots & {(x_n,\sigma)}
			\arrow[from=1-1, to=1-2]
			\arrow[from=1-2, to=1-3]
			\arrow[from=1-3, to=1-4]
			\arrow[from=2-1, to=2-2]
			\arrow[from=2-2, to=2-3]
			\arrow[from=2-3, to=2-4]
		\end{tikzcd}
	\end{equation}
For any $0\leq k\leq n+1$, we define $P_k$ to be the quotient $P_k =\bigslant{(P \times C_2)}{\sim}_k$.  Where $\sim_k$ is the equivalence relation generated by $(x_i,e) \sim_k (x_j,\sigma)$ if and only if $i=j\geq k$.

\begin{remark}\label{Remark:ApproachPkfornk} \leavevmode
	\begin{enumerate}[label=\stlabel{Remark:ApproachPkfornk}, ref=\arabic*]
		\item \label{Remark:ApproachPkfornk.1} Notice that the relation $\sim_k$ does not depend on $n$ and thus we can of course allow $k$ to be greater than $n+1$. In this case it is easy to show that for a finite non-empty totally ordered poset $P$ we have that $P_k \cong P_{k+1}$ for all $k\geq n+1$. 
		
		\item \label{Remark:ApproachPkfornk.2}  Let $[n]=\set{0<1<2< \cdots<n}$. Then, for $k=0$ and $k=n+1$ we have isomorphisms of posets: 
		\[ P_0 \cong [n] \quad \text{and} \quad P_{n+1}\cong [n]\coprod[n]\] 
The group $C_2$ acts canonically on $P_k$ by defining:  
		\begin{align*} 
			C_2 \times P_k &\lto P_k  \\ (g, [x_i,h]) &\longmapsto [x_i,gh] , \quad 0\leq i \leq n, \quad g,h \in C_2, \quad x_i \in P.
		\end{align*}
		This action is trivial when $k=0$ and if $k=n+1$ then the group $C_2$ acts by swapping components on $P_{n+1}\cong [n] \coprod [n]$. Aditionally, notice that the isotropy groups of elements $[(x_i,g)] \in P_k$ have orders: 
		\[\left| {C_2}_{[(x_i,g)]} \right|= \begin{cases}
			1 &\text{ if } 0 \leq i \leq k-1 \\ 2 &\text{ if } k \leq i \leq n
		\end{cases}\]
		
		\item \label{Remark:ApproachPkfornk.3} More generally, for any $0\leq k\leq n+1$ we have that $P_k \cong [n]_k$. This isomorphism is moreover an isomorphism of $C_2$-finite posets where $P_k$ and $[n]_k$ are both endowed with canonical $C_2$-actions as in \enumref{Remark:ApproachPkfornk}{2}. The $C_2$-poset $[n]_k$ can be then pictured as: 
	
	\begin{figure}[H]
		\centering
		\begin{tikzpicture}[align=center, 
			minimum size=5mm,
			bend angle=45] 
			\tikzset{node distance = 0.02cm and 0.7cm}
			\node (0e) {$[0,e]$};
			\node (1e) [right =of 0e] {$[1,e]$};
			\node (2e) [right =of 1e] {$\ldots$}; 
			\node (3e) [right =of 2e] {$[n,e]$};
			
			\path[->] (0e) edge (1e);
			\path[->] (1e) edge (2e);
			\path[->] (2e) edge (3e);
		\end{tikzpicture}
	\captionsetup{labelformat=empty}
	\caption{$k=0$}\label{Fig:k=0}
	\end{figure}

	\begin{figure}[H]
		\begin{minipage}[c]{0.49\linewidth}
				\begin{tikzpicture}[align=center, 
					minimum size=5mm,
					bend angle=45] 
					\tikzset{node distance = 0.06cm and 0.3cm}
					\node (00) {$[0,e]$} ;
					\node (10) [below right =of 00] {$[1,e]$};
					\node (20) [below right =of 10, rotate=-20] {$\ldots$};  
					\node (30) [below right = of 20] {$[k,e]$}; 
					\node (21) [below left =of 30,rotate=20] {$\ldots$};
					\node (11) [below left = of 21] {$[1,\sigma]$};
					\node (01) [below left = of 11] {$[0,\sigma]$};
					\node (40) [right =of 30]{$\ldots$};
					\node (50) [right =of 40]{$[n,e]$};
					
					\path[->] (00) edge (10);
 					\path[->] (10) edge (20);
					\path[->] (20) edge (30);
					\path[->] (01) edge (11);
					\path[->] (11) edge (21);
					\path[->] (21) edge (30);
					
					\path[->] (30) edge (40);
					\path[->] (40) edge (50);
			\end{tikzpicture}
		\end{minipage}   \hspace{20pt}             
		\begin{minipage}[c]{\linewidth}
			\begin{tikzpicture}[align=center, 
				minimum size=5mm,
				bend angle=45, allow upside down] 
				\tikzset{node distance = 0.7cm and 0.7cm}
				
				\node (10)  {$[1,e]$};
				\node (00) [left =of 10] {$[0,e]$} edge[->] (10); 
				\node (20) [ right =of 10] {$\ldots$} ; 
				\node (30) [right = of 20] {$[n,e]$};
			
				\node (31) [below = of 30] {$[n,\sigma]$};
				\node (21) [left =of 31] {$\ldots$};
				\node (11) [left = of 21] {$[1,\sigma]$};
				\node (01) [left = of 11] {$[0,\sigma]$};
			
				\path[->] (10) edge node(mid1)[midway]{} (20);
				\path[->] (20) edge (30);
		
				\path[->] (01) edge (11);
				\path[->] (11) edge node(mid2)[midway]{} (21);
				\path[->] (21) edge (31);

				\path[->] (mid1) edge[bend left= 40] node[midway,right]{$\sigma$}(mid2);
				\path[->] (mid2) edge[bend left =40] node[midway,left]{$\sigma$} (mid1);
			\end{tikzpicture}
		\end{minipage}
	\par 
	\begin{minipage}[t]{.49\linewidth}
		\captionsetup{labelformat=empty}
		\caption{$0<k<n+1$}\label{Fig:k between 0 and n+1}  
	\end{minipage}\hspace{0.05\linewidth}%
	\begin{minipage}[t]{.49\linewidth}  
		\captionsetup{labelformat=empty}
		\caption{$k=n+1$}\label{Fig:k=n+1} 
	\end{minipage} 
	\end{figure}
If $k=0$ then $[i,e]=[i,\sigma]$ for every $i=0,1,...,n$. In this case the $C_2$-action on $[n]_0$ is trivial and then $[n]_0 \cong [n]$. In the case where $0<k<n+1$, the action by $\sigma \in C_2$ on $[n]_k$ is swapping the two branches displayed in the left hand side figure. Additionally, notice that $[i,e]=[i,\sigma]$ for each $k \leq i \leq n$. The elements in the subset $\set{[k,e],[k+1],e},...,[n,e]$ will be called \emph{real vertices} of $[n]_k$. Similarly, elements in the complement of such set are called \emph{non real vertices}. Using this terminology, the object $[n]_0$ consist of only real vertices. \\
Finally, if $k=n+1$ then $[n]_{n+1} \cong [n]\coprod[n]$ since there are no nontrivial identifications in $[n] \times G$ with respect to the relation $\sim_{n+1}$. The action by $C_2$ is swapping the two disconnected branches displayed in right hand side figure. In this case $[n]_{n+1}$ consist of only non real vertices.
	\end{enumerate}
\end{remark}

\begin{notation}\label{notation:G-as-C2}
	From now on and unless otherwise stated the group $C_2$ will be denoted by $G$. If $G$ acts on a set $A$ and $a\in A$, we will denote by $G_a$ the corresponding isotropy group.   
\end{notation}

\begin{para}\label{para:Gdelta and some remarks} \normalfont 
		Let $G\isovposets$ be the category of finite isovariant $G$-posets whose objects are finite $G$-posets and whose morphisms are isovariant order preserving maps. That is, if $A$ and $B$ are two finite $G$-posets then a morphism $f: A \lto B$ in $G\isovposets$ satisfies: 
	\begin{enumerate}[label=\stlabel{para:Gdelta and some remarks}, ref=\arabic*]
		\item \label{para:Gdelta and some remarks.1} $f$ is order preserving so that for all $a\leq a'$ we have that $f(a) \leq f(a')$.
		\item \label{para:Gdelta and some remarks.2} $f$ is equivariant \ie, $f(g\cdot a)=g\cdot f(a)$ for all $a \in A$ and $g \in G$.
		\item \label{para:Gdelta and some remarks.3} $f$ strictly preserves isotropy groups, \ie, $G_a=G_{f(a)}$ for all $a \in A$.  
	\end{enumerate}	
	
	\begin{definition}\label{definition:Gdelta}
		We define $G\Deltacat$ as the category whose objects are the partially ordered $G$-sets of the form $[n]_k$. Morphisms are given as: 
		\[ \Hom_{G\Deltacat}([n]_k, [m]_l) \coloneqq \Hom_{G\isovposets}([n]_k, [m]_l).\]
		Composition of morphisms correspond then to the usual composition of maps in $G\isovposets$.
	\end{definition}	
%

The category $G\Deltacat$ is not very far from the simplex category $\Deltacat$ in the sense that it also admits a combinatorial description for its morphisms: Let us fix $[n]_k \in G\Deltacat$ with $0 \leq k \leq n+1$. Among all isovariant order preserving maps in $G\Deltacat$ there are special ones:
\end{para}

\begin{para}\label{Para: isovariant coface maps} \normalfont \textbf{Coface maps.} Recall that for any $0\leq i\leq n+1$ there are order-preserving coface maps $d^i: [n]\lto [n+1]$ given by sending:
	
	\begin{equation*}
		 j\longmapsto 
		\begin{cases}
			j \quad &\text{if} \quad j<i \\
			j+1 \quad &\text{if} \quad j\geq i		
		\end{cases}
	\end{equation*}
	These maps induce $G$-poset maps $\begin{tikzcd} {[n]} \times G  \arrow[r,"{(d^i_0,\idd)}"] &{[n+1]} \times G 	\end{tikzcd}$ and moreover if $i \geq k$ we have induced isovariant coface maps $d^i_0: [n]_k \lto [n+1]_k$ given by: 
	
	\begin{equation*}
	[j,g]\longmapsto 
		\begin{cases}
			[j,g] \quad &\text{if} \quad j<i \\
			[j+1,g] \quad &\text{if} \quad j\geq i		
		\end{cases}
	\end{equation*}
	Similarly if $i\leq k$	we also have induced isovariant coface maps $d_1^i: [n]_k \lto [n+1]_{k+1}$ defined in exactly the same way. Notice that $d_0^i$ and $d_1^i$ only differ in that they have different codomains.  
The coface maps having sub-index equal to zero correspond to coface maps omitting some \emph{real vertex} thus sometimes we will refer to $d_0^i$ as a \emph{real coface map}. Similarly, the coface maps having sub-index equal to one  $d_1^i$ correspond to coface maps omitting a \emph{non-real vertex} so that we will called them \emph{non-real coface maps}. Aditionally notice that for $i=k$ we have both, a real coface map $d_0^k: [n]_k \lto [n+1]_k$ and a non-real coface map $d_1^k: [n]_k \lto [n+1]_{k+1}$.

\end{para}

\begin{example}\label{Ex:cofacemaps from 21 to 31}
	Let us describe, geometrically, the real and non-real coface maps $d^i_{0}:[2]_1 \lto [3]_{1}$ for $1\leq i \leq 3$ and $d^i_1:[2]_1 \lto [3]_2$ for $i\leq 1$ respectively. In this case we then have real coface maps $d_0^1, d_0^2, d_0^3$ and non-real coface maps $d_1^0, d_1^1$.  
	\begin{figure}[h]
		\caption{{}} \label{fig:cofacemaps from 21 to 31}
		\begin{center}
			\begin{tikzpicture}[line join = round, line cap = round,baseline=1ex,scale=1.2, transform shape, font=\tiny]
				\coordinate [label=above:{$[0,e]$}] (00) at (0,{sqrt(2)},0);
				\coordinate [label=left:{$[3,e]$}] (3) at ({-.5*sqrt(3)},0,-.5);
				\coordinate [label=below right:{$[2,e]$}] (2) at (0,0,0.6);
				\coordinate [label=right:{$[1,e]$}] (1) at ({.5*sqrt(3)},0,-.5);
				\coordinate [label=below:{$[0,\sigma]$}] (01) at (0.9,-0.5, {sqrt(5)});
				
				\coordinate [label=above:{$[0,e]$}] (000) at (-4,{sqrt(2)},0); 
				\coordinate [label=left:{$[1,e]$}] (10) at ({-4-.5*sqrt(3)},0,-.5);
				\coordinate [label=right:{$[2,e]$}] (20) at ({-4+.5*sqrt(3)},0,-.5);
				\coordinate [label=below:{$[0,\sigma]$}] (001) at (-4+0.9,-0.5, {sqrt(5)});
				\coordinate [label=above:{$[0,e]$}] (0000) at (4,{sqrt(2)},0); 
				\coordinate [label=left:{$[1,e]$}] (100) at ({4-.5*sqrt(3)},0,-.5);
				\coordinate [label=right:{$[2,e]$}] (200) at ({4+.5*sqrt(3)},0,-.5);
				\coordinate [label=below:{$[0,\sigma]$}] (0001) at (4+0.9,-0.5, {sqrt(5)});
				
				\coordinate [label=above:{$[0,e]$}] (00000) at (0,{sqrt(2)-5},0); 
				\coordinate [label=left:{$[1,e]$}] (1000) at ({-.5*sqrt(3)},0-5,-.5);
				\coordinate [label=right:{$[2,e]$}] (2000) at ({.5*sqrt(3)},0-5,-.5);
				\coordinate [label=below:{$[0,\sigma]$}] (00001) at (0.9,-0.5-5, {sqrt(5)});
				
				\begin{scope}[decoration= {markings, mark=at position 0.5 with {\arrow{stealth}}}, xshift=0cm]
					\filldraw[fill=cyan, opacity=0.3] (1)--(2)--(3)--cycle; 
					\filldraw[fill=lightgray,fill opacity=0.7] (2)--(3)--(00)--cycle; 
					\filldraw[fill=lightgray, fill opacity=0.2] (1)--(2)--(00)--cycle;
					\filldraw[fill=lightgray, fill opacity=0.2] (1)--(2)--(01)--cycle;
					\filldraw[fill=lightgray, fill opacity=0.7] (2)--(3)--(01)--cycle;
					
					\draw[postaction=decorate] (00)--(1);
					\draw[postaction=decorate] (00)--(2);
					\draw[postaction=decorate] (00)--(3);
					\draw[postaction=decorate] (01)--(1);
					\draw[postaction=decorate] (01)--(2);
					\draw[postaction=decorate] (01)--(3);
					\draw[postaction=decorate] (1)--(2);
					\draw[postaction=decorate] (2)--(3);
					\draw[cyan, postaction=decorate,opacity=0.3] (1)--(3);
				\end{scope}  
				
				\begin{scope}[decoration={markings,
						mark=at position 0.5 with {\arrow{stealth}}}]
					\filldraw[fill=lightgray, opacity=0.3] (000)--(10)--(20)--cycle;
					\filldraw[fill=lightgray, opacity=0.3] (10)--(20)--(001)--cycle;
					\draw[cyan, postaction=decorate] (10)--(20);
					\draw[postaction=decorate] (000)--(10);
					\draw[postaction=decorate] (000)--(20);
					\draw[postaction=decorate] (001)--(10);
					\draw[postaction=decorate] (001)--(20);
				\end{scope}	
				
				\begin{scope}[decoration={markings,
						mark=at position 0.5 with {\arrow{stealth}}}]
					\filldraw[fill=lightgray, opacity=0.3] (0000)--(100)--(200)--cycle;
					\filldraw[fill=lightgray, opacity=0.3] (100)--(200)--(0001)--cycle;
					\draw[cyan, postaction=decorate] (100)--(200);
					\draw[postaction=decorate] (0000)--(100);
					\draw[postaction=decorate] (0000)--(200);
					\draw[postaction=decorate] (0001)--(100);
					\draw[postaction=decorate] (0001)--(200);
				\end{scope}
				
				\begin{scope}[decoration={markings,
						mark=at position 0.5 with {\arrow{stealth}}}]
					\filldraw[fill=lightgray, opacity=0.3] (00000)--(1000)--(2000)--cycle;
					\filldraw[fill=lightgray, opacity=0.3] (1000)--(2000)--(00001)--cycle;
					\draw[cyan, postaction=decorate] (1000)--(2000);
					\draw[postaction=decorate] (00000)--(1000);
					\draw[postaction=decorate] (00000)--(2000);
					\draw[postaction=decorate] (00001)--(1000);
					\draw[postaction=decorate] (00001)--(2000);
				\end{scope}	
				\path[->, shorten >=20pt, shorten <=20pt,bend left=40](00000) edge node[left]{$d_0^2$} (01);
				\path[->, shorten >=16pt, shorten <=16pt,bend left=40](20) edge node[above]{$d_0^1$} (3);
				\path[->, shorten >=16pt, shorten <=16pt,bend left=40](100) edge node[below]{$d_0^3$} (1); 
			\end{tikzpicture} 
		\end{center}
	\end{figure}
	In the \Cref{fig:cofacemaps from 21 to 31} we have geometrical representations of the coface maps $d^i_0: [2]_1 \lto [3]_1$ where $i=0,1,2$. Here the two triangles joined by one edge represent $[2]_1$ and the two tetrahedra joined by a triangle represent $[3]_1$. Notice that this agrees with the terminology we introduced before as $[3]_1$ has a \enquote{triangle} of real points which is being glued to the vertices $[0,e]$ and $[0,\sigma]$. Aditionally, notice that for the real vertices $[i,e]$ with $i=1,2,3$ we could have also chosen $[i,\sigma]$ since $[i,e] = [i,\sigma]$ for $i=1,2,3$.  The coface maps then send the two triangles joined by an edge to the respective \enquote{faces} of $[3]_1$ as in the definition above.  For instance $d^1_0: [2]_1 \lto [3]_1$ sends: 
	\begin{gather*}
		[0,e] \longmapsto [0,e] \\ [0,\sigma] \longmapsto [0,\sigma] \\ [1,e] \longmapsto [2,e] \\ [2,e] \longmapsto [3,e]
	\end{gather*}
	so that $d^1_0$ is sending $[2]_1$ to the left coface of $[3]_1$ in  \Cref{fig:cofacemaps from 21 to 31}. Of course the reader must be warned since there is an abuse in our approach as we are using geometrical models for the $G$-sets $[2]_1$ and $[3]_1$. These will of course correspond to the topological realizations of representable functors $\Delta^{n,k}: G\Deltacat^{\text{op}} \lto \Set$ as in classical simplicial homotopy theory. We hope however that this abuse will help the reader to understand why we have defined coface maps in such a way.

	\noindent Similarly the next figure describes geometrically the two non-real coface maps $d^0_1, d^1_1: [2]_1 \lto [3]_2$. 
	
	\begin{figure}[H]
	\caption{{}} \label{fig:cofacemaps from 21 to 32}	
	\begin{center}
		\begin{tikzpicture}[line join = round, line cap = round,baseline=1ex,scale=1, transform shape, font=\tiny]
			\coordinate[label=above:{$[0,e]$}] (00) at (1,2,0);
			\coordinate[label=below:{$[0,\sigma]$}] (01) at (1,-2,0);
			\coordinate[label=left:{$[2,e]$}] (2) at (0,0,0);
			\coordinate[label=right:{$[3,e]$}] (3) at (2.5,0,0);
			\coordinate[label=above left:{$[1,e]$}] (10) at (1.5,0.5,0);
			\coordinate[label=below left:{$[1,\sigma]$}] (11) at (1.5,-0.5,0);
			
			\coordinate[label=above:{$[0,e]$}] (000) at (-5+1,2,0);
			\coordinate[label=below:{$[0,\sigma]$}] (001) at (-5+1,-2,0);
			\coordinate[label=left:{$[1,e]$}] (02) at (-5+0,0,0);
			\coordinate[label=right:{$[2,e]$}] (03) at (-5+2.5,0,0);
			
			\coordinate[label=above:{$[0,e]$}] (0000) at (5+1,2,0);
			\coordinate[label=below:{$[0,\sigma]$}] (0001) at (5+1,-2,0);
			\coordinate[label=left:{$[1,e]$}] (002) at (5+0,0,0);
			\coordinate[label=right:{$[2,e]$}] (003) at (5+2.5,0,0);

			\begin{scope}[decoration= {markings, mark=at position 0.5 with {\arrow{stealth}}}, xshift=0cm]
				\filldraw[fill=lightgray, opacity=0.5] (2)--(10)--(3)--cycle; 
				\filldraw[fill=lightgray,fill opacity=0.5] (2)--(11)--(3)--cycle; 
				\filldraw[fill=lightgray, fill opacity=0.2] (2)--(00)--(10)--cycle;
				\filldraw[fill=lightgray, fill opacity=0.8] (00)--(10)--(3)--cycle;
				\filldraw[fill=lightgray, fill opacity=0.2] (2)--(11)--(01)--cycle;
				\filldraw[fill=lightgray, fill opacity=0.8] (01)--(11)--(3)--cycle; 
				\draw[cyan, postaction=decorate] (2)--(3);
				\draw[postaction=decorate] (00)--(2);
				\draw[postaction=decorate] (00)--(10);
				\draw[postaction=decorate] (00)--(3);
				\draw[postaction=decorate] (01)--(2);
				\draw[postaction=decorate] (01)--(3);
				\draw[postaction=decorate] (01)--(11);
				\draw[postaction=decorate] (10)--(2);
				\draw[postaction=decorate] (10)--(3);
				\draw[postaction=decorate] (11)--(2);
				\draw[postaction=decorate] (11)--(3);
				
				\filldraw[fill=lightgray, opacity=0.5] (02)--(001)--(03)--cycle;
				\filldraw[fill=lightgray, opacity=0.5] (02)--(000)--(03)--cycle;
				\draw[postaction=decorate] (001)--(02);
				\draw[postaction=decorate] (001)--(03);
				\draw[postaction=decorate] (000)--(02);
				\draw[postaction=decorate] (000)--(03);
				\draw[cyan,postaction=decorate] (02)--(03);

				\filldraw[fill=lightgray, opacity=0.5] (002)--(0001)--(003)--cycle;
				\filldraw[fill=lightgray, opacity=0.5] (002)--(0000)--(003)--cycle;
				\draw[postaction=decorate] (0000)--(002);
				\draw[postaction=decorate] (0000)--(003);
				\draw[postaction=decorate] (0001)--(002);
				\draw[postaction=decorate] (0001)--(003);
				\draw[cyan,postaction=decorate] (002)--(003);
			\end{scope}	
			\path[->, shorten >=20pt, shorten <=20pt,bend left=60](03) edge node[above]{$d^1_1$} (2);
			\path[->, shorten >=20pt, shorten <=20pt,bend left=60](002) edge node[below]{$d^0_1$} (3);
		\end{tikzpicture}	
	\end{center}
	\end{figure}
\end{example}

\begin{para}\label{Par: isovariant deg maps} \normalfont \textbf{Codegeneracy maps.} 
	In a similar way the codegeneracy maps $ \begin{tikzcd} {[n+1]} \arrow[r,"s^i"] &{[n]} \end{tikzcd}$,  $0 \leq i \leq n$ in $\Deltacat$ induce the following isovariant order preserving maps:  For $i\geq k$ we have \emph{real codegeneracy maps} $s_0^i:[n+1]_k \lto [n]_k$ defined by sending
		\begin{equation*}
			[j,g]\longmapsto 
			\begin{cases}
				[j,g] \quad &\text{if} \quad j\leq i \\
				[j-1,g] \quad &\text{if} \quad j > i		
			\end{cases}
		\end{equation*}
		 If $i< k$, same formula above defines \emph{non-real codegeneracy maps} $s_1^i: [n+1]_{k+1}\lto [n]_k$.
	%
\end{para}

\begin{para}\label{paragraph: isovariant swapping maps}\normalfont \textbf{Swapping maps.} The nontrivial element $\sigma$ of the group $G$ induces an isovariant order preserving map also denoted by $\sigma: [n]_k \lto [n]_k$ defined by sending 
	\[ [j,g] \longmapsto [j,\sigma\cdot g], \quad \text{for all } \quad 0\leq j \leq n \text{  and   } g \in G.\] 
Notice that $\sigma$ is an isomorphism in $G\Deltacat$ as $\sigma \circ \sigma = \idd_{[n]_k}$. Recall that a branch of $[n]_k$ consists of a component of the non-real part of $[n]_k$ as in the following figure:
	
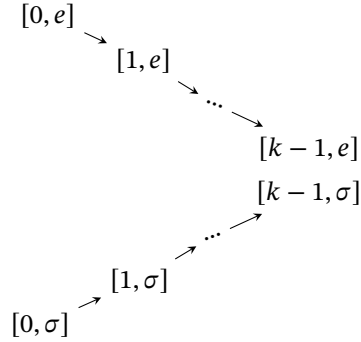
\begin{figure}[H]
		\begin{center}
			\begin{tikzpicture}[align=center, scale=0.7,
				minimum size=4mm,
				>=stealth,
				bend angle=25] 
				\tikzset{node distance = 0.02cm and 0.3cm}
			
				\node (10)  {$[1,e]$};
				\node (00) [above left =of 10] {$[0,e]$} edge[->] (10); 
				\node (20) [below right =of 10, rotate=-20] {$\ldots$};  
				\node (30) [below right = of 20] {$[k-1,e]$}; 
				\node (31) [below  = of 30] {$[k-1,\sigma]$};
				\node (21) [below left =of 31,rotate=20] {$\ldots$};
				\node (11) [below left = of 21] {$[1,\sigma]$};
				\node (01) [below left = of 11] {$[0,\sigma]$};
				
				\path[->] (10) edge (20);
				\path[->] (20) edge (30);
				\path[->] (01) edge (11);
				\path[->] (11) edge (21);
				\path[->] (21) edge (31);
			\end{tikzpicture}
		\end{center}
	\captionsetup{labelformat=empty}
	\end{figure}
\noindent The swapping map $\sigma: [n]_k \lto [n]_k$ can be then thought of as the map swapping the two branches of $[n]_k$ and keeping fixed the real vertices (see also \enumref{Remark:ApproachPkfornk}{3}).   
\end{para}
Just as in the simplex category $\Deltacat$, it is easy to show that morphisms in $G\Deltacat$ are generated by the codegeneracies, coface and swapping maps. Specifically,

\begin{lemma}\label{Lemma: decomposition as coface and codeg maps}
	Every $\theta \in \Hom_{G\Deltacat}([m]_l, [n]_k)$ can be written as $\theta = g \circ  \gamma \circ \eta$ such that $\gamma$ is a finite composition of coface maps, $\eta$ is a finite composition of codegeneracies and $g \in G$. Furthermore, $g, \gamma$ and $\eta$ are uniquely determined by $\theta$.  
\end{lemma}

\begin{para}\label{para: combinatorial description of GDelta}
	\normalfont Using the \Cref{Lemma: decomposition as coface and codeg maps} above one can describe the morphisms in the category $G\Deltacat$ as being generated by the coface, codegeneracy and swapping maps subject to the following \emph{isovariant cosimplicial relations}: Let $n\geq 0$ and $0\leq k \leq n+1$:   
	
	\begin{equation}
		\begin{cases}
			d_1^j \circ d_1^i = d_1^i \circ d_1^{j-1}, \quad i<j<k; \\ 
			d_0^j \circ d_0^i = d_0^i \circ d_0^{j-1}, \quad k<i<j; \\ 
		\end{cases} 
	\end{equation}
	\begin{equation}
		\begin{cases}
			s_1^j \circ s_1^i = s_1^{i-1} \circ s_1^{j}, \quad j<i<k; \\ 
			s_0^j \circ s_0^i = s_0^{i-1} \circ s_0^{j}, \quad k\leq j<i; \\ 
		\end{cases} 
	\end{equation}
	\begin{equation}
		s_{\varepsilon}^j \circ d_{\delta}^i = 
		\begin{cases}
			d_{\delta}^i \circ s_{\varepsilon}^{j-1}, \qquad &i<j \text{ and } \delta=\varepsilon=0 \ \text{or } \delta=\varepsilon=1; \\
			\idd, \qquad &i=j,j+1 \text{ and } \delta=\varepsilon=0 \ \text{or } \delta=\varepsilon=1; \\
			d_{\delta}^{i-1} \circ s_{\varepsilon}^{j}, \qquad &i>j+1 \text{ and } \delta=\varepsilon=0 \ \text{or } \delta=\varepsilon=1.
		\end{cases}
	\end{equation}
	\begin{equation}
		\begin{cases}
			d^i_{\varepsilon} \circ \sigma = \sigma \circ d^i_{\varepsilon}, \quad &\varepsilon=0,1 \\
			s^j_{\delta} \circ \sigma = \sigma \circ s^j_{\delta}, \quad &\delta=0,1 \\ 
			\sigma \circ \sigma = \idd.
		\end{cases}
	\end{equation}
\end{para}

The category $G\Deltacat$ is a skeletal category in the sense of \cite[Chapter 8, \S8.1]{Cisinski2016LPCMDTH}: 
\begin{proposition}\label{Prop: The category Gdelta is einlenber-Zilber}
	The category $G\Deltacat$ is an skeletal category.
\end{proposition}

\begin{proof}
	Let $G\Deltacat_{+}$ denote the class of monomorphisms in $G\Deltacat$ and $G\Deltacat_{-}$ be the class of epimorphisms and define $d: \ob(G\Deltacat) \lto \NN$ by $d([n]_k)= n$. The axioms Sq0 and Sq1 from \cite[Definition 8.1.1]{Cisinski2016LPCMDTH} are verified immediately. To show axiom Sq2, consider a morphism $\theta: [n]_k \lto [m]_l$ in $G\Deltacat$. Using \Cref{Lemma: decomposition as coface and codeg maps} we can factor $\theta$ either as $\theta =\pi \circ \delta$ or $\theta = \sigma \circ \pi \circ \delta$. It is sufficient to consider the former decomposition as any isomorphism is in $G\Deltacat_{+}$ and also in $G\Deltacat_{-}$ and $\sigma$ is an isomorphism. The decomposition $\theta = \pi \circ \delta$ is such that $\pi$ is a composition of coface maps (coface maps are monomorphisms) and therefore is in $G\Deltacat_{+}$ and similarly $\delta$ is a composition of codegeneracy maps thus it belongs to $G\Deltacat_{-}$ (codegeneracies are epimorphisms). Finally if $\alpha: [n]_k \lto [m]_l$ is in $G\Deltacat_{-}$ then $\alpha$ is in particular an epimorphism of $G$-posets and any epimorphism of $G$-posets is surjective \cite[\nopp 7.6]{MH} so that it has a right inverse, \ie, there exist $\beta : [m]_l \lto [n]_k$ in $G\Deltacat$ such that $\alpha \circ \beta = \idd_{[m]_l}$ and thus $\alpha$ admits a section. This proves first part of axiom Sq3. It is clear that any two epimorphisms in $G\Deltacat$ are equal if and only if they have the same set of sections. We conclude that $G\Deltacat$ is an skeletal category.  
\end{proof}

	\section{Isovariant Simplicial Sets}\label{section:isovariant simplicial sets}
\begin{definition}
	An \emph{isovariant simplicial set} $X$ is defined as contravariant functor $X: G\Deltacat^{\text{op}}\lto \Set$. Morphisms between isovariant simplicial sets correspond to natural transformations between these. Isovariant simplicial sets assemble into a category denoted by $\ssetisov$, sometimes we might use the notation $\Fun[G\Deltacat^{\text{op}},\Set]$.
\end{definition}


\begin{example}\label{Example:isovariant_topological_nk_simplex}
	Let $G\Top$ be the category of compactly generated not necessarily weak Hausdorff spaces (see Remark 2.7 \cite{Yeakel2019IET}) with continuous left $G$-actions. Recall  that we refer to its objects as $G$-spaces and morphisms are continuous $G$-equivariant maps of $G$-spaces. Let $G\Top^{\text{isov}}$ be the category of $G$-spaces with isovariant maps. There is a covariant functor $|\Delta^{\bullet,\bullet}|: G\Deltacat \lto G\Top^{\text{isov}}$ defined by the following: 
	\begin{itemize}
		\item[$\bullet$] An object $[n]_k$ is sent to the space $|\Delta^{n,k}|$ defined by the pushout: 
		
		\[
		\begin{tikzcd}
			|\Delta^{\{k,k+1,...,n\}}|  \arrow[r,"\iota^n_k"] \arrow[d,"\iota^n_k",swap] &|\Delta^n| \arrow[d] \\ 
			|\Delta^n| \arrow[r] & |\Delta^{n,k}| \arrow[lu, phantom, "\ulcorner", very near start]
		\end{tikzcd}
		\]
		where $|\Delta^n|$ denotes the usual topological $n$-simplex. 
		Similarly $|\Delta^{\{k,k+1,...,n\}}|$ denotes the topological $n-k+1$-simplex: 
		\[ |\Delta^{\{k,k+1,...,n\}}|=\setbar{\paren{t_0,...,t_{n-k}}\in [0,1]^{n-k+1}}{\sum_{i=0}^{n-k}t_i=1}\]
		so that $\iota^n_k: |\Delta^{\{k,k+1,...,n\}}| \lto |\Delta^n|$ is the continuous map sending 
		\[(t_0,...,t_{n-k}) \longmapsto (\underbrace{0,...,0}_{k \text{ times}},t_0,t_1,...,t_{n-k})\] 
		Notice that $|\Delta^{n,k}|$ comes naturally with an action by $G=C_2$ where the nontrivial element acts by swapping elements which are not identified in the pushout and trivially on elements that are identified. This ties up with the geometrical intuition we had for $[n]_k$. For instance $|\Delta^{3,1}|$ is represented as in Example \ref{Ex:cofacemaps from 21 to 31} by the following figure: 
		
		\begin{figure}[H]
			\caption{Topological standard $(2,1)$ simplex} \label{fig:TopologicalDelta31}
			\begin{center}
				\begin{tikzpicture}[line join = round, line cap = round,baseline=1ex,scale=1.5, transform shape, font=\tiny]
					\coordinate [label=above:{$[0,e]$}] (00) at (0,{sqrt(2)},0);
					\coordinate [label=left:{$[3,e]$}] (3) at ({-.5*sqrt(3)},0,-.5);
					\coordinate [label=below right:{$[2,e]$}] (2) at (0,0,0.6);
					\coordinate [label=right:{$[1,e]$}] (1) at ({.5*sqrt(3)},0,-.5);
					\coordinate [label=below:{$[0,\sigma]$}] (01) at (0.9,-0.5, {sqrt(5)});
					
					\begin{scope}[decoration= {markings, mark=at position 0.5 with {\arrow{stealth}}}, xshift=0cm]
						\filldraw[fill=cyan, opacity=0.3] (1)--(2)--(3)--cycle; 
						\filldraw[fill=lightgray,fill opacity=0.7] (2)--(3)--(00)--cycle; 
						\filldraw[fill=lightgray, fill opacity=0.2] (1)--(2)--(00)--cycle;
						\filldraw[fill=lightgray, fill opacity=0.2] (1)--(2)--(01)--cycle;
						\filldraw[fill=lightgray, fill opacity=0.7] (2)--(3)--(01)--cycle;
						
						\draw[postaction=decorate] (00)--(1);
						\draw[postaction=decorate] (00)--(2);
						\draw[postaction=decorate] (00)--(3);
						\draw[postaction=decorate] (01)--(1);
						\draw[postaction=decorate] (01)--(2);
						\draw[postaction=decorate] (01)--(3);
						\draw[postaction=decorate] (1)--(2);
						\draw[postaction=decorate] (2)--(3);
						\draw[cyan, postaction=decorate,opacity=0.3] (1)--(3);
					\end{scope} 
				\end{tikzpicture}
			\end{center}
		\end{figure}
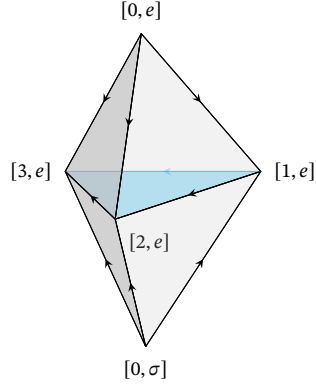
		\item[$\bullet$] A morphism $\theta: [n]_k \lto [m]_l$ in $G\Deltacat$ induces continuous maps 
		
		\begin{align*}
			\theta_{\ast}: |\Delta^n|& \lto |\Delta^m| \\ (t_0,...,t_n) &\longmapsto \left( \sum_{[j,e]\in \theta^{-1}[0,e]}t_j  \ ,...,\sum_{[j,e]\in \theta^{-1}[l-1,e]}t_j \ , \sum_{[j,e]\in \theta^{-1}[l,e]}t_j,...,\sum_{[j,e]\in \theta^{-1}[m,e]}t_j \right)
		\end{align*}
%
		and 
		\begin{align*}
			\theta_{\ast}^{k,l}: |\Delta^{\{k,k+1,...,n\}}|& \lto |\Delta^{\{l,l+1,...,m\}}| \\ (t_0,...,t_{n-k}) &\longmapsto \left( \sum_{[j,e]\in \theta^{-1}[l,e]}t_{j-k} \ , ...,\sum_{[j,e]\in \theta^{-1}[m,e]}t_{j-k} \right)
		\end{align*} 

so that,
	\begin{equation*}
		\theta_{\ast} \circ \iota^n_k = \iota^m_l \circ \theta^{k,l}_{\ast} 
	\end{equation*}
	These maps then assemble into a diagram, where the top and left faces commute in addition to the front and back faces: 
	
	\[
	\begin{tikzcd}[row sep=1.5em, column sep = 1.5em]
		|\Delta^{\{k,...,n\}}| \arrow[rr,"\iota^n_k"] \arrow[dr, swap,"\theta^{k,l}_{\ast}"] \arrow[dd,"\iota^n_k",swap] &&
		|\Delta^n| \arrow[dd] \arrow[dr,"\theta_{\ast}"] \\
		& |\Delta^{\{l,...,m\}}| \arrow[rr,"\iota^m_l",near start,crossing over] &&
		|\Delta^m| \arrow[dd] \\
		|\Delta^n| \arrow[rr] \arrow[dr, "\theta_{\ast}"] && |\Delta^{n,k}| \arrow[lu, phantom, "\ulcorner", very near start] \arrow[dr,"\theta_{\ast}",dashed] \\
		& |\Delta^m| \arrow[rr]&& |\Delta^{m,l}| 
		\arrow[from=2-2, to=4-2, "\iota^m_l",near start, crossing over]
	\end{tikzcd}
	\]
	By the universal property of pushouts there is a unique canonical map $|\Delta^{n,k}|\lto |\Delta^{m,l}|$ which we will also denote by $\theta_{\ast}$. Notice that $\theta_{\ast}$ is isovariant by construction.
	\end{itemize}	
\end{example}

\begin{remark}\label{Ex:standard nk isovariant simplicial set}
	Let $[n]_k$ be an object in  $G\Deltacat$. The standard $(n,k)$-isovariant simplicial set is defined as:
	\begin{center}
		$\Delta^{n,k} : \begin{array}{ll}
			G\Deltacat^{\text{op}} &\longrightarrow \Set \\
			\left[m\right]_l &\longmapsto \Delta^{n,k}[m]_l:= \Hom_{G\Deltacat}([m]_l,[n]_k)
		\end{array}$
	\end{center}
	That is, $\Delta^{n,k}$ is the contravariant functor on $G\Deltacat$ represented by the object $[n]_k$ which is clearly well defined. If $X : G\Deltacat^{\text{op}} \longrightarrow \Set$ is any other isovariant simplicial set and if we denote $X([n]_k)$ by $X_{n,k}$. As is it usual, Yoneda's lemma provide a natural isomorphism: 
%
	\[ \Hom_{\ssetisov}(\Delta^{n,k} , X) \cong X_{n,k}. \]
\end{remark}

\begin{example}[Isovariant singular simplicial set]\label{ex:isovariantsingularsimplicialset}
	Let $X$ in $G\Top^{\text{isov}}$. The \emph{isovariant singular simplicial set} is the isovariant simplicial set defined by: 
	\begin{center}
		$ \text{Sing}^{\text{isov}}_X:  \begin{array}{ll}
			G\Deltacat^{\text{op}} &\longrightarrow \Set \\ \left[ n \right]_k &\longmapsto \Hom_{G\Top^{\text{isov}}}(|\Delta^{n,k}|,X)
		\end{array}$
	\end{center}
	Notice that if $\theta : [n]_k \longrightarrow [m]_l$ is any morphism in $G\Deltacat$ then 
	\[
	\text{Sing}_X^{\text{isov}}(\theta): \Hom_{G\Top^{\text{isov}}}(|\Delta^{m,l}|,X) \longrightarrow \Hom_{G\Top^{\text{isov}}}(|\Delta^{n,k}|,X) \]
	is defined by sending $f \mapsto f \circ \theta_{\ast}$, where $\theta_{\ast}: |\Delta^{n,k}| \lto |\Delta^{m,l}|$ is the induced isovariant map from \Cref{Example:isovariant_topological_nk_simplex}.
	It is easy to prove that $\Sing_X^{\text{isov}}$ is indeed a functor. Furthermore, $\Sing^{\text{isov}}$ can be considered as a functor from the category $G\Top^{\text{isov}}$ to $\ssetisov$. 
	\[
	\Sing^{\text{isov}}:  \begin{array}{ll}
		G\Top^{\text{isov}} &\longrightarrow \ssetisov \\ X &\longmapsto \Sing^{\text{isov}}_X
	\end{array}
	\]
	where for each continuous isovariant map $f: X \lto Y$, $\Sing^{\text{isov}}(f)$ is the isovariant simplicial set map defined by
	\[ \text{Sing}^{\text{isov}}(f)_{[n]_k} (\alpha) = f \circ \alpha \]
	for each $[n]_k \in G\Deltacat$ and $\alpha \in \Hom_{{G\Top}^{\text{isov}}}(|\Delta^{n,k}|, X)$. 
\end{example}
Let $X: G\Deltacat^{\text{op}} \longrightarrow \Set$ be an isovariant simplicial set. For a morphism $\theta$ in $G\Deltacat$ we will usually write $\theta^{\ast}$ for the image $X(\theta)$ under $X$. For the particular morphisms we have defined in \ref{Para: isovariant coface maps}, \ref{Par: isovariant deg maps} and \ref{paragraph: isovariant swapping maps} we simply write $d^i_{\alpha_i}$ and $s^i_{\beta_i}$ (where $\alpha_{i}, \beta_i \in \{0,1\}$) and $\sigma$ for $X(d^i_{\alpha_i})$ and $X(s^i_{\beta_i})$ and $X(\sigma)$ respectively, these are called \emph{face}, \emph{degeneracy} and \emph{swapping} maps respectively. Thanks to the functoriality of $X$, these maps satisfy dual identities to the isovariant cosimplicial identities from \ref{para: combinatorial description of GDelta} called \emph{isovariant simplicial identities}.

	%
		As a consequence we have that any isovariant simplicial set $X: G\Deltacat^{\text{op}} \longrightarrow \Set$ is completely determined by a sequence of $G$-sets $X_{n,k}$ and isovariant maps $d^i_{\alpha_i}$, $s^i_{\beta_i}$ and $\sigma$: 
			
		\[
		\begin{tikzcd}[column sep=2cm, row sep=2cm,every label/.append style = {font = \tiny}]
			X_{0,0} \arrow[r,bend left=70,"s_0^0",near end] 
			& X_{1,0} 
			\arrow[l,shift left=.75ex,"d^1_0"] 	
			\arrow[l,shift right=.75ex,swap,"d_0^0"] 
			\arrow[r,bend left=70,"s_0^0",near end] 
			\arrow[r,bend right=70,"s^1_0",swap, near end] 
			&X_{2,0}  
			\arrow[l,shift right=0.5ex, scale=4,"d_0^0",swap] 
			\arrow[l, shift right=-0.3ex,"d^1_0",swap] 
			\arrow[l,shift right=-0.7ex, scale=4,"d^2_0",swap]  
			&X_{3,0}  
			\arrow[l,shift right=0.9ex, scale=4,"d_0^0",swap] 
			\arrow[l, shift right=0.3ex,scale=4,"d^1_0",swap] 
			\arrow[l,shift right=-0.3ex, scale=4,"d^2_0",swap] 
			\arrow[l,shift right=-0.9ex, scale=4,"d^3_0",swap] \ \ \cdots 
			\\	X_{0,1}  
			&X_{1,1} 
			\arrow[ul,"d^0_1", shorten >=20pt] 
			\arrow[l,"d^1_0",swap]
			&X_{2,1} 	
			\arrow[l,shift left=.75ex,"d^1_0"] 	
			\arrow[l,shift right=.75ex,swap,"d^2_0"] 
			\arrow[ul,"d^0_1",shorten >=20pt,crossing over]
			& X_{3,1} \ \ \cdots 
			\arrow[l,shift right=0.5ex, scale=4,"d^1_0",swap] 
			\arrow[l, shift right=-0.34ex,"d^2_0",swap] 
			\arrow[l,shift right=-0.7ex, scale=4,"d^3_0",swap]
			\arrow[lu,"d^0_1",shorten >=20pt, crossing over]
			\\ & X_{1,2} 
			\arrow[ul,shift left=.75ex,"d_1^1",shorten >=20pt] 	
			\arrow[ul,shift right=.75ex,swap,"d^0_1",shorten >=20pt]
			&\underset{\displaystyle \vdots}{X_{2,2}} 
			\arrow[ul,shift left=.75ex,"d_1^1",shorten >=20pt] 	
			\arrow[ul,shift right=.75ex,swap,"d^0_1",shorten >=20pt]
			\arrow[l,"d^2_0"]
			&\underset{\displaystyle \vdots}{X_{3,2}} \ \ \cdots
			\arrow[ul,shift left=.75ex,"d_1^1",shorten >=20pt] 	
			\arrow[ul,shift right=.75ex,swap,"d^0_1",shorten >=20pt]
			\arrow[l,shift left=.75ex,"d^2_0"] 	
			\arrow[l,shift right=.75ex,swap,"d^3_0"]
		\end{tikzcd}
		\]
		which satisfy the isovariant simplicial identities. Notice that in the last diagram we have decided to only draw some of the degeneracy maps.

		\begin{para}\label{para:isovariantboundaryandhorn} \normalfont \textbf{Simplicial isovariant spheres.} Just as in \Cref{Example:isovariant_topological_nk_simplex}, we also have a functor 
			\[ \Delta^{\bullet, \bullet} : G\Deltacat \lto \ssetisov\]
			defined objects by sending $[n]_k \longmapsto \Delta^{n,k}$. On morphisms it comes defined by postcomposition. Moreover, each coface map $d^i_{\alpha_i}: [n]_k\lto [n+1]_{k+\alpha_i}$ induces an isovariant simplicial set map $\Delta^{n,k} \lto \Delta^{n+1, k+\alpha_i}$ denoted again by $d^i_{\alpha_{i}}$. If $n>0$ we define: 
			
			\begin{enumerate}[label=\stlabel{para:isovariantboundaryandhorn}, ref=\arabic*]
				\item \label{para:isovariantboundaryandhorn.1}The $1-(n,k)$-sphere, denoted by $\partialbf^1 \Delta^{n,k}$, as the union of the images of the $i$-th non-real face maps $d^1_i:\Delta^{n-1,k-1} \lto \Delta^{n,k}$ for all $0 \leq i \leq k-1$, \ie, if we denote by $\partialbf_1^i\Delta^{n,k} = \im (d_1^i: \Delta^{n-1,k-1} \lto \Delta^{n,k})$, then 
				
				\[\partialbf_1 \Delta^{n,k} := \bigcup_{0\leq i \leq k-1} {\partialbf}_1^i \Delta^{n,k} \subseteq \Delta^{n,k}.\]
				\item \label{para:isovariantboundaryandhorn.2} Similarly, the $0-(n,k)-$sphere, denoted by $\partialbf_0\Delta^{n,k}$ is defined as 
				\[\partialbf_0 \Delta^{n,k} := \bigcup_{k\leq i \leq n}\partialbf^i_0 \Delta^{n,k} \subseteq \Delta^{n,k}\]
				where $\partialbf^i_0 \Delta^{n,k}$ denotes again the image of the real isovariant simplicial set face map $d^i_0: \Delta^{n-1,k}\lto \Delta^{n,k}$. Moreover both \enquote{unions} (colimits) are calculated pointwise as is any colimit in a category of presheaves. 
			\end{enumerate} 
			\emph{The simplicial isovariant $(n,k)$-sphere}, denoted by $\partialbf \Delta^{n,k}$, is the sub-object of $\Delta^{n,k}$ given by 
			\[ \partialbf \Delta^{n,k} = \partialbf^1\Delta^{n,k} \bigcup \partialbf^0\Delta^{n,k}\]
			Note that $\partialbf \Delta^{0,0} = \partialbf \Delta^{0,1} =\emptyset$, where $\emptyset$ is the unique isovariant simplicial set which consist of the empty set at each $[m]_l \in G\Deltacat$. Sometimes we will also refer to $\partialbf \Delta^{n,k}$ as the \emph{simplicial isovariant boundary} of $\Delta^{n,k}$. If $x$ denotes a $(m,l)$-non-degenerate simplex of $\partialbf\Delta^{n,k}$ then one must have that $m<n-1$. In particular, \cite[Proposition 8.1.22]{Cisinski2016LPCMDTH} implies that for any $n\geq 1$ and $0 \leq k \leq n+1$ we have
		\[\sk_{n-1}(\Delta^{n,k}) = \partialbf\Delta^{n,k}.\]
	\end{para}
		
		
		\begin{remark}\label{Remark:Notationdeltank_in_terms_of_its_nondeg_simplices}
			The standard $(n,k)$-isovariant simplicial set $\Delta^{n,k}$ has $n-k+1$ non-degenerate  $(n-1,k)$-simplices and $k$ non-degenerate $(n-1,k-1)$-simplices. These are determined precisely by the images of the corresponding non-real and real face maps: 
			\begin{equation}
				\begin{tikzcd}\label{Equation:non degenerate simplices of delta.1}
					\Delta^{n-1,k} \arrow[r,"d_0^i"] &\Delta^{n,k}, \quad i=k,k+1,...,n
				\end{tikzcd} 
			\end{equation}
			\begin{equation}\label{Equation:non degenerate simplices of delta.2}
				\begin{tikzcd}
					\Delta^{n-1,k-1} \arrow[r,"d_1^i"] &\Delta^{n,k}, \quad i=0,1,...,k-1
				\end{tikzcd} 
			\end{equation}
			By Yoneda's Lemma, the non-real and real face maps from \ref{Equation:non degenerate simplices of delta.1} and \ref{Equation:non degenerate simplices of delta.2} are completely determined by its images on $(0,0)$ and $(0,1)$ vertices thus we will sometimes write
			\[\Delta^{n,k} = \left\langle v_0^{\nreal}, v_1^{\nreal},...,v_{k-1}^{\nreal}, v_k^{\real}, ..., v_n^{\real} \right\rangle \] where each $v_i^{\nreal}$ represents a $(0,1)$-vertex  and $v_j^{\real}$ a $(0,0)$-vertex. Hence, with this notation the image of the non-real face map $d_0^i: \Delta^{n-1.k} \lto \Delta^{n,k}$ will be written as
			\[\partialbf_0^i(\Delta^{n,k}) = \left\langle v_0^{\nreal}, v_1^{\nreal},...,v_{k-1}^{\nreal},v_k^{\real},...,\widehat{v_i^{\real}},...,v_n^{\real}  \right\rangle, \quad k\leq i \leq n \ \]
			and similarly 
			\[\partialbf_1^i(\Delta^{n,k}) = \left\langle v_0^{\nreal}, v_1^{\nreal},...,v_{i-1}^{\nreal},\widehat{v_i^{\nreal}},...,v_k^{\real},...,v_n^{\real}  \right\rangle, \quad 0\leq i \leq k-1 \]
			As a consequence  
			\[ \partialbf \Delta^{n,k} = \left( \bigcup_{i=0}^{k-1} \left\langle v_0^{\nreal}, v_1^{\nreal},...,v_{i-1}^{\nreal},\widehat{v_i^{\nreal}},...,v_k^{\real},...,v_n^{\real}  \right\rangle \right) \bigcup \left(\bigcup_{i=k}^n \left\langle v_0^{\nreal}, v_1^{\nreal},...,v_{k-1}^{\nreal},v_k^{\real},...,\widehat{v_i^{\real}},...,v_n^{\real}  \right\rangle \right). \]
		\end{remark}

		\begin{para}\label{para: isovariant horn}  	\normalfont \textbf{Simplicial isovariant horns.}
			Let $n>0$, $0\leq k\leq n+1$ and $0\leq l \leq n$. The simplicial isovariant $l$-horn, denoted by $\Lambda^{n,k}_l$ is the sub-object of $\Delta^{n,k}$ given by: 
			
			\[ \Lambda^{n,k}_l \coloneqq \left(\bigcup_{\substack{0 \leq i \leq k-1 \\ i\neq l}}\partialbf_i^1\Delta^{n,k}\right) \bigcup \left(\bigcup_{\substack{k \leq i \leq n \\ i\neq l}}\partialbf_i^0\Delta^{n,k}\right).\]
		\end{para}
		\begin{remark}\label{Remark:simplices of isovariant horns}
			The $(a,b)$-simplices of the isovariant $l$-th horn $\Lambda^{n,k}_l$ can be described as: 
			\[ \left(\Lambda^{n,k}_l\right)_{a,b} = \setbar{\alpha: [a]_b \lto [n]_k}{[n]_k \smallsetminus \left\{[l,e], [l,\sigma] \right\} \nsubseteq \alpha[a]_b}.\]	
		Notice that $[l,e]$ could be the equal to $[l,\sigma]$ in which case the condition in the last equality turns into $[n]_k \smallsetminus \left\{[l,e]\right\}\nsubseteq \alpha[a]_b$. Furthermore, in terms of the notation introduced in \Cref{Remark:Notationdeltank_in_terms_of_its_nondeg_simplices}, we will also write: 
		
		\begin{equation}\label{Equation:notationforhorns}
			\begin{gathered}
				\Lambda^{n,k}_l = \partialbf\Delta^{n,k} \smallsetminus \left\langle v_0^{\nreal},v_1^{\nreal},..., \widehat{v_l^{\nreal}},...,v_{k-1}^{\nreal}, v_k^{\real},...,v_n^{\real}  \right\rangle,  \qquad  0\leq l \leq k-1 \\  \\
				\Lambda^{n,k}_l = \partialbf\Delta^{n,k} \smallsetminus \left\langle v_0^{\nreal},v_1^{\nreal},...,v_{k-1}^{\nreal},v_k^{\real},...,\widehat{v_l^{\real}},...,v_n^{\real}  \right\rangle,  \qquad k\leq l \leq n 
			\end{gathered}
		\end{equation}
		to indicate that the $l$-th horn $\Lambda^{n,k}_l$ correspond to the maximal isovariant subsimplicial set of $\partialbf \Delta^{n,k}$ containing all non-degenerate simplices of $\partialbf \Delta^{n,k}$ except for \[\left\langle [0,e],[0,\sigma], [1,e],[1,\sigma], ...,\widehat{[l,e]},\widehat{[l,\sigma]},..., [n,e]  \right\rangle \in (\partialbf \Delta^{n,k})_{n-1,k-\varepsilon}\] where $\varepsilon=0$ or $\varepsilon=1$. 
		\end{remark}

		\begin{definition}\label{def: Admisible simplicial isovariant horn}
			Let $n>0$, $0\leq k\leq n+1$ and $0\leq l \leq n$. The $l$-horn $\Lambda^{n,k}_l$ is said to be \emph{admissible} if there exists $a \in \{l-1,l\}$ such that
			\[ d_{\varepsilon}^l\left(s_{\varepsilon}^a(\Delta^{n,k}) \right)  = \partialbf_{\varepsilon}^l(\Delta^{n,k}), \quad \text{ where  } \varepsilon=0 \text{ or } 1.\]
			In any other case we say that $\Lambda^{n,k}_l$ is \emph{non admissible}. Notice that for $k=0$ or $k=n+1$ all horns are admissible and, the face $\partialbf_{\varepsilon}^l(\Delta^{n,k})$ from last definition is precisely the face (either real or non-real) removed from $\partialbf \Delta^{n,k}$ in the definition of $\Lambda_l^{n,k}$ (see \Cref{Equation:notationforhorns}). From now on $n-k$ will be called the \emph{dimension of the real part} of $\Delta^{n,k}$ or simply the \emph{real dimension} of $\Delta^{n,k}$. Similarly, $k$ will be called the \emph{non-real dimension} of $\Delta^{n,k}$. 
		\end{definition}
		
		This last definition may seem irrelevant as it is almost always satisfied, for instance in the case where $n=4$ and $k=3$ all horns are admissible since the coface and codegeneracy maps: 
		\[ \begin{tikzcd}[every label/.append
			style={font=\tiny}]
			{[3]}_3 \arrow[r,shift left=0.75ex,"d_0^3"] \arrow[r,shift right=0.75ex,"d_0^4",swap] &{[4]}_3 \end{tikzcd} \qquad \begin{tikzcd}[every label/.append
			style={font=\tiny}]  {[3]}_2 \arrow[r,shift left=2.3ex,"d_1^0"] \arrow[r,shift right=2.3ex,"d_1^2"] \arrow[r,"d_1^1"] &{[4]}_3
		\end{tikzcd}\]
		
		\[ \begin{tikzcd}[every label/.append
			style={font=\tiny}]
			{[4]}_3 \arrow[r,"s_0^3"]  &{[3]}_3 \end{tikzcd} \qquad \begin{tikzcd}[every label/.append
			style={font=\tiny}]  {[4]}_3 \arrow[r,shift left=0.8ex,"s_1^0"] \arrow[r,shift right=0.8ex,"s_1^1",swap] &{[3]}_2
		\end{tikzcd}\]
		and their corresponding induced isovariant simplicial set maps (via Yoneda's lemma) justify the equalities: 
		
		\begin{gather*}
			d_1^0(s_1^0(\Delta^{4,3})) = \partialbf_1^0(\Delta^{4,3}) \\ 
			d_1^1(s_1^0(\Delta^{4,3})) = \partialbf_1^1(\Delta^{4,3}) \\
			d_1^2(s_1^1(\Delta^{4,3})) = \partialbf_1^2(\Delta^{4,3}) \\
			d_0^3(s_0^3(\Delta^{4,3})) = \partialbf_0^3(\Delta^{4,3}) \\
			d_0^4(s_0^3(\Delta^{4,3})) = \partialbf_0^4(\Delta^{4,3}) \\
		\end{gather*}
		so that $\Lambda^{4,3}_0$, $\Lambda^{4,3}_1, \Lambda^{4,3}_2, \Lambda^{4,3}_3$ and $\Lambda^{4,3}_4$ are all admissible. Nevertheless this is not always true, one can consider for instance the $0$-th horn $\Lambda^{2,1}_0 \subseteq \Delta^{2,1}$, in this case there is no codegeneracy of the form $s_1^0 : [2]_1 \lto [1]_0$ such that 
		
		\[ d_0^1(s_0^1(\Delta^{2,1})) = \partialbf_0^1(\Delta^{2,1}).\]
		This last assertion is trivially true since there is no order-preserving isovariant map from $[2]_1$ to $[1]_0$. The $1$-horn $\Lambda^{2,1}_1 \subseteq \Delta^{2,1}$ is however admissible. \Cref{def: Admisible simplicial isovariant horn} guarantees the existance of some proper face of $\Delta^{2,1}$ also contained in $\Lambda^{2,1}_1$ such that $\Delta^{2,1}$ and, $\Lambda^{2,1}_1$ can be both \emph{retracted isovariantly} to such face. In the following picture we have from left to right the isovariant simplicial set $\Delta^{2,1}$, the mentioned proper face and the horn $\Lambda^{2,1}_1$. This is analogous to what happens in stratified homotopy, see \cite[Definition 1.10]{douteau2020simplicial}. 
		
		\begin{center}
			\begin{tikzpicture}[line join = round, line cap = round,baseline=1ex,scale=1.1, transform shape, font=\tiny]
				\coordinate[label=above:{$[0,e]$}] (00) at (1,2,0);
				\coordinate[label=below:{$[0,\sigma]$}] (01) at (1,-2,0);
				\coordinate[label=left:{$[2,e]$}] (2) at (0,0,0);
				\coordinate (3) at (2.5,0,0);
				
				\coordinate[label=above:{$[0,e]$}] (000) at (-5+1,2,0);
				\coordinate[label=below:{$[0,\sigma]$}] (001) at (-5+1,-2,0);
				\coordinate[label=left:{$[1,e]$}] (02) at (-5+0,0,0);
				\coordinate[label=right:{$[2,e]$}] (03) at (-5+2.5,0,0);
				
				\coordinate[label=above:{$[0,e]$}] (0000) at (5+1,2,0);
				\coordinate[label=below:{$[0,\sigma]$}] (0001) at (5+1,-2,0);
				\coordinate[label=left:{$[1,e]$}] (002) at (5+0,0,0);
				\coordinate[label=right:{$[2,e]$}] (003) at (5+2.5,0,0);

				\begin{scope}[decoration= {markings, mark=at position 0.5 with {\arrow{stealth}}}, xshift=0cm]
					\draw[postaction=decorate] (00)--(2);
					\draw[postaction=decorate] (01)--(2);
					\fill[cyan] (2) circle[radius=2pt];
					\fill (00) circle[radius=2pt];
					\fill (01) circle[radius=2pt];
					
					\filldraw[fill=lightgray, opacity=0.5] (02)--(001)--(03)--cycle;
					\filldraw[fill=lightgray, opacity=0.5] (02)--(000)--(03)--cycle;
					\draw[postaction=decorate] (001)--(02);
					\draw[postaction=decorate] (001)--(03);
					\draw[postaction=decorate] (000)--(02);
					\draw[postaction=decorate] (000)--(03);
					\draw[cyan,postaction=decorate] (02)--(03);
					\fill[cyan] (02) circle[radius=2pt];
					\fill[cyan] (03) circle[radius=2pt];
					\fill (001) circle[radius=2pt];
					\fill (000) circle[radius=2pt];
					
					\draw[postaction=decorate] (0000)--(002);
					\draw[postaction=decorate] (0001)--(002);
					\draw[cyan,postaction=decorate] (002)--(003);
					\fill (0000) circle[radius=2pt];
					\fill (0001) circle[radius=2pt];
					\fill[cyan] (002) circle[radius=2pt];
					\fill[cyan] (003) circle[radius=2pt];
				\end{scope}	
				\path[->, shorten >=20pt, shorten <=20pt,bend right=60](2) edge (03);
				\path[->, shorten >=20pt, shorten <=20pt,bend left=60](3) edge (002);
			\end{tikzpicture}	
		\end{center}			
		
		\Cref{def: Admisible simplicial isovariant horn} can be restated in the following equivalent way: 
		
		\begin{proposition}\label{Proposition:admissible horns}
			Let $n>0$. The $l$-th horn $\Lambda^{n,k}_l$ is non-admissible if and only if $k=1$ and $l=0$ or $k=l=n$. That is, the only non-admissible horns are those of the form $\Lambda^{n,1}_0$ and $\Lambda^{n,n}_n$. 
		\end{proposition}	
		
		\begin{proof}
			Consider first the case where $k\geq 2$ and $k \neq n$. Then there are well defined coface and codegeneracy maps: 
			\[
			\begin{tikzcd}
				{[n]}_k \arrow[r,"s_1^{l}"] & {[n-1]}_{k-1} \arrow[r,"d_1^l"] & {[n]}_k
			\end{tikzcd} \quad 0\leq l \leq k-1
			\]
			\[ 
			\begin{tikzcd}
				{[n]}_k \arrow[r,"s_0^l"] & {[n-1]}_k \arrow[r,"d_0^l"] &{[n]}_k
			\end{tikzcd} \quad k \leq l \leq n
			\]	
			An easy calculation using \ref{Para: isovariant coface maps} and \ref{Par: isovariant deg maps} implies that $d_1^l(s_1^l(\Delta^{n,k})) = \partialbf_1^l(\Delta^{n,k})$ and $d_0^l(s_0^l(\Delta^{n,k})) = \partialbf_0^l(\Delta^{n,k})$ so that all horns $\Lambda^{n,k}_l$ with $k\geq 2$, $k \neq n$ are admissible. Hence if we suppose that $\Lambda^{n,k}_l$ is not admissible and $k \neq n$ necessarily we have that $k\leq 1$. Since for $k=0$ all horns are also admissible then we just have to study the case $k=1$. In the case where $k \neq n$ we are then led to study the admissible horns of $\Delta^{n,1}$. Suppose first $0<l\leq n$. Then $l\geq k=1$ and we have well defined codegeneracies and coface maps: 
			\begin{gather*}
				s_0^l: [n]_1 \lto [n-1]_1 \\ 
				d_0^l:[n-1]_1 \lto [n]_1
			\end{gather*}
			Their corresponding isovariant simplicial set maps $s_0^l: \Delta^{n,1}\lto \Delta^{n-1,1}$ and $d_0^l:\Delta^{n-1,1} \lto \Delta^{n,1}$ satisfy: 
			\begin{align*}
				d_0^l(s_0^l(\Delta^{n,1})) &= d_0^l \left( s_0^l \left( \left\langle v_0^{\nreal},v_1^{\real},...,v_l^{\real},...,v_{n}^{\real} \right\rangle \right)\right)  \\
				&= d_0^l\left(\left\langle v_0^{\nreal},v_1^{\real},...,v_l^{\real},v_l^{\real},v_{l+1}^{\real}...,v_{n-1}^{\real} \right\rangle\right) \\
				&= \left\langle v_0^{\nreal},v_1^{\real},...,\widehat{v_l^{\real}},v_{l+1}^{\real},...,v_{n}^{\real} \right\rangle \\
				&= \partialbf_0^l(\Delta^{n,1}).
			\end{align*}
			Hence, for $0<l\leq n$ we have that all $l$-horns of $\Delta^{n,1}$ are admissible. The $0$-horn $\Lambda^{n,1}_0$ is not admissible since there are no codegeneracies of the form
			
			\[
			\begin{tikzcd}
				{[n]}_1 \arrow[r,"s_1^{0}"]  & {[n-1]}_0
			\end{tikzcd} \quad \text{or} \quad 
			\begin{tikzcd}
				{[n]}_1 \arrow[r,"s_1^{-1}"]  & {[n-1]}_0
			\end{tikzcd} 
			\]
			The first map being equivariant but not isovariant and the second having a negative superscript which is not allowed. As a consequence if $\Lambda^{n,k}_l$ is non-admissible and $k\neq n$ then we are forced to have $k=1$ and $l=0$. On the other hand if $k=n \geq 2$ we are led to study the admissible horns of $\Delta^{n,n}$. In this case we have well defined coface and codegeneracy maps: 
			
			\[
			\begin{tikzcd}
				{[n]}_n \arrow[r,"s_1^{l}"] & {[n-1]}_{n-1} 
			\end{tikzcd} \quad 0\leq l \leq n-2
			\]
			\[
			\begin{tikzcd}
				{[n-1]}_{n-1} \arrow[r,"d_1^{l}"] & {[n]}_{n} 
			\end{tikzcd} \quad 0\leq l \leq n-1
			\]
			
			\[ 
			\begin{tikzcd}
				{[n-1]}_n \arrow[r,"d_0^n"] & {[n]}_n 
			\end{tikzcd}
			\]	
			These maps and their corresponding isovariant simplicial set maps in $\ssetisov$ are such that 
			\[ d_1^l(s_1^l(\Delta^{n,n})) = \partialbf_1^l(\Delta^{n,n})\] 
			for any $0\leq l \leq n-2$ so that all horns of the form $\Lambda^{n,n}_{l}$ with $n\geq 2$ and $0\leq l\leq n-2$ are admissible. Hence, if we suppose that $\Lambda^{n,n}_l$ is non-admissible necessarily we have that $l=n-1$ or $l=n$. However the horn $\Lambda^{n,n}_{n-1}$ is admissible since 
			
			\[ d_1^{n-1}(s_1^{n-2}(\Delta^{n,n})) = \partialbf_1^{n-1}(\Delta^{n,n}).\]	
			As a consequence we must have that $l=n$. Notice that $\Lambda^{n,n}_n$ is indeed non admissible as there are no codegeneracies of the form $s_0^{n-1}, s_0^n:[n]_n \lto [n-1]_n$ (as $[n]_n$ contains a real point whereas $[n-1]_n$ consist of only non-real points). \\
			Finally if $k=n=1$, notice that $\Delta^{1,1}$ does not admit admissible horns, \ie, $\Lambda^{1,1}_l$ is non-admissible for every $l=0,1$. This can be seen by observing that for any $a\in\{-1,0,1\}$, there are no codegeneracy maps of the form: 
			\[\begin{tikzcd}
				{[1]_1} & {[0]_1} && {[1]_1} & {[0]_0}
				\arrow["{s_0^{a}}", from=1-1, to=1-2]
				\arrow["{s_1^{a}}", from=1-4, to=1-5]
			\end{tikzcd}\]
			In other words, $\Lambda^{1,1}_l$ being non admissible implies that $l=0,1$. The converse direction is trivial from definitions \Cref{Para: isovariant coface maps}, \Cref{Par: isovariant deg maps} and \Cref{def: Admisible simplicial isovariant horn}. 
		\end{proof}
		
		\begin{remark}
			Notice that in the particular case where $n=k=1$, last proposition implies that the horns $\Lambda^{1,1}_1$ and $\Lambda^{1,1}_0$ of $\Delta^{1,1}$ are both non-admissible, this witness our philosophy that $(1,1)$-simplices should be thought as \enquote{points} and, as in classical simplicial homotopy theory, points should not have horns. We must also remark that in our theory a \enquote{point} could have different connotations. For instance, we could have \emph{real points} seen as $(0,0)$-simplices,  \emph{non-real points} seen as $(0,1)$-simplices or simply \emph{points} as $(1,1)$-simplices. We will see that there is a well defined notion of homotopy between isovariant simplicial sets and that the admissible horns inclusions correspond precisely to the horn inclusions that are isovariant elementary homotopy equivalences, we will prove this later in \ref{Lemma: characterization of isovariant horn inclusions}  once we introduce the notion of isovariant homotopy.
		\end{remark}

		\subsection{Geometric Realizations}\label{sec: realization of isovariant simplicial sets}
		
		The isovariant singular simplicial set functor from \Cref{ex:isovariantsingularsimplicialset} has a left adjoint called \emph{geometric realization} or simply \emph{realization}. There is a quick way to construct the realization functor using the theory of Kan extensions. Recall from \Cref{Example:isovariant_topological_nk_simplex} we have constructed a functor $|\Delta^{\bullet,\bullet}|: G\Deltacat \lto G\Top^{\text{isov}}$ and denote by $\catcal{Y}: G\Deltacat \lto \ssetisov$ the usual Yoneda embedding. 
		
		Since $G\Deltacat$ is small and $G\Top^{\text{isov}}$ is cocomplete and $\ssetisov$ is locally small, then the left Kan extension $\Lan_{\catcal{Y}}|\Delta^{\bullet,\bullet}| : \ssetisov \lto G\Top^{\text{isov}}$ of $|\Delta^{\bullet,\bullet}|$ along the Yoneda embedding exists \cite[\nopp 1.2.1]{riehl_2014} and can be computed as the coend
		
		\begin{align*} \Lan_{\catcal{Y}}|\Delta^{\bullet,\bullet}|(X) &= \int^{[n]_k \in G\Deltacat} \Hom(\catcal{Y}([n]_k), X) \cdot |\Delta^{n,k}| \\ &= \int^{[n]_k \in G\Deltacat} \Hom(\Delta^{n,k}, X) \cdot |\Delta^{n,k}| \\ &\simeq \int^{[n]_k \in G\Deltacat} X_{n,k} \cdot |\Delta^{n,k}|
		\end{align*}	 
		As coends are special types of colimits, notice that $|X|$ can be computed as the coequalizer \cite[\nopp 1.2.4]{riehl_2014}
		\begin{equation}\label{equation:isovaiantgeometricrealizationcoequalizer} \displaystyle
			|X| = \text{coeq} \left( 
			\begin{tikzcd} \displaystyle
				\coprod_{[n]_k\rightarrow [m]_l \in G\Deltacat} X_{m,l} \times |\Delta^{n,k}| \arrow[r,shift right=-0.3em] \arrow[r, shift left=-0.3em]& \displaystyle \coprod_{[n]_k \in G\Deltacat}  X_{n,k} \times |\Delta^{n,k}| 
			\end{tikzcd}
			\right)
		\end{equation}
		This realization $|X|$ can be equivalently computed as the colimit \cite[\nopp 1.2.6]{riehl_2014}: 
		\[
		|X|\simeq \colim \left(
		\begin{tikzcd}
			G\Deltacat \downarrow X \arrow[r,"U"] & G\Deltacat \arrow[r,"|\Delta^{\bullet,\bullet}|"] & G\Top^{\text{isov}} 
		\end{tikzcd}
		\right)
		\]
		where $G\Deltacat \downarrow X$ is the comma category having  objects isovariant simplicial set maps $\Delta^{n,k} \lto X$. Morphisms in $G\Deltacat \downarrow X$ from $x: \Delta^{n,k} \lto X$ to $y: \Delta^{m,l} \lto X$ are defined as commutative diagrams of isovariant simplicial sets
		
		\[
		\begin{tikzcd}
			\Delta^{n,k} \arrow[rr,"\theta"] \arrow[rd,"x",swap] &&\Delta^{m,l} \arrow[ld,"y"] \\ &X
		\end{tikzcd}
		\]
		Notice that $\theta$ is induced from a unique isovariant order-preserving map $\theta: [n]_k \lto [m]_l$. Moreover $U: G\Deltacat \downarrow X \lto G\Deltacat$ denotes the canonical forgetful functor. In particular, since the Yoneda embedding is fully faithful we have from \cite[\nopp 1.4.5]{riehl_2014} that \[\Lan_{\catcal{Y}}|\Delta^{\bullet,\bullet}| \circ \catcal{Y} \simeq |\Delta^{\bullet,\bullet}|.\] In other words, the realization of a representable isovariant simplicial set $\Delta^{n,k}$ coincides with the $G$-space $|\Delta^{n,k}|$ constructed in \Cref{Example:isovariant_topological_nk_simplex}. Just as in the classical simplicial theory, the realization functor is left adjoint to $\Sing^{\text{isov}}$. The proof follows essentially as in the classical case so we omit it. 
		
		\begin{lemma}\label{Lemma: realization adjoint to singular}
			The realization defined above is left adjoint to the isovariant singular simplicial set: 
			$\begin{tikzcd}
				|-| :\ssetisov \arrow[r,shift left=0.75ex] & G\Top^{\text{isov}}: \Sing^{\normalfont{\text{isov}}} \arrow[l,shift left=0.75ex] 
			\end{tikzcd}$
		\end{lemma}
		
		\begin{example}
			$|\Delta^{2,1}|$ is isomorphic in $G\Top^{\text{isov}}$ to the $2$-dimensional closed disk $\mathbb{D}^2$ with the $C_2$-action given by reflection with respect to the $y$-axis.    
		\end{example}

		\subsection{Exact Cylinders}\label{Subsection:Exact cylinders}
		
		Let $\catname{A}$ be a small category, recall that a \emph{homotopical structure} on the presheaf category $\Psh(\catname{A})$ consist of the data of an exact cylinder $I$, together with a class of $I$-anodyne extensions \cite[Définition 1.3.6]{Cisinski2016LPCMDTH}. 
		In this section we will see that there is a natural way to define an exact cylinder for $\ssetisov$ in a explicit manner. 
		
		\begin{para} \normalfont
			Recall that $G\isovposets$  denotes the category of finite $G$-posets together with isovariant order-preserving maps. There is a functor  $\thk: G\Deltacat \lto G\isovposets$, named \emph{thickening functor}, defined on objects by 
			\[ [n]_k \longmapsto \thk \left( [n]_k \right)\coloneqq [n]_k\times \{0,1\}\]
			where $\{0,1\}$ is the finite $G$-poset with trivial action by $G$ and it is ordered by $0<1$. Thus, $[n]_k\times \{0,1\}$ has component-wise ordering. The action by $G$ on $\thk([n]_k)$ is made only on the first factor as it is trivial on the second.  The finite $G$-poset $\thk([n]_k)$ can be then seen as a thickening of $[n]_k$ as in the \Cref{fig:thickenning}. 
			
			If $\alpha: [n]_k \lto [m]_l$ is in $G\Deltacat$, the thickening morphism $\thk(\alpha): \thk([n]_k) \lto \thk([m]_l)$ is given by sending: 
			
			\begin{equation}\label{Equation:definition of thk}
				([j,g],\delta) \longmapsto (\alpha[j,g],\delta), \quad g \in G, \ \delta \in \{0,1\}.
			\end{equation}
			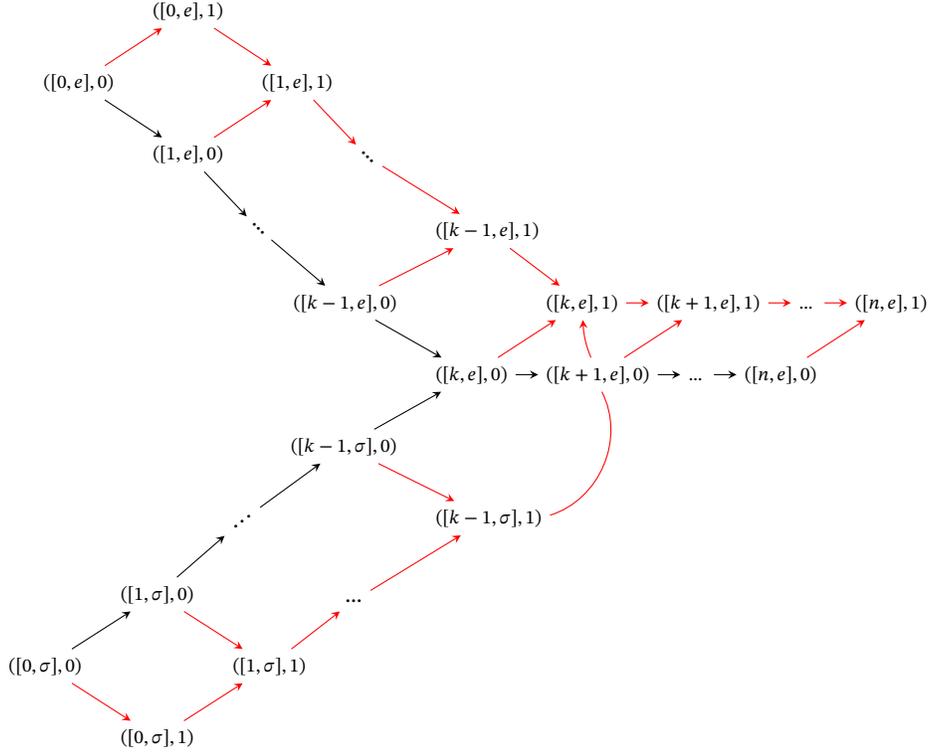
\begin{figure}[h]
				\caption{Thickening of $[n]_k$} \label{fig:thickenning}
				\begin{center}
					\begin{tikzpicture}[transform shape, align=center, 
						minimum size=3mm,
						>=stealth,
						bend angle=45,scale=1, every node/.style={scale=0.9}] 
						\tikzset{node distance = 0.5cm and 0.3cm}
						
						\node (100)  {\footnotesize $([1,e],0)$};
						\node (000) [above left =of 100] {\footnotesize $([0,e], 0)$} edge[->] (100); 
						\node (200) [below right =of 100, rotate=-45] {$\ldots$} ; 
						\node (300) [below right = of 200] {\footnotesize$([k-1,e],0)$}; 
						\node (40) [below right = of 300] {\footnotesize$([k,e],0)$}; 
						
						\node (50) [right =of 40] {\footnotesize$([k+1,e],0)$};
						\node (70) [right = of 50] {\footnotesize$\ldots \vphantom{()}$};
						\node (80) [right = of 70] {\footnotesize$([n,e],0)$};
						
						\node (310) [below left = of 40] {\footnotesize$([k-1,\sigma],0)$};
						\node (210) [below left =of 310] {$\udots \vphantom{()}$};
						\node (110) [below left = of 210] {\footnotesize$([1,\sigma],0)$};
						\node (010) [below left = of 110] {\footnotesize$([0,\sigma],0)$};
						
						\node (001) [above right =of 000] {\footnotesize$([0,e],1)$}; 
						\node (101) [above right =of 100] {\footnotesize$([1, e],1)$};
						\node (201) [below right =of 101,rotate=-45] {$\ldots$}; 
						\node (301) [above right =of 300] {\footnotesize$([k-1,e],1)$}; 
						\node (401) [above right =of 40] {\footnotesize $([k,e],1)$}; 
						
						\node (51) [right =of 401] {\footnotesize $([k+1,e],1)$};
						\node (71) [right =of 51] {\footnotesize $\ldots \vphantom{()}$};
						\node (81) [right =of 71] {\footnotesize $([n,e],1)$};
						
						\node (011) [below right =of 010] {\footnotesize$([0,\sigma],1)$}; 
						\node (111) [below right =of 110] {\footnotesize$([1,\sigma],1)$};
						\node (211) [above right =of 111] {$\ldots$}; 
						\node (311) [below right =of 310] {\footnotesize$([k-1,\sigma],1)$};

						\path[->] (100) edge (200);
						\path[->] (200) edge (300);
						
						\path[->] (300) edge (40);
						\path[->] (40) edge (50);
						\path[->] (50) edge (70);
						\path[->] (70) edge (80);
						
						\path[->] (010) edge (110);
						\path[->] (110) edge (210);				
						\path[->] (210) edge (310);
						\path[->] (310) edge (40);

						\path[red, ->] (011) edge  (111);
						\path[red, ->] (111) edge  (211);
						\path[red, ->]  (211) edge  (311);
						\path[red, -, bend right=50] (311)  edge (50);
						\path[red, ->, bend right=-10] (50) edge (401);
						
						\path[red, ->] (000) edge  (001);
						\path[red, ->] (100) edge  (101);
						\path[red, ->]  (300) edge  (301);
						\path[red, ->]  (40) edge (401);
						
						\path[red, ->] (401) edge  (51);
						\path[red, ->] (51) edge  (71);
						\path[red, ->] (71) edge  (81);
						
						\path[red, ->] (010) edge  (011);
						\path[red, ->] (110) edge  (111);
						\path[red, ->] (310) edge  (311);
						
						\path[red, ->]  (001) edge (101);
						\path[red, ->]  (101) edge (201);
						\path[red, ->]  (201) edge (301);
						\path[red, ->]  (301) edge (401);
						
						\path[red, ->] (50) edge  (51);
						\path[red, ->] (80) edge  (81);
					\end{tikzpicture}
				\end{center} 
			\end{figure}
			\noindent Clearly $\thk(\alpha)$ is well defined and moreover for if $[j,g] \in [n]_k$ with $g\in G$ , $0\leq j\leq n$ and $\delta \in \{0,1\}$  we have 
			\begin{align*}
				\thk(\alpha)	(g' \cdot([j,g],\delta)) &= \thk(\alpha)([j,g'g], \delta) \qquad &{\text{Action}} \\ &= (\alpha([j,g'g]),\delta)  \qquad  &\ref{Equation:definition of thk} \\  &= g'\cdot (\alpha([j,g]), \delta) \qquad &{\text{Equivariance of } \alpha} \\ &= g'\cdot \thk(\alpha)([j,g],\delta).
			\end{align*}
			Thus $\thk(\alpha)$ is equivariant. Similar easy calculations show that $\thk(\alpha)$ is isovariant and order-preserving.  
		\end{para} 
		For any object $[n]_k$ in $G\Deltacat$ we let $I^{n,k}: G\Deltacat^{\text{op}}\lto \Set$ be the restriction to $G\Deltacat^{\text{op}}$ of the functor represented by $\thk[n]_k$:
		
		\[\begin{tikzcd}
			{G\Deltacat^{\text{op}}} & {G\isovposets^{\text{op}}} \\
			& \Set
			\arrow["I^{n,k}"', from=1-1, to=2-2]
			\arrow[hook, from=1-1, to=1-2]
			\arrow["{\Hom_{G\isovposets}(-,\thk[n]_k)}", from=1-2, to=2-2]
		\end{tikzcd}\]
%
%
%
%
		This assigning extends to a functor $I:G\Deltacat \lto \ssetisov$ which is defined on objects by sending $[n]_k$ to $I^{n,k}$. Thus, it makes sense to consider the left Kan extension $\interval$ of $I$ along the Yoneda embedding $\catcal{Y}$: 
		
		\[ 
		\begin{tikzcd}[row sep=large]
			G\Deltacat \arrow[rr,"I"] \arrow[dr,"\catcal{Y}",hookrightarrow,swap] &&\ssetisov \\ &\ssetisov \arrow[ur,"\interval \coloneqq \Lan_{\catcal{Y}}(I)",dashed,swap]
		\end{tikzcd}
		\]
		
		\begin{proposition}\label{prop: left kan extension of R is the cylinder object for issov spaces}
			The functor $\interval$ defined as above is a functorial exact cylinder for the category $\ssetisov$. 
		\end{proposition}
		\noindent We need then to show that there exists natural transformations 
		\[ (\partial_0, \partial_1): \idd \coprod \idd \lto \interval\] 
		\[ \rho: \interval \lto \idd\]
		such that for any isovariant simplicial set $X$, the quadruple $(\interval X, \partial^X_0, \partial^X_1, \rho^X)$, obtained from $\interval$ and the natural transformations $\partial_0, \partial_1, \rho$, is a cylinder of $X$ and that it satisfies the properties: 
		
		\begin{enumerate}
			\item The functor $\interval$ commutes with small colimits and preserves monomorphisms. 
			\item For each monomorphism $\iota: X \linj Y$ in $\ssetisov$, the commutative squares 
			\[ 
			\begin{tikzcd}
				X \arrow[r,"\iota", hookrightarrow] \arrow[d,"\partial^X_{\varepsilon}",swap] &Y \arrow[d,"\partial^Y_{\varepsilon}"] \\ \interval X \arrow[r,"\interval(\iota)", swap] & \interval Y
			\end{tikzcd}
			\]
			is a pullback square for $\varepsilon=0,1$. 
		\end{enumerate}

		Before getting into the proof we will first show that we have already cylinder objects for the representables $\Delta^{n,k}$ for any $[n]_k \in G\Deltacat$.
		
		\begin{lemma}
			For each $[n]_k \in G\Deltacat$ there are isovariant simplicial set maps 
			\begin{gather*} 
				\partial_{\varepsilon}^{n,k}: \Delta^{n,k}\lto I^{n,k}, \quad \varepsilon=0,1 \\ 
				\rho^{n,k}: I^{n,k} \lto \Delta^{n,k}
			\end{gather*}
			such that $\rho^{n,k} \circ \partial^{n,k}_{\varepsilon}= \idd$. Moreover, $\partial_{\varepsilon}^{n,k}$ preserves monomorphisms in the sense that if $\alpha \in (\Delta^{n,k})_ {m,l}$ is a monomorphism then $(\partial_{\varepsilon}^{n,k})_{m,l}(\alpha)$ is a monomorphism.
		\end{lemma}
		
		\begin{proof}
			Let $[m]_l \in G\Deltacat$ and define $\left(\partial_{\varepsilon}^{n,k}\right)_{m,l}: (\Delta^{n,k})_{m,l} \lto I^{n,k}([m]_l)$ by sending
			\[ \alpha \longmapsto (\alpha(-),\varepsilon)\]
			Explicitly, 
			\[
			\left(\partial_{\varepsilon}^{n,k}\right)_{m,l}(\alpha)([j,g])= (\alpha[j,g], \varepsilon)
				\]
				where $\varepsilon = 0,1$. We must prove that $\partial_{\varepsilon}^{n,k}$ is a natural transformation. Let $f: [m]_l \lto [a]_b$ be a morphism in $G\Deltacat$, then the diagram 
				
				\begin{equation}\label{eq: partial^nk is natural}
					\begin{tikzcd}[column sep=large, row sep=large]
						(\Delta^{n,k})_{a,b} \arrow[r,"(\partial_{\varepsilon}^{n,k})_{a,b}"]  \arrow[d,"\Delta^{n,k}(f)",swap]& I^{n,k}([a]_b) \arrow[d,"I^{n,k}(f)"] \\ (\Delta^{n,k})_{m,l} \arrow[r,"(\partial_{\varepsilon}^{n,k})_{m,l}",swap] &I^{n,k}([m]_l)
					\end{tikzcd}
				\end{equation}
				commutes for if $\alpha \in (\Delta^{n,k})_{a,b}$ we have that
				
				\begin{align*}
					\left(I^{n,k}(f) \right) \left((\partial_{\varepsilon}^{n,k})_{a,b}(\alpha)\right) &= \left((\partial_{\varepsilon}^{n,k})_{a,b}(\alpha)\right) \circ f \\ &= (\alpha(-),\varepsilon) \circ f 
				\end{align*}
				but $(\alpha(-),\varepsilon ) \circ f = ((\alpha \circ f)(-), \varepsilon)$ so that
				
				\begin{align*}
					((\alpha \circ f)(-), \varepsilon) = (\partial_{\varepsilon}^{n,k})_{m,l}(\alpha \circ f) = (\partial_{\varepsilon}^{n,k})_{m,l}\left((\Delta^{n,k}(f))(\alpha)\right)
				\end{align*}
				we conclude that \[ I^{n,k}(f) \circ (\partial_{\varepsilon}^{n,k})_{a,b} = (\partial_{\varepsilon}^{n,k})_{m,l} \circ \Delta^{n,k}(f)\]
				and thus $\ref{eq: partial^nk is natural}$ is indeed commutative for $\varepsilon=0,1$. To construct $\rho^{n,k}: I^{n,k} \lto \Delta^{n,k}$ we proceed in a similar way as for $\partial_{\varepsilon}^{n,k}$: Define $\rho_{m,l}^{n,k}: I^{n,k}([m]_l) \lto (\Delta^{n,k})_{m,l}$ by sending
				
				\[ \alpha \lto \pr \circ \alpha\]
				where $\alpha \in R_{n,k}([m]_l)$ and $\pr: \thk([n]_k)\lto [n]_k$ is the projection that sends: 
				\[
				([j,g],\varepsilon) \longmapsto [j,g]
				\]
				An easy calculation similar to the one we did for $\partial_{\varepsilon}^{n,k}$ shows that for any $f: [m]_l \lto [a]_b$ the diagram 
				
				\begin{equation}\label{eq: rho^nk is natural}
					\begin{tikzcd}[column sep=huge, row sep=huge]
						I^{n,k}([a]_b) \arrow[r,"(\rho^{n,k})_{a,b}"]  \arrow[d,"I^{n,k}(f)",swap]& (\Delta^{n,k})_{a,b} \arrow[d,"\Delta^{n,k}(f)"] \\ I^{n,k}([m]_l) \arrow[r,"(\rho^{n,k})_{m,l}",swap] &(\Delta^{n,k})_{m,l}
					\end{tikzcd}
				\end{equation}
				commutes, so that $\rho^{n,k}$ is a well defined isovariant simplicial set map. Moreover notice that for $\alpha \in (\Delta^{n,k})_{m,l}$ we have
				
				\begin{align*}
					\rho^{n,k}_{m,l}\left( (\partial^{n,k}_{\varepsilon})_{m,l} ( \alpha)\right) &= \rho^{n,k}_{m,l} (\alpha(-),\varepsilon)\\ &= \pr \circ (\alpha(-), \varepsilon) = \alpha
				\end{align*}
				so that $\rho^{n,k} \circ \partial^{n,k}_{\varepsilon}= \idd$. The second part of the lemma is a direct consequence from the definition of $\partial^{n,k}_{\varepsilon}$ since if $\alpha \in (\Delta^{n,k})_{m,l}$ is a monomorphism and  $u,v: Z \lto [m]_l$ are two morphisms in $G\isovposets$ such that
				\[ \partial^{n,k}_{\varepsilon}(\alpha) \circ u = \partial_{\varepsilon}^{n,k}(\alpha) \circ v,\]
				by definition of $\partial^{n,k}_{\varepsilon}$ last equality is equivalent to 
				
				\[ ((\alpha \circ u)(-), \varepsilon) = ((\alpha \circ v)(-), \varepsilon)\]
				and since $\alpha$ is a monomorphism we then have that $u=v$.
			\end{proof}
			
			\begin{para}\label{para:cylinder for Isov_G}\normalfont
				Let $X$ be an isovariant simplicial set. Notice that $X$ is in particular a presheaf of sets over the category $G\Deltacat$ so that by Density theorem \cite[\nopp 6.2.17]{leinster2014basic} $X$ can be written as a colimit of representables: 
				\[ 
				X \simeq \colim \left(
				\begin{tikzcd}
					{G\Deltacat \downarrow X} \arrow[r,"U"] & G\Deltacat \arrow[r,"\catcal{Y}"] &\ssetisov
				\end{tikzcd}\right)
				\]
				where $U$ denotes the forgetful functor from the slice category $G\Deltacat \downarrow X$ to $G\Deltacat$. Similarly $\interval X$ was defined by a pointwise left Kan extension so that it can be written as the following colimit \cite[\nopp 1.2.6 ]{riehl_2014}: 
				
				\[ 
				\interval X = \colim \left(
				\begin{tikzcd}
					{G\Deltacat \downarrow X} \arrow[r,"U"] & G\Deltacat \arrow[r,"I"] &\ssetisov
				\end{tikzcd}\right)
				\]
				The natural maps $\partial_{\varepsilon}^{n,k}: \Delta^{n,k} \lto I^{n,k}$ and $\rho^{n,k}: I^{n,k} \lto \Delta^{n,k}$ then define a morphisms of diagrams: 
				
				\[
				\begin{tikzcd}[column sep=huge]
					G\Deltacat \downarrow X
					\arrow[bend left=50]{r}[name=U,label=above:$\catcal{Y} \circ U$]{}
					\arrow[bend right=50]{r}[name=D,label=below:$I \circ U$]{} &
					\ssetisov
					\arrow[shorten <=10pt,shorten >=10pt,Rightarrow,to path={(U) -- node[label=right:$\partial_{\varepsilon}^{\bullet, \bullet}$] {} (D)}]{}
				\end{tikzcd} \quad \text{and} \quad 
				\begin{tikzcd}[column sep=huge]
					G\Deltacat \downarrow X
					\arrow[bend left=50]{r}[name=U,label=above:$\catcal{Y} \circ U$]{}
					\arrow[bend right=50]{r}[name=D,label=below:$I \circ U$]{} &
					\ssetisov
					\arrow[shorten <=10pt,shorten >=10pt,Rightarrow,to path={(D) -- node[label=right:$\rho^{\bullet, \bullet}$] {} (U)}]{}
				\end{tikzcd}
				\]
				The universal property of colimits together with the functoriality of the colimit functor $\colim: [G\Deltacat, \ssetisov] \lto \ssetisov$ guarantee the existance of unique morphisms
				 \cite[\nopp {\stackstag{002D}} ]{stacks-project}
				
				\[ \begin{tikzcd}
					X \simeq \colim_{G\Deltacat \downarrow X}(\catcal{Y} \circ U) \arrow[r,"\partial_{\varepsilon}^X"] & \colim_{G\Deltacat \downarrow X} (I \circ U) \simeq \interval X
				\end{tikzcd}   \]
				and
				\[ \begin{tikzcd}
					\interval X \simeq \colim_{G\Deltacat \downarrow X}(I \circ U) \arrow[r,"\rho^X"] & \colim_{G\Deltacat \downarrow X} (\catcal{Y} \circ U) \simeq X
				\end{tikzcd}   \]
				such that the diagrams
				
				\[
				\begin{tikzcd}
					\Delta^{n,k} \arrow[r] \arrow[d,"\partial^{n,k}",swap] &X \arrow[d,"\partial^X"] \\ I^{n,k} \arrow[r] &\interval X
				\end{tikzcd} \quad 
				\begin{tikzcd}
					I^{n,k} \arrow[r] \arrow[d,"\rho^{n,k}",swap] &\interval X \arrow[d,"\rho^X"] \\ \Delta^{n,k} \arrow[r] &X
				\end{tikzcd}
				\]
				both commute. This implies that $\rho^X \circ \partial^X = \idd_X$ as $\rho^{n,k} \circ \partial^{n,k}=\idd_{\Delta^{n,k}}$. Notice then that for any isovariant simplicial set $X$, the quadruple $(\interval X, \partial_0^X, \partial^X_1,\rho^X)$ constructed as above is a cylinder for $X$, functoriality on $X$ is clear so that $(\interval, \partial_0, \partial_1,\rho)$ defines a functorial cylinder for $\ssetisov$ which achieves the first part of the proof of Proposition \ref{prop: left kan extension of R is the cylinder object for issov spaces}. It remains to show that such functorial cylinder is exact. We will first prove a few technical lemmas that will justify why $\interval$ preserves monomorphisms. 
			\end{para}
			
			\begin{remark}\label{Para:jumps and preservation of monos} \normalfont
				Let $X$ be an isovariant simplicial set. We have seen that $\interval X$ is defined by: 
				\[ 
				\interval X = \colim \left(
				\begin{tikzcd}
					{G\Deltacat \downarrow X} \arrow[r,"U"] & G\Deltacat \arrow[r,"I"] &\ssetisov
				\end{tikzcd}\right)
				\]
				Therefore, for any $[n]_k \in G\Deltacat$ we have 
				\[ 
				\left(\interval X\right)_{n,k} = \colim \left(
				\begin{tikzcd}
					{G\Deltacat \downarrow X} \arrow[r,"U"] & G\Deltacat \arrow[r,"I"] &\ssetisov \arrow[r,"\text{ev}_{[n]_k}"] &\Set
				\end{tikzcd}\right)
				\]
Equivalently, we can also write (\cite[\nopp 3.7.2]{borceux1994handbook})
				
				\[
				(\interval X)_{n,k} = \colim \left(
				\begin{tikzcd}
					\left( [n]_k / \thk\right)^{\op} \arrow[r,"U^{\ast}"] &G\Deltacat^{\op} \arrow[r,"X"] &\Set 
				\end{tikzcd} \right)
				\]
				where $\left( [n]_k / \thk \right)$ denotes the comma category with objects given by morphisms of isovariant $G$-posets $[n]_k \lto \thk[m]_l$. Morphisms in this category are defined in the obvious way and, $U^{\ast}$ denotes the forgetful functor. Since this is natural in $X$ we can then write $(\interval -)_{n,k}$ as the composite: 
				
				\[ 
				(\interval -)_{n,k}: \begin{tikzcd}
					\ssetisov \arrow[r,"-\circ U^{\ast}"] &\Fun[\left([n]_k / \thk \right)^{\op},\Set] \arrow[r,"\colim"] &\Set
				\end{tikzcd} 
				\]
			\end{remark}
			The forgetful functor $U^{\ast}$ has a right adjoint and in particular it preserves finite limits and thus monomorphisms. By the above, the question about whether $\interval$ preserves monomorphisms reduces then to show that for any $[n]_k$ the colimit functor: 
			
			\begin{equation}\label{Equation:colimit functor in slice category} \begin{tikzcd}
					\colim: \Fun[\left([n]_k / \thk \right)^{\op},\Set] \arrow[r] &\Set
				\end{tikzcd}
			\end{equation}
			preserves monomorphisms. Hereafter we will first characterize the categories $J$ for which monomorphisms of diagrams of sets of shape $J$ are preserved by the colimit functor. The next lemma may be known in the literature but as we don't know a reference we will add the proof.  
			
			\begin{lemma}\label{Lemma:characterization colimit preserves monos}
				Let $J$ be a small category. Then the colimit functor $\colim: \Psh(J) \lto \Set$ preserves monomorphisms if and only if every cospan in $J$ can be completed to a commutative diagram, \ie any diagram $\begin{tikzcd}
					a \arrow[r,"f"] &c & b \arrow[l,"g",swap]
				\end{tikzcd}$ in $J$ is quadrable in the sense that it admits a cone $(d,p_1,p_2)$ such that the square below commutes.
				\[
				\begin{tikzcd}
					d \arrow[r,"p_1"] \arrow[d,"p_2",swap] & a \arrow[d,"f"] \\ b \arrow[r,"g",swap] &c
				\end{tikzcd}
				\]
			\end{lemma}
			
			\begin{proof}
				We will show the sufficient condition by contrareciprocal so suppose that $\begin{tikzcd}[column sep=16pt]
					a \arrow[r,"f"] &c & b \arrow[l,"g",swap] \end{tikzcd}$ is a cospan in $J$ that is not quadrable. We will see that there are monomorphisms that cannot be preserved by the colimit functor. The morphisms $f$ and $g$ in $J$ induce a canonical morphism of presheaves: 
				\[ \begin{tikzcd}\Hom_J(-,a) \coprod \Hom_J(-,b) \arrow[rr,"f_{\ast} \coprod g_{\ast}"] &&\Hom_J(-,c)   \end{tikzcd} 
				\] 
				Since the category of presheaves $\Psh(J)$ admit epi/mono factorizations, the last morphism can be factored through some presheaf $Z$ as follows: 
				
				\[\begin{tikzcd}
					{\Hom_J(-,a) \coprod \Hom_J(-,b)} && {\Hom_J(-,c)} \\
					& Z
					\arrow["s"', two heads, from=1-1, to=2-2]
					\arrow["r"', hook, from=2-2, to=1-3]
					\arrow["{f_{\ast}\coprod g_{\ast}}", from=1-1, to=1-3]
				\end{tikzcd}\]
				where $r$ is an epimorphism and $r$ a monomorphism. We claim that $r$ is a monomorphism of presheaves that is not preserved by the colimit functor, \ie $\colim(r): \colim_{J^{\op}} Z \lto \colim_{J^{\op}} \Hom_J(-,c)$ is not a monomorphism. Indeed, let us first notice that the images of $f_{\ast}$ and $g_{\ast}$ in $\Hom_J(-,c)$ must be disjoint, otherwise there will exist some $\alpha: k \lto c$ and $\beta: k\lto a$ and $\gamma: k \lto b$ morphisms in $J$ such that 
				
				\[f\circ \beta = f_{\ast}(\beta) = \alpha = g_{\ast}(\gamma) = g \circ \gamma\]
				That is, the square 
				\[\begin{tikzcd}
					k & a \\
					b & c
					\arrow["f", from=1-2, to=2-2]
					\arrow["g"', from=2-1, to=2-2]
					\arrow["\gamma"', from=1-1, to=2-1]
					\arrow["\beta", from=1-1, to=1-2]
					\arrow["\alpha", from=1-1, to=2-2]
				\end{tikzcd}\]
				is commutative. Such thing is not possible as we have supposed that the cospan determined by $f$ and $g$ is not quadrable. As a consequence, $Z$ is a coproduct of presheaves $Z_1$ and $Z_2$ that are themselves quotients of $\Hom_J(-,a)$ and $\Hom_J(-,b)$ respectively. In particular $\colim_{J^{\op}} Z$ consist of a set with only two elements while $\colim_{J^{\op}} \Hom_J(-,c)$ consist of a single element (recall that the colimit of a representable presheaf is a singleton set). Thus $\colim(r): \colim_{J^{\op}} Z \lto \colim_{J^{\op}} \Hom_J(-,c)$ cannot be a monomorphism. 
				
				Conversely, let $f:X \lto Y$ be a monomorphism of presheaves of set over $J$. We will show that $\colim(f): \colim_{J^{\op}} X \lto \colim_{J^{\op}} Y$ is a monomorphism of sets. First, recall that for a presheaf $X$ we have that \cite[{\stackstag{002U}}]{stacks-project}: 
				
				\[\colim_{J^{\op}} X = \left(\coprod_{j \in J} X(j)\right)/\sim\]
				where $\sim$ denotes the equivalence relation defined by the following: two elements $x_i \in X(i)$ and $x_j \in X(j)$ are equivalent if and only if there exists $n \geq 0$, a chain of morphisms in $J$: 
				
				\[\begin{tikzcd}
					& {i_1} &&&& {i_{2n-1}} \\
					{i=i_0} && {i_2} && {i_{2n-2}} && {i_{2n}=j}
					\arrow[from=1-2, to=2-1]
					\arrow[""{name=0, anchor=center, inner sep=0}, from=1-2, to=2-3]
					\arrow[from=1-6, to=2-7]
					\arrow[""{name=1, anchor=center, inner sep=0}, from=1-6, to=2-5]
					\arrow["\cdots"{description}, draw=none, from=0, to=1]
				\end{tikzcd}\]
				and elements $x_{i_k} \in X(i_k)$ with $k=0,1,...,2n$ such that they map to each other under the induced maps from the chain: 
				\[ X(i_{2k})\lto X(i_{2k-1}) \quad \text{and} \quad X(i_{2k-2})\lto X(i_{2k-1}).\] 
				
				Suppose then that $a \in X(i_0)$ and $b \in X(i_1)$ are two elements representing equivalence classes in $\colim_{J^{\op}} X$ and each they are such that the maps $f_{i_0}: X(i_0) \lto Y(i_0)$ and $f_{i_1}:X(i_1) \lto Y(i_1)$ send $a$ and $b$ to the same class in $\colim_J Y$. In other words, $f_{i_0}(a)$ and $f_{i_1}(b)$ are equal in the colimit $\colim_{J^{\op}} Y$. Then by definition, there exists $n\geq 0$ and a chain of morphisms in $J$: 
				\begin{equation}\label{Equation:zigzag colimit characterization}\begin{tikzcd}
						& {k_1} && {} && {k_{2n-1}} \\
						{i_0=k_0} && {k_2} && {k_{2n-2}} && {k_{2n}=i_1}
						\arrow[from=1-2, to=2-1]
						\arrow[""{name=0, anchor=center, inner sep=0}, from=1-2, to=2-3]
						\arrow[from=1-6, to=2-7]
						\arrow[""{name=1, anchor=center, inner sep=0}, from=1-6, to=2-5]
						\arrow["\cdots"{description}, draw=none, from=0, to=1]
				\end{tikzcd}\end{equation}
				and elements $y_{k_j} \in Y(k_j)$ where $j=0,1,...,2n$ such that $y_{k_0}=y_{i_0}=f(a)$, $y_{k_{2n}}=y_{i_1}=f(b)$ and they are sent to each other under the corresponding induced morphisms 
				\[
				Y(k_{2m})\lto Y(k_{2m-1}) \quad \text{and} \quad Y(k_{2m-2})\lto Y(k_{2m-1}).
				\]
				By repeatedly completing each cospan in \ref{Equation:zigzag colimit characterization} and after a finite number of completions we have that there exists some object $i \in J$ and morphisms $\begin{tikzcd}[column sep=12pt]
					i_0 & i \arrow[l,"\varphi",swap] \arrow[r,"\psi"] &i_1 
				\end{tikzcd}$ 
				and an element $c \in Y(i)$ such that $Y(\varphi)(f_{i_0}(a))=c$ and $Y(\psi)(f_{i_1}(b))=c$. By naturality of $f$ we have that 
				\begin{equation*}
					\begin{split}
						Y(\varphi) \circ f_{i_0} = f_i \circ X(\varphi) \\ 
						Y(\varphi) \circ f_{i_1} = f_i \circ X(\psi)
					\end{split}
				\end{equation*}
				therefore, 
				\begin{equation*}
					f_i(X(\varphi(a))) = c = f_i(X(\varphi(b)))
				\end{equation*}
				since $f$ is a monomorphism last equation implies that $X(\varphi)(a) = X(\psi)(b)$ which implies that $a$ and $b$ are equal in $\colim_J X$. 
			\end{proof}
			
			We will use the last lemma to justify why the colimit functor from \ref{Equation:colimit functor in slice category} preserves monomorphisms. Explicitly, we will show that any cospan in the comma category $([n]_k/\thk)$ is quadrable. Before doing so, let us recall that the category of isovariant $G$-posets $G\isovposets$ has biproducts given as follows: Let $A$ and $B$ two isovariant $G$-posets and define their isovariant product $A \times_{\text{isov}}B$ by: 
			
			\[ A \times_{\text{isov}}B \coloneqq \left\{(a,b) \in A\times B \ / \ G_a = G_b \right\}\]  
			That is, the isovariant product consist of the ordered pairs in the usual cartesian product $A \times B$ having the same isotropy groups. The action of $G$ is defined by the diagonal action. Notice that since $G$ is the two-element group, isotropy groups are either trivial or equal to $G$. Moreover, whenever an element has $G$ as isotropy group this means that is an element fixed by the action of $G$. Otherwise is non-fixed. Therefore if we denote by $A^G$ the fixed points of $A$ by the action of $G$ and by $A^{\cancel{G}}$ the non-fixed points of $A$ and similarly for $B$, we can write their isovariant product as: 
			
			\[A \times_{\text{isov}} B = (A^G \times B^G)\cup (A^{\cancel{G}} \times B^{\cancel{G}}).\]
			
			The natural projections $A\times_{\text{isov}} B \lto A$ and $A \times_{\text{isov}} B \lto B$ are of course isovariant order-preserving maps and thus the isovariant product we have just defined $\times_{\text{isov}}$ is a categorical biproduct in the category $G\isovposets$. Neverthless, this category does not have all arbitrary products as clearly it does not have a final object. \\
			
			On the other hand, the category $G\isovposets$ also has fiber products given by the usual formulas as in the category of sets. That is, if $f:A\lto C$ and $h:B \lto C$ are two isovariant maps of $G$-posets, then their fiber product $A \times_C B$ is given by: 
			
			\[A \times_C B \coloneqq \left\{(a,b)\in A \times_{\text{isov}} B \ / \ f(a)=g(b)\right\}. \]
			The action of $G$ on $A \times_C B$ is again given by the diagonal action and it is easy to see that $A \times_C B$ is universal.   
			
			\begin{proposition}\label{Proposition:comma category nk/thk is quadrable}
				Let $[n]_k$ be an object in $G\Deltacat$. Then the comma category $([n]_k/\thk)$ 
				has the property that any cospan is quadrable. 
			\end{proposition}
			
			\begin{proof}
				We begin by first noticing that a cospan in the comma category $([n]_k/\thk)$ consist of a diagram with commutative triangles: 
				\begin{equation}\label{Equation: quadrable cospan in relevant comma cat}
					\begin{tikzcd}[column sep=small,row sep=small]
						&& {\text{th}[m]_l} \\
						&&& {[n]_k} & {} \\
						{\text{th}[p]_q} && {\text{th}[a]_b} \\
						& {[n]_k}
						\arrow["\alpha"', from=4-2, to=3-1, swap]
						\arrow["\beta"', from=4-2, to=3-3]
						\arrow["f", from=3-1, to=3-3]
						\arrow["g"', from=1-3, to=3-3]
						\arrow["\beta", from=2-4, to=3-3]
						\arrow["\gamma"', from=2-4, to=1-3]
					\end{tikzcd}
				\end{equation}
				We will denote by $P$ the fiber product of $f$ and $g$ in the category of $G$-isovariant posets. Then since $g \circ \gamma = f \circ \alpha = \beta$, by the universal property of the fiber product there exists a unique $h:[n]_k \lto P$ such that the following diagram commutes: 
				
				
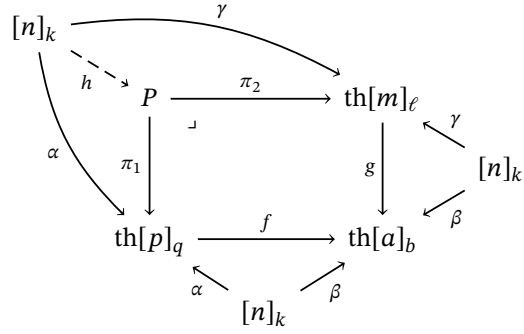
\[\begin{tikzcd}[column sep=small,row sep=small]
					{[n]_k} \\
					& P && {\text{th}[m]_l} \\
					&&&& {[n]_k} & {} \\
					& {\text{th}[p]_q} && {\text{th}[a]_b} \\
					&& {[n]_k}
					\arrow["\alpha"', from=5-3, to=4-2,swap]
					\arrow["\beta"', from=5-3, to=4-4]
					\arrow["f", from=4-2, to=4-4]
					\arrow["g"', from=2-4, to=4-4]
					\arrow["\beta", from=3-5, to=4-4]
					\arrow["\gamma"', from=3-5, to=2-4]
					\arrow["{\pi_1}"', from=2-2, to=4-2]
					\arrow["{\pi_2}", from=2-2, to=2-4]
					\arrow["h"', dashed, from=1-1, to=2-2]
					\arrow["\gamma", bend left={18pt}, from=1-1, to=2-4]
					\arrow["\alpha"', bend right={18pt}, from=1-1, to=4-2]
					\arrow["\lrcorner"{anchor=center, pos=0.125}, draw=none, from=2-2, to=4-4]
				\end{tikzcd}\]
				Notice that if the cospan from \ref{Equation: quadrable cospan in relevant comma cat} were quadrable in $([n]_k/\thk)$ then there would be $[r]_s, \varphi: \thk[r]_s \lto \thk [n]_k, \psi: \thk[r]_s \lto \thk[m]_l$ and $\delta:[n]_k \lto \thk[r]_s$ such that the following diagram commutes:
				
				
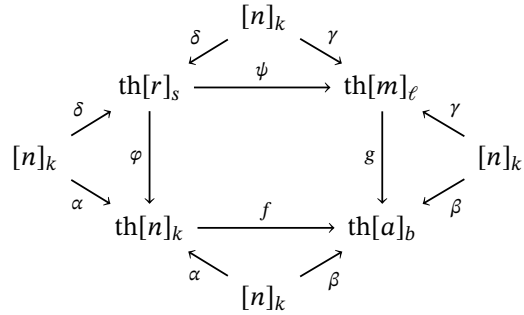
\[\begin{tikzcd}[column sep=small,row sep=small]
					&& {[n]_k} \\
					& {\text{th}[r]_s} && {\text{th}[m]_l} \\
					{[n]_k} &&&& {[n]_k} & {} \\
					& {\text{th}[n]_k} && {\text{th}[a]_b} \\
					&& {[n]_k}
					\arrow["\beta"', from=5-3, to=4-4]
					\arrow["f", from=4-2, to=4-4]
					\arrow["g"', from=2-4, to=4-4]
					\arrow["\beta", from=3-5, to=4-4]
					\arrow["\gamma"', from=3-5, to=2-4]
					\arrow["\psi", from=2-2, to=2-4]
					\arrow["\varphi"', from=2-2, to=4-2]
					\arrow["\delta"', from=1-3, to=2-2]
					\arrow["\gamma", from=1-3, to=2-4]
					\arrow["\delta", from=3-1, to=2-2]
					\arrow["\alpha"', from=3-1, to=4-2]
					\arrow["\alpha", from=5-3, to=4-2]
				\end{tikzcd}\]
				Again the commutativity of such diagram and the universal property of the fiber product guarantee the existence of a unique $\bar{h}: \thk[r]_s \lto P$ such that $\pi_1 \circ \bar{h} = \varphi$ and $\pi_2 \circ \bar{h} = \psi$.
%
				Therefore we will have that 
				\begin{equation*}
					\begin{split}
						\pi_1 \circ \bar{h} \circ \delta = \varphi \circ \delta = \alpha \\ 
						\pi_2 \circ \bar{h} \circ \delta = \psi \circ \delta = \gamma
					\end{split}
				\end{equation*}
				and by uniqueness one must have that $h = \bar{h} \circ \delta$. Therefore, to complete the cospan from \ref{Equation: quadrable cospan in relevant comma cat} it is enough to show that any isovariant map of $G$-posets of the form $h: [n]_k \lto P$ factors through the thickening of some $[r]_s$ as follows:
				
				
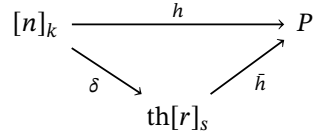
\[\begin{tikzcd}
					{[n]_k} && P \\
					& {\text{th}[r]_s}
					\arrow["h", from=1-1, to=1-3]
					\arrow["\delta"', from=1-1, to=2-2]
					\arrow["{\bar{h}}"', from=2-2, to=1-3]
				\end{tikzcd}\]
				Notice that in this case the data $(\varphi, \psi, \delta)$ will be then determined by the last factorization of $h$. Nevertheless, an isovariant map of $G$-posets $h: [n]_k \lto P$ can be always be factored by the thickening of $[n]_k$ by considering for instance
				
				
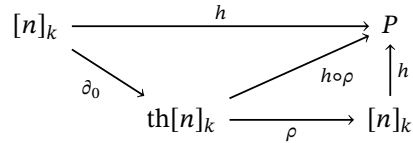
\[\begin{tikzcd}
					{[n]_k} &&& P \\
					& {\text{th}[n]_k} && {[n]_k}
					\arrow["h", from=1-1, to=1-4]
					\arrow["{\partial_0}"', from=1-1, to=2-2]
					\arrow["{h \circ \rho}"'{pos=0.6}, from=2-2, to=1-4]
					\arrow["\rho"', from=2-2, to=2-4]
					\arrow["h"', from=2-4, to=1-4]
				\end{tikzcd}\]
				Here $\partial_0: [n]_k \lto \thk[n]_k$ denotes the canonical map sending $[i,g] \longmapsto ([i,g],0)$ and $\rho: \thk[n]_k \lto [n]_k$ is the canonical projection and this completes the proof.   
			\end{proof}

					\begin{proposition}\label{prop: I preserves colimits and monos}
						Let $\interval: \ssetisov \lto \ssetisov$ be the functor defined in \ref{para:cylinder for Isov_G}. Then $\interval$ preserves colimits and monomorphisms.
					\end{proposition}
					\begin{proof}
						First part holds simply from the fact that $\interval$ is a pointwise left Kan extension along the Yoneda embedding so that $\interval$ is given by a certain colimit (see \ref{para:cylinder for Isov_G}) and therefore it must preserve colimits. Second part follows immediately from \Cref{Para:jumps and preservation of monos}, \Cref{Lemma:characterization colimit preserves monos} and \Cref{Proposition:comma category nk/thk is quadrable}. 
					\end{proof}	
					%
				
				\begin{lemma}\label{lemma: cylinder functor and pullbacks}
					For each monomorphism $\iota: X \  \linj \ Y$ in $\ssetisov$, the commutative squares 
					\[ 
					\begin{tikzcd}
						X \arrow[r,"\iota", hookrightarrow] \arrow[d,"\partial^X_{\varepsilon}",swap] &Y \arrow[d,"\partial^Y_{\varepsilon}"] \\ \interval X \arrow[r,"\interval (\iota)", swap,hookrightarrow] & \interval Y
					\end{tikzcd}
					\]
					are pullback squares for $\varepsilon=0,1$. 
				\end{lemma}
				
				\begin{proof}
					First thing to notice is that all morphisms in the last square are monomorphisms since $\interval$ preserves monomorphisms and we have seen that $\partial_{\varepsilon}^X$ and $\partial_{\varepsilon}^Y$ are also monomorphisms. Furthermore, pullbacks in $\ssetisov$ are computed pointwise so that it is enough to prove that for any $[n]_k \in G\Deltacat$ and $\varepsilon=0,1$ the squares
					
					\begin{equation}\label{Equation:pullback diagram for exact cylinder} 
						\begin{tikzcd}
							X_{n,k} \arrow[r,"\iota_{n,k}", hookrightarrow] \arrow[d,"\partial^X_{\varepsilon}",swap,hookrightarrow] &Y_{n,k} \arrow[d,"\partial^Y_{\varepsilon}",hookrightarrow] \\ (\interval X)_{n,k} \arrow[r,"\interval(\iota)_{n,k}", swap,hookrightarrow] & (\interval Y)_{n,k}
						\end{tikzcd}
					\end{equation}
					are pullback squares in $\Set$. To simplify notation and since $[n]_k$ is fixed, we will omit the subscripts $n,k$ and we will denote by $f$ the morphism $\interval(\iota)_{n,k}$. Hence, as all maps in diagram \ref{Equation:pullback diagram for exact cylinder} are injective then the pullback $P$ of $\partial_{\varepsilon}^Y$ along $f$ is given by
					\[ P \coloneqq (\partial_{\varepsilon}^Y)^{-1}\left(f(\interval X)\right).\]
					On the other hand, by the universal property of pullbacks and the commutativity of diagram \ref{Equation:pullback diagram for exact cylinder} there exists a unique map $\theta: X \lto P$ such that
					\[\begin{tikzcd}
						X && Y \\
						& P \\
						{\mathbb{I}X} && {\mathbb{I}Y}
						\arrow["{\partial_{\varepsilon}^X}"', from=1-1, to=3-1]
						\arrow["f"', from=3-1, to=3-3]
						\arrow["{\partial_{\varepsilon}^Y}", from=1-3, to=3-3]
						\arrow["\iota", from=1-1, to=1-3]
						\arrow["\theta"', dashed, from=1-1, to=2-2]
						\arrow[from=2-2, to=1-3]
						\arrow[from=2-2, to=3-1]
					\end{tikzcd}\]
					commutes. The map $\theta$ is in particular given by $\theta(x)=\iota(x)$ for each $(n,k)$-simplex $x$ in $X$. This is of course well defined since 
					\[\partial_{\varepsilon}^Y(\iota(x)) = f(\partial_{\varepsilon}^X(x)) \in f(\interval X).\] 
					so that $\iota(x) \in P$. Additionally, the map $\theta$ is injective as $\iota$ is. We will show next that $\theta$ is also surjective. Let $y \in P$, then by definition of $P$ there exists $z \in \interval X$ such that 
					\[\partial_{\varepsilon}^Y(y)= f(z).\]
					Let $\bar{x}\coloneqq \rho^X(z)$. Then, using the naturality of $\rho: \interval \lto \idd$ and the equation above, we have:  
					\begin{align*}
						\theta(\bar{x})&= \iota(\rho^X(z)) \\ &=\rho^Y(f(z)) \\ &= \rho^Y(\partial_{\varepsilon}^Y(y)) \\&=y.
					\end{align*}
					Thus, $\theta$ is indeed surjective. As a consequence
					\[ X_{n,k}\simeq (\partial_{\varepsilon}^Y)^{-1}\left(f(\interval X)\right) = (\interval X)_{n,k}\times_{(\interval Y)_{n,k}}Y_{n,k}.  \] 
					We conclude that the squares \ref{Equation:pullback diagram for exact cylinder} are indeed pullback squares. 
				\end{proof}
				\Cref{prop: I preserves colimits and monos} and \Cref{lemma: cylinder functor and pullbacks} then finish the proof of \Cref{prop: left kan extension of R is the cylinder object for issov spaces}.

				\subsection{Normal Monomorphisms}
				
				\begin{definition}[{\cite[8.1.23]{Cisinski2016LPCMDTH}}]\label{Definition:dominant sections}
					Let $X$ be an isovariant simplicial set and $u: \Delta^{m,l} \lto X$ be a non degenerate $(m,l)$-simplex of $X$. The simplex $u$ is said to be \emph{dominant} if the pair $([m]_l,u)$ does not have nontrivial automorphisms in the over category $G\Deltacat\downarrow X$. \\
					Furthermore, an $(n,k)$-simplex $v: \Delta^{n,k} \lto X$ is called \emph{normal} if there exists a pair $(\pi,u)$ where $\pi: [n]_k \lto [m]_l$ is in $G\Deltacat_{-}$ (is an epimorphism) and $u:\Delta^{m,l}\lto X$ is a dominant simplex such that 
					\[X(\pi)(u)=v.\]
					An isovariant simplicial set map $X$ is said to be \emph{normal} if all simplices of $X$ are normal. Notice that any non degenerate simplex $v$ of $X$ is normal if and only if is dominant \cite[\nopp 8.1.23]{Cisinski2016LPCMDTH}.
				\end{definition}		
				
				\begin{remark}\label{Remark:meaning dominant sections}
					Recall that an automorphism of $([m]_l,u)$ in $G\Deltacat \downarrow X$ consist of an isovariant simplicial set isomorphism $\varphi: \Delta^{m,l}\lto \Delta^{m,l}$ such that the map $\varphi$ and its inverse $\varphi^{-1}$ are both morphisms over $X$ in the sense that the diagrams
					\[\begin{tikzcd}
						{\Delta^{m,l}} && {\Delta^{m,l}} & {\text{and}} & {	\Delta^{m,l}} && {\Delta^{m,l}} \\
						& X &&&& X
						\arrow["\varphi", from=1-1, to=1-3]
						\arrow["u"', from=1-1, to=2-2]
						\arrow["u", from=1-3, to=2-2]
						\arrow["{\varphi^{-1}}", from=1-5, to=1-7]
						\arrow["u"', from=1-5, to=2-6]
						\arrow["u", from=1-7, to=2-6]
					\end{tikzcd}\]
					both commute. For instance, if we consider $X= \Delta^{0,1}$ and $\sigma: \Delta^{0,1}\lto \Delta^{0,1}$ is the $(0,1)$-simplex induced from the swapping map $\sigma:[0]_1 \lto [0]_1$ then $\sigma$ is normal since $\sigma$ is non degenerate and if $\varphi \in \Aut_{G\Deltacat \downarrow \Delta^{0,1}}([0]_1,\sigma)$, then we have two commutative diagrams
					\[\begin{tikzcd}
						{\Delta^{0,1}} && {\Delta^{0,1}} & {\text{and}} & {	\Delta^{0,1}} && {\Delta^{0,1}} \\
						& {\Delta^{0,1}} &&&& {\Delta^{0,1}}
						\arrow["\varphi", from=1-1, to=1-3]
						\arrow["\sigma"', from=1-1, to=2-2]
						\arrow["\sigma", from=1-3, to=2-2]
						\arrow["{\varphi^{-1}}", from=1-5, to=1-7]
						\arrow["\sigma"', from=1-5, to=2-6]
						\arrow["\sigma", from=1-7, to=2-6]
					\end{tikzcd}\]
					Since the only automorphisms of $\Delta^{0,1}$ are the identity and $\sigma$ itself, notice that $\varphi$ cannot be $\sigma$, otherwise by the commutativity of last diagrams we will have that $\sigma \circ \sigma = \sigma$ which is not possible as $\sigma^2 = \idd$. As a consequence $\varphi$ must be $\idd_{\Delta^{0,1}}$ and therefore 
					\[ \Aut_{G\Deltacat \downarrow \Delta^{0,1}}([0]_1,\sigma)= \left\{\idd_{\Delta^{0,1}}\right\}.\]
					In general, representable presheaves are always normal (\cite[\nopp 8.1.23]{Cisinski2016LPCMDTH}) so that in particular the isovariant simplicial sets of the form $\Delta^{n,k}$ are normal.
				\end{remark}	
				
				
				\begin{definition}\label{Definition:normal monomorphism}
					A monomorphism of isovariant simplicial sets $X\lto Y$ is \emph{normal} if any non degenerate simplex $y:\Delta^{m,l}\lto Y$, which does not factor through $X$, is normal. Notice that an isovariant simplicial set $X$ is normal if and only if the canonical monomorphism $\emptyset \lto X$ from the initial isovariant simplicial set to $X$ is normal. 
				\end{definition}
				
				Normal monomorphisms make sense in any category of presheaves over an skeletal category. A fundamental result in this theory is that normal monomorphisms form a saturated class. This class is moreover \enquote*{generated} by the boundary inclusions. This was proved in full generality by Cisinski in \cite[\nopp 8.1.35]{Cisinski2016LPCMDTH}. As our category $G\Deltacat$ is skeletal this result also makes sense for isovariant simplicial sets: 
				
				\begin{proposition}
					The class of normal monomorphisms in $\ssetisov$ is equal to the smallest saturated class of isovariant simplicial set maps containing the boundary inclusions $\partialbf\Delta^{n,k}\lto \Delta^{n,k}$. In particular the class of normal monomorphisms is closed under pushouts, transfinite compositions and retracts.
				\end{proposition}
				
				\begin{para}\label{Paragraph:Properties_normal_monomorphisms}  \normalfont Normal monomorphisms can be characterized by certain free actions on simplices. In this note we will observe the last characterization and some other properties already known in the context of Dendroidal Sets \cite{Cisinski2011}. 
					\begin{enumerate}[label=\stlabel{Paragraph:Properties_normal_monomorphisms}, ref={\arabic*}] 
						\item \label{Paragraph:Properties_normal_monomorphisms.1} Suppose $X \lto Y$ is a normal monomorphism of isovariant simplicial sets and let $[n]_k$ be any object in $G\Deltacat$. There is a natural action of $\Aut_{G\Deltacat}([n]_k)$ on the $(n,k)$-simplices of $Y$ given by: 
						\begin{center}	$ 
							\begin{array}{lll}
								\Aut_{G\Deltacat}([n]_k)&\times Y_{n,k} &\longrightarrow Y_{n,k} \\
								&(\psi,y) &\longmapsto y \circ \psi
							\end{array}$
						\end{center}
						where, as usual, we are identifying $\psi$ with its corresponding morphism $\psi: \Delta^{n,k} \lto \Delta^{n,k}$ under the Yoneda isomorphism and $y:\Delta^{n,k}\lto Y$. Now, suppose that $y$ is an $(n,k)$-simplex of $Y$ which does not factor through $X$ or equivalently $y \in Y_{n,k}\smallsetminus X_{n,k}$ and additionally consider the class of morphisms: 
						\[ \catcal{N}= 
						\left\{
						\begin{aligned}
							\iota: X \linj Y \in \ssetisov \ : \ &\iota \text{ is a monomorphism and for any } [n]_k, \\ 
							&\Aut_{G\Deltacat}([n]_k) \text{ acts freely on } Y_{n,k}\smallsetminus X_{n,k}. 
						\end{aligned}
						\right\}
						\]
						The class $\catcal{N}$ is closed under pushouts, retracts and transfinite compositions. Moreover it contains the boundary inclusions $\partialbf \Delta^{n,k}\lto \Delta^{n,k}$ so that by saturation the class of normal monomorphisms is contained in $\catcal{N}$. The converse assertion is evident. As a consequence normal monomorphisms can be characterized as monomorphisms satisfying the clause that defines the class of morphisms $\catcal{N}$. An immediate consequence of this fact is that an isovariant simplicial set $X$ is normal if and only if for any $[n]_k$, the action of the group of automorphisms $\Aut_{G\Deltacat}([n]_k)$ on $X_{n,k}$ is free. 
						\item \label{Paragraph:Properties_normal_monomorphisms.2} Given any isovariant simplicial set map $f:X\lto Y$ with $Y$ normal then it turns out that $X$ is normal as well. This follows in the same way as in (\cite[\nopp 3.4.5]{Moerdijk2010}). The fact that $Y$ is normal implies that it can be built up by attaching cells in the sense that the following square is a pushout square
						\[ 
						\begin{tikzcd}
							\displaystyle
							\coprod_{([n]_l,x) \in \Sigma} (\partial \Delta^{n,k})_{x} \arrow[r] \arrow[d,hookrightarrow] & \sk_{n-1}(Y) \arrow[d,hookrightarrow] \\ \displaystyle \coprod_{([n]_l,x) \in \Sigma} (\Delta^{n,k})_{x} \arrow[r] &  \sk_n(Y)
						\end{tikzcd}
						\]
						Here $\Sigma$ denotes the set of isomorphisms classes of non degenerate objects $([n]_l,x)$ in $G\Deltacat\downarrow Y$, $0\leq l\leq n+1$ and, $Y=\colim_{n}\sk_n(Y)$. This implies that, for any non-degenerate $(n,k)$-simplex $y$ of $Y$, the corresponding isovariant simplicial set map $y: \Delta^{n,k}\lto Y$ has a left cancellation property \cite[Lemma 3.4.5]{Moerdijk2010} with respect to epimorphisms in $G\Deltacat$: If $\alpha, \beta:\Delta^{m,l} \twoheadrightarrow \Delta^{n,k}$ are epimorphisms such that $y \circ \alpha = y\circ \beta$ then $\alpha = \beta$. 
						
						To show that $X$ is normal we can then use the characterization stated in \enumref{Paragraph:Properties_normal_monomorphisms}{1}. Let $x:\Delta^{n,k}\lto X$ be a $(n,k)$-simplex of $X$ and suppose that $\Psi$ is an automorphism of $[n]_k$ such that 
						\[ x \circ \Psi = x.\]
						Then, by post-composing with $f$ we have that $f \circ x \circ \Psi = f \circ x$. If $f\circ x$ is non-degenerate then, the fact that $Y$ is normal and the property \enumref{Paragraph:Properties_normal_monomorphisms}{1} imply that $\Psi = \idd$. If $f \circ x$ is degenerate then there exist $y \in Y_{m,l}$ non-degenerate and $s:[n]_k \twoheadrightarrow [m]_l$ a codegeneracy such that 
						\[f \circ x = y \circ s.\]
						Hence, we have that 
						\[y \circ s \circ \Psi = f \circ x \circ \Psi = f \circ x = y \circ s\]
						thus, by left cancellation we must have
						\[s \circ \Psi = s.\]
						In other words, the following two equivalent diagrams commute:
						\[\begin{tikzcd}
							{[n]_k} & {[n]_k} && {\Delta^{n,k}} & {\Delta^{n,k}} \\
							& {[m]_l} &&& {\Delta^{m,l}}
							\arrow["\Psi","\simeq"', from=1-4, to=1-5]
							\arrow[""{name=0, anchor=center, inner sep=0}, "s"', two heads, from=1-4, to=2-5]
							\arrow["s", two heads, from=1-5, to=2-5]
							\arrow["{\Psi }","\simeq"', from=1-1, to=1-2]
							\arrow["s"', two heads, from=1-1, to=2-2]
							\arrow[""{name=1, anchor=center, inner sep=0}, "s", two heads, from=1-2, to=2-2]
							\arrow[shorten <=32pt, shorten >=32pt, Leftrightarrow, from=1, to=0]
						\end{tikzcd}\] 
						But this is possible only when $\Psi = \idd$. By using \enumref{Paragraph:Properties_normal_monomorphisms}{1} we conclude that $X$ is normal.
						\item \label{Paragraph:Properties_normal_monomorphisms.3} The last item implies that if $K \linj L$ is any monomorphism and $L$ is normal then $K \linj L$ must be a normal monomorphism.
					\end{enumerate}
				\end{para}	
	
	\section{Model Structures for Isovariant Simplicial Sets}\label{Section: model structure for ssetisov}

\subsection{Isovariant Homotopy}
Using the cylinder functor constructed in last section we can then define isovariant homotopies as in \cite[1.3.3]{Cisinski2016LPCMDTH}.
\begin{definition}\label{Definition:isovariant homotopy}
	Let $f, g: X \lto Y$ be two isovariant simplicial set maps. An \emph{isovariant elementary homotopy} from $f$ to $g$ is an isovariant simplicial set map $H: \interval X \lto Y$ such that the diagram 
	\[ 
	\begin{tikzcd}
		X \arrow[dr,"\partial^X_0",swap] \arrow[drr, "f", bend left=30] \\ &\interval X \arrow[r,"H"] & Y \\ 
		X \arrow[ru,"\partial^X_1"] \arrow[urr,"g",bend right=30,swap]
	\end{tikzcd}
	\]
	commutes. The quotient of $\Hom_{\ssetisov}(X,Y)$ by the smallest equivalence relation generated by the relation of elementary isovariant homotopy will be denoted by $[X,Y]_{\text{isov}}$, if the context is clear sometimes we will simply write $[X,Y]$. Isovariant homotopy equivalences are then isovariant simplicial set maps that induce isomorphisms in the homotopy category $\catname{Ho}(\ssetisov)$. Equivalently,  $\ f: X \lto Y$ in $\ssetisov$ is an isovariant homotopy equivalence if for any isovariant simplcial set $W$, the induced map  $\ f^{\ast}: [Y,W] \lto [X,W]$ is a bijection. 
\end{definition}

The first fundamental result of this section is the characterization of admissible horns we mentioned in \ref{def: Admisible simplicial isovariant horn}. This characterization and its proof are partially inspired from \cite[\nopp 1.13]{douteau2020simplicial}.

%

\begin{lemma}\label{Lemma: characterization of isovariant horn inclusions} 
	Let $n> 0$, $0\leq k \leq n+1$ and $0\leq l \leq n$. The following are equivalent:
	\begin{enumerate}[label=\textbf{\arabic*.}, ref=\arabic*]
		\item The $l$-horn $\Lambda^{n,k}_l$ is admissible.
		\item The inclusion $\Lambda^{n,k}_l \ \linj \ \Delta^{n,k}$ is an isovariant elementary homotopy equivalence. 
	\end{enumerate} 
\end{lemma}

\begin{proof}
	Suppose first that $\Lambda^{n,k}_l$ is admissible and suppose aditionally that $0 \leq l \leq k-1$. Thus $\Lambda^{n,k}_l$ being admissible means that there exist $a \in  \{l,l-1\}$ such that \[ d_1^{l}\circ s_1^a(\Delta^{n,k}) = \partialbf_1^{l}(\Delta^{n,k}) = \im(d_1^l: \Delta^{n-1,k-1}\lto \Delta^{n,k}).\]
	We will prove the assertion when $a=l$, the case $a=l-1$ follows the same idea. Define $\varphi_l: \Delta^{n,k}\lto \Delta^{n,k}$ as the composition:
	
	\[\varphi_l \coloneqq \begin{cases}
		d_1^{l+1} \circ s_1^l \quad &\text{if } l<k-1 \\ 
		d_1^{l} \circ s_1^{l-1} \quad &\text{if } l=k-1
	\end{cases}\]
	First, notice that in each case, the isovariant simplicial set map $\varphi_l$ is isovariantly homotopic to the identity $\idd_{\Delta^{n,k}}$. To construct an explicit homotopy we can first define ${\widetilde{H}}_l: \thk([n]_k) \lto [n]_k$ by, 
	
	\begin{equation*}
		\Htilde_l\paren{[j,g],\delta} = 
		\begin{cases}
			[l,g] &\text{if  }	j=l+1, \ l<k-1, \ \delta=0 \\ 
			[l-1,g] &\text{if  } j=l, \ l=k-1, \ \delta=0 \\ 
			[j,g] &\text{otherwise}
		\end{cases} 
	\end{equation*}
	Notice that in any case $\widetilde{H}_l$ is well defined and it is an order-preserving isovariant map. Consider $H_l \coloneqq \Hom_{G\isovposets}(-, \widetilde{H}_l): I^{n,k} \lto \Delta^{n,k}$, where recall $I^{n,k}=\Hom_{G\isovposets}(-,\thk([n]_k)) \cong \interval \Delta^{n,k}$. Thus $H_l$ is the natural transformation defined at each object $\alpha \in \Hom_{G\isovposets}([a]_b, \thk([n]_k))$ by
	\[ (H_l)_{[a]_b}(\alpha) = \widetilde{H} \circ \alpha. \] 
	Hence if $[j,g] \in [a]_b$ we have that
	\begin{align*}
		\left(\left(H_l \circ \partial_0^{n,k} \right)_{[a]_b}(\alpha)\right)([j,g]) &= \left(H_l \circ \partial_0^{n,k}(\alpha)\right)([j,g]) \\ &=H_l(\alpha(-),0)([j,g]) \\ &=\widetilde{H}_l\left(\alpha([j,g]),0\right) \\
		&=	\begin{cases}
					[l,g'] &\text{if  }	\alpha([j,g])=[l+1,g'], \ l<k-1 \\ 
					[l-1,g'] &\text{if  } \alpha([j,g])=[l,g'], \ l=k-1, \\ 
					\alpha[j,g] &\text{otherwise}
			\end{cases}
	\end{align*}
On the other hand, using the definition of $\varphi_l$ and the coface and codegeneracy maps we have that 
	\begin{align*}
		\paren{(\varphi_l)_{[a]_b}(\alpha)}([j,g])  		&=	\begin{cases}
					d_1^{l+1}(s_1^l(\alpha([j,g]))) &\text{if  } \alpha([j,g])=[l+1,g'], \ l<k-1 
					\\
					d_1^l(s_1^{l-1}(\alpha([j,g]))) &\text{if  } \alpha([j,g])=[l,g'], \ l=k-1  \\
					\alpha[j,g] &\text{otherwise}
				\end{cases} \\
			&=  \begin{cases}
					[l,g'] &\text{if  } \alpha([j,g])=[l+1,g'], \ l<k-1 \\
					[l-1,g'] &\text{if  } \alpha([j,g])=[l,g'], \ l=k-1 \\
					\alpha[j,g] &\text{otherwise}
				\end{cases}
	\end{align*}
%
%
%
%
As a consequence we have that $H_l \circ \partial_0^{n,k} = \varphi_l$. Similar easy calculations show that 
	$H_l \circ \partial^{n,k}_1 = \idd$. We conclude that for $0\leq l\leq k-1$ the isovariant simplicial set map $H_l: \interval \Delta^{n,k} \lto \Delta^{n,k}$ is an isovariant homotopy from $\varphi_l$ to $\idd$:
	\begin{equation}\label{Equation:Homotpy H_l from phil to the identity}
		\begin{tikzcd}
			\Delta^{n,k} \arrow[dr,"\partial^{n,k}_0",swap] \arrow[drr, "\varphi_l", bend left=30] \\ &I^{n,k} \cong \interval \Delta^{n,k} \arrow[r,"H_l"] & \Delta^{n,k} \\ 
			\Delta^{n,k} \arrow[ru,"\partial^{n,k}_1"] \arrow[urr,"\idd",bend right=30,swap]
		\end{tikzcd}
	\end{equation}
	To finish the proof, observe that the image of $\varphi_l$ is contained in the $l$-horn $\Lambda^{n,k}_l$ and the image of the restriction of $H_l$ to $ \interval \Lambda_l^{n,k}$ is also contained in $\Lambda_l^{n,k}$: For the former assertion, suppose that $\alpha: [a]_b \lto [n]_k$ is an $(a,b)$-simplex of $\im(d_1^{l+1} \circ s_1^l)$ such that $\alpha$ factors as 
	\[ 	\begin{tikzcd}
		{[a]_b} \arrow[rr,"\alpha"] \arrow[dr,"\beta",swap] &&{[n]_k} \\ &{[n]_{k}} \arrow[ru,"d_1^{l+1} \circ s_1^l",swap] 
	\end{tikzcd}\]
	notice also that 
	\[d_1^{l+1}(s_1^l[j,g]) = \begin{cases}
		[j,g] \quad &\text{if  } j \neq l \\ 
		[j-1,g] \quad &\text{if  } j = l+1
	\end{cases}\]
	which implies that the image of $d_1^{l+1} \circ s_1^l \circ \beta$ is  contained $[n]_k \smallsetminus \{[l+1,e], [l+1,\sigma]\}$ and it contains the vertices $[l,e]$ and $[l,\sigma]$. In other words, $\alpha[a]_b \subseteq [n]_k \smallsetminus \{[l+1,e], [l+1,\sigma]\}$ and it contains $[l,e]$ and $[l,\sigma]$ which implies that $[n]_k \smallsetminus \{[l,e], [l,\sigma]\} \nsubseteq \alpha[a]_b $  so that $\alpha$ is an $(a,b)$-simplex of $\Lambda^{n,k}_l$ (see \Cref{Remark:simplices of isovariant horns}). Hence $\im(d_1^{l+1} \circ s_1^l) \subseteq \Lambda^{n,k}_l$. The second assertion follows a similar argument since any $(a,b)$-simplex $\alpha$ of the image of $H_l|_{\interval \Lambda_l^{n,k}}$ will factor as 
	\[ 	\begin{tikzcd}
		{[a]_b} \arrow[rr,"\alpha"] \arrow[dr,"\beta",swap] &&{[n]_k} \\ &{\thk[n]_{k}} \arrow[ru,"\widetilde{H}",swap] 
	\end{tikzcd}\]
	with the additional property that $\thk[n]_k \smallsetminus \{([l,e],\delta), ([l,\sigma], \delta)\} \nsubseteq \beta[a]_b$. Moreover this is also true for the case $l=k-1$. We conclude then that the restriction of the elementary homotopy $H_l$ from \Cref{Equation:Homotpy H_l from phil to the identity} is then an elementary homotopy from $\varphi_l \circ \iota$ to the identity $\idd_{\Lambda^{n,k}_l}$, where $\iota$ denotes the $l$-th horn inclusion into $\Delta^{n,k}$: 
	\[ 
	\begin{tikzcd}[column sep=huge]
		\Lambda^{n,k}_l \arrow[dr,"\partial^{n,k}_0",swap] \arrow[drr, "\varphi_l \circ \iota", bend left=30] \\ & \interval \Lambda^{n,k} \arrow[r,"H_l|_{\interval \Lambda^{n,k}}"] & \Lambda^{n,k}_l \\ 
		\Lambda^{n,k}_l \arrow[ru,"\partial^{n,k}_1"] \arrow[urr,"\idd",bend right=30,swap]
	\end{tikzcd}
	\]
	That is, $\varphi_l$ provides an isovariant left homotopic inverse of the inclusion $\iota$, a symmetric argument shows that $\varphi_l$ is also an isovariant right homotopic inverse to $\iota$. Notice that we have worked under the hypothesis that $0\leq l \leq k-1$. In the case where $k \leq l \leq n$ the proof follows the same idea with minor changes as one has to work with real face and real degeneracy maps of the form $d_0^i$ and $s_0^i$. \\ 
	Conversely, suppose now that the inclusion $\iota: \Lambda_l^{n,k} \lto \Delta^{n,k}$ is an isovariant elementary homotopy equivalence so that there exist an isovariant map $\psi : \Delta^{n,k} \lto \Lambda^{n,k}_l$ and elementary homotopies $H$ and $K$ from $\psi \circ \iota$ to $\idd_{\Lambda^{n,k}_l}$ and $\iota \circ \psi$ to $\idd_{\Delta^{n,k}}$ respectively. In particular
	\begin{gather}
		H \circ \partial_0 = H(-,0) = \psi \circ \iota \label{Eq:H0}\\ 
		H \circ \partial_1 = H(-,1) = \idd_{\Lambda^{n,k}_l} \label{Eq:H1}
	\end{gather}
	Suppose aditionally that $\Lambda^{n,k}_l$ is not admissible. Then by \Cref{Proposition:admissible horns} we have that $k=1$ and $l=0$ or $k=l=n$. In the first case we have then $\iota: \Lambda^{n,1}_0 \lto \Delta^{n,0}$ and recall from \Cref{Remark:Notationdeltank_in_terms_of_its_nondeg_simplices} that we write 
	\[\Delta^{n,1} = \left\langle v_0^{\nreal},v_1^\real,...,v_n^\real \right\rangle\]
	The isovariant simplicial set map $\psi$ must then send $v_0^{\nreal} \in \Delta^{n,1}$ to $v_0^{\nreal}$ in $\Lambda^{n,1}_0$ precisely because $\psi$ is isovariant. Moreover using the underlying partial order on the vertices of $\interval \Lambda^{n,1}_0$ we have that 
	\begin{equation}\label{Equation: homotopies psi inclusion and identity}
		H(v_i^{\real},0) \leq H(v_i^{\real},1) \quad \text{and} \quad H(v_0^{\nreal},0) \leq H(v_0^{\nreal},1)
	\end{equation}
	In particular 
	\[ (\psi \circ \iota) (v_1^{\real}) = H(v_1^{\real},0) \leq H(v_1^{\real},1) = v_1^{\real}\]
	and again by isovariance we must have that $\psi(v_1^{\real})=v_1^{\real}$. Let us then consider the set 
	\[ M= \left\{1<i\leq n \ / \ \psi(v_i^{\real}) \neq v_i^{\real} \right\}.\]
	Notice that $M$ cannot be empty, otherwise $\psi$ will send the whole $(n,1)$-simplex $\Delta^{n,1}$ to $\Lambda^{n,1}_0$ with the additional property that $\psi \circ \iota = \idd$ which is clearly not possible. Let $m = \min(M)$, by minimality, the fact that $\psi$ preserves the underlying order on vertices and \Cref{Equation: homotopies psi inclusion and identity} we have that $\psi(v^{\real}_{m}) = v_{m-1}^{\real}$ and clearly $m>1$. Consider the $(n-1,1)$-simplex $x = \left\langle v_0^{\nreal}, v_1^{\real},...,\widehat{v^{\real}_{m-1}},...,v_n^{\real}\right\rangle = \im(d_0^{m-1}:\Delta^{n-1,1} \lto \Delta^{n,1})$. Notice that $x$ clearly belongs to the horn $\Lambda^{n,1}_0$ and since the functor $\interval$ preserves monomorphisms, in particular it preserves inclusions so that we have that 
	\[\interval x \subseteq \interval \Lambda^{n,1}_0.\]
	More precisely, we have that 
	\[y=\interval x= \left\langle (v_0^{\nreal},0),(v_0^{\nreal},1),...,(\widehat{v_{m-1}^{\real}},0), (\widehat{v_{m-1}^{\real}},1),(v_m^{\real},0),(v_m^{\real},1),...,(v_n^{\real},0),(v_n^{\real},1) \right\rangle.\]
	Using \Cref{{Eq:H0},{Eq:H1}} we then have
	\begin{align*}
		H(y) &= \left\langle \psi(v_0^{\nreal}),v_0^{\nreal},\psi(v_1^{\real}), v_1^{\real},...,\psi(v_m^{\real}),v_m^{\real},...,\psi(v_n^{\real}), v_n^{\real}  \right\rangle
		\\ &=  \left\langle v_0^{\nreal},v_0^{\nreal},v_1^{\real},v_1^{\real},...,v_{m-1}^{\real},v_m^{\real},...,\psi(v_n^{\real}),v_n^{\real} \right\rangle
	\end{align*}
	where we have used the fact that $\psi(v_m^{\real})=v_{m-1}^{\real}$, $H(-,0)=\psi \circ \iota$ and $H(-,1)= \idd$. We conclude that $H(y) \cong \Delta^{n,1}$ which is a contradiction since the image of $H$ must be contained in $\Lambda^{n,1}_0$. The case where $k=l=n$ follows a similar idea. 
As a consequence the $l$-horn $\Lambda^{n,k}_l$ is admissible.  
\end{proof}


\subsection{Fiberwise Isovariant Homotopy}
We have used the cylinder defined in \Cref{Subsection:Exact cylinders} to define isovariant homotopies. Such a cylinder functor can also be used to define fiberwise homotopies. In other words, for any isovariant simplicial set $B$, the exact cylinder $\interval$ induces an exact cylinder functor for the over category $\ssetisov \downarrow B$. Let us first recall that there is a canonical forgetful functor 

\[
U_B: \ssetisov \downarrow B\simeq \Psh(G\Deltacat/B)\lto \ssetisov 
\] 
This functor is faithful and it preserves pullbacks and small colimits. In particular any morphism
\[ 
\begin{tikzcd}
	X && Y \\
	& B
	\arrow["f", from=1-1, to=1-3]
	\arrow["{p_X}"', from=1-1, to=2-2]
	\arrow["{p_Y}", from=1-3, to=2-2]
\end{tikzcd}
\]
in $\ssetisov\downarrow B$ is a monomorphism (resp. epimorphism) if and only if its image under $U_B$ is a monomorphism (resp. epimorphism), which amounts to say that $f$ is a monomorphism (resp. epimorphism). Moreover, commutative squares in $\ssetisov \downarrow B$ are pushouts (resp. pullback) squares if and only if its image by $U_B$ are pushout (resp. pullback) squares. Using this remark, the exact cylinder functor $\interval: \ssetisov \lto \ssetisov$ defines an exact cylinder functor \[ \interval_B: \ssetisov\downarrow B \lto \ssetisov \downarrow B \] where, if $X \overset{p_X}{\lto} B$ is an object in $\ssetisov \downarrow B$ then $\interval_B X$ is the cylinder defined by  

\[\begin{tikzcd}
	{X \coprod X} && {\interval_B X\coloneqq \interval X} && X \\
	&& {\interval B} \\
	&& B
	\arrow["{(p_X,p_X)}"', from=1-1, to=3-3]
	\arrow["{p_X}", from=1-5, to=3-3]
	\arrow["{(\partial_0^X,\partial_1^X)}", from=1-1, to=1-3]
	\arrow["{\rho^X}", from=1-3, to=1-5]
	\arrow["{\interval(p_X)}", from=1-3, to=2-3]
	\arrow["{\rho^B}", from=2-3, to=3-3]
\end{tikzcd}\]
The naturality of $\rho$ imply that for each $\varepsilon=0,1$:

\[\rho^B \circ \interval(p_X)\circ \partial_{\varepsilon}^X = \rho^B \circ \partial_{\varepsilon}^B \circ p_X = p_X. \]
As a consequence the two triangles in the diagram above commute.  

\begin{proposition}\label{Proposition:Exact cylinder of over category}
	For any $B$ isovariant simplicial set, the cylinder functor $\interval_B$ defined above is an exact cylinder functor for the over category $\ssetisov\downarrow B$.
\end{proposition}

\begin{proof}
	This follows immediately from the definition of $\interval_B$, the properties of the forgetful functor $U_B$ and the fact that $\interval$ is an exact cylinder functor for $\ssetisov$. 
\end{proof}

From now on $B$ will denote a fixed isovariant simplicial set. 

\begin{definition}\label{fiberwise isovariant homotopy}
	Let $X \overset{p_X}{\lto} B$ and $Y \overset{p_Y}{\lto} B$ be two isovariant simplicial sets over $B$ and $f,g: X \lto Y$ two isovariant simplicial set maps over $B$. A \emph{fiberwise elementary isovariant homotopy} from $f$ to $g$ consist of an isovariant simplicial set map $H$ over $B$: 
	\[\begin{tikzcd}
		{\interval X} && Y \\
		& B
		\arrow["H", from=1-1, to=1-3]
		\arrow["{p_{\interval X}}"', from=1-1, to=2-2]
		\arrow["{p_Y}", from=1-3, to=2-2]
	\end{tikzcd}\]
	such that the usual diagram
	
	\[ 
	\begin{tikzcd}
		X \arrow[dr,"\partial^X_0",swap] \arrow[drr, "f", bend left=30] \\ &\interval X \arrow[r,"H"] & Y \\ 
		X \arrow[ru,"\partial^X_1"] \arrow[urr,"g",bend right=30,swap]
	\end{tikzcd}
	\]
	commutes as a diagram in $\ssetisov \downarrow B$. 
\end{definition}
\begin{remark}
	Just as in the last section, the quotient of $\Hom_{\ssetisov\downarrow B}(X,Y)$ by the smallest equivalence relation generated by fiberwise elementary isovariant homotopies will be denoted by $[X,Y]_{{\text{isov}_B}}$, if the context is clear sometimes we will simply write $[X,Y]_B$. 
%
%
\end{remark}	

\subsection{Isovariant Anodyne Extensions}\label{Section:anodyne extensions}
So far we have constructed an exact cylinder functor $\interval$ for the category of isovariant simplicial sets $\ssetisov$ and an exact cylinder $\interval_B$ for the over category of isovariant simplicial sets over a fixed object $B$. In this section we will show that the class of admissible horn inclusions can be used to define a class of  $I$-anodyne extensions in the sense of \cite[Définition 1.3.10]{Cisinski2016LPCMDTH} for the over category $\ssetisov\downarrow B$ where $B$ is assumed to be normal.  We will start by recalling some definitions about saturated class of morphisms and a few technical lemmas that will be used in the proof. The main conclusion will be that $\ssetisov \downarrow B$ has a Cisinski's type model structure provided that $B$ is normal.
 This idea is inspired from the model structure for Dendroidal Sets constructed by Moerdijk and Cisinski in \cite{Cisinski2011}.

\begin{lemma}\label{lemma: Nondegenerate sections of interval I}
	Let $0 < l \leq n$ and $0<k \leq n$ such that $l\neq k,k+1$. Then any isovariant order-preserving map $\theta: [n+1]_l \lto \thk([n]_k)$ in $G\isovposets$ must have repetitions, \ie there are $x,y \in [n+1]_l$ such that $\theta(x)= \theta(y)$.
\end{lemma}
\begin{proof}
	Suppose first that $l<k$ and suppose that $\theta: [n+1]_l \lto \thk([n]_k)$ does not have repetitions, in other words, $\theta$ preserves the order in a strict sense so that we have $\theta(0,\varepsilon)<\theta(1,\varepsilon)<\theta(2,\varepsilon)< \cdots < \theta(l-1,\varepsilon)$ for $\varepsilon=0,1$ and, 
	\begin{equation}\label{eq:strict inequalities proof non degenerate sections} \theta(l)<\theta(l+1)< \cdots < \theta(n+1).
	\end{equation}
	This not possible for the longest strict path in the real part of $\thk([n]_k)$ must have $n-k+2$ vertices and since $l<k$ we have that $n-l+2>n-k+2$ so that in the chain of strict inequalities (\ref{eq:strict inequalities proof non degenerate sections}) there must be an equality which is then a contradiction. Similarly if $l>k+1$ it is not possible to have strict inequalities $\theta(0,\varepsilon)<\theta(1,\varepsilon)<\theta(2,\varepsilon)< \cdots < \theta(l-1,\varepsilon), \ \ \varepsilon=0,1$,  
	since the longest strict path of non-real points in $\thk([n]_k)$ must have length $k+1$ and since $l>k+1$ these last inequalities cannot be all strict.   	
\end{proof}

%

\begin{lemma}\label{lemma: Non-degenerate simplices of R_n,k}
	Let  $0<k<n+1$. The isovariant simplicial set $I^{n,k}$ has $n-k+1$ non-degenerate $(n+1,k)$-simplices and $k$ non-degenerate $(n+1,k+1)$-simplices (up to composition with $\sigma$). The $(n+1,k)$-simplices can be identified with strings of morphisms in $\thk([n]_k)$ having the form 
	\begin{figure}[H]
		\caption{{}}
		\label{fig:nondegsimplices_Ink.1}
		\begin{center}
			\begin{tikzpicture}[transform shape, align=center, 
				minimum size=3mm,
				>=stealth,
				bend angle=45,scale=0.8, every node/.style={scale=0.9}]
				\tikzset{node distance = 0.25cm and 0.25cm}
				
				\node (100)  {\footnotesize $([1,e],0)$};
				\node (000) [above left =of 100,yshift=0.2cm] {\footnotesize $([0,e],0)$} edge[->] (100); 
				\node (200) [below right =of 100,rotate=-45] {$\ldots$} ; 
				\node (300) [below right = of 200] {\footnotesize$([k-1,e],0)$}; 
				\node (40)  [below right = of 300] {\footnotesize$([k,e],0)$};
				\node (310) [below left = of 40] {\footnotesize$([k-1,\sigma],0)$};
				\node (210) [below left = of 310] {$\udots$};
				\node (110) [below left = of 210] {\footnotesize$([1,\sigma],0)$};
				\node (010) [below left = of 110] {\footnotesize$([0,\sigma],0)$};
				\node (60) [right = of 40] {\footnotesize$\cdots$};
				\node (70) [right = of 60] {\footnotesize$([j,e],0)$}; 
				\node (71) [above right = of 70] {\footnotesize$([j,e],1)$};
				\node (81) [right = of 71] {\footnotesize$\ldots$};
				\node (91) [right = of 81] {\footnotesize$([n,e],1)$};
				
				\path[->] (100) edge (200);
				\path[->] (200) edge (300);
				\path[->] (300) edge (40);
				
				\path[->] (40) edge (60);
				\path[->] (60) edge (70);				
				\path[red,->] (70) edge (71);
				\path[red,->] (71) edge (81);
				\path[red,->] (81) edge (91);
				
				\path[->] (010) edge (110);
				\path[->] (110) edge (210);				
				\path[->] (210) edge (310);
				\path[->] (310) edge (40);
			\end{tikzpicture}
		\end{center}
	\end{figure}
	\noindent Similarly the $k$ non degenerate $(n+1,k+1)$-simplices of $I^{n,k}$ can be identified with strings of morphisms of the form:  
	\begin{figure}[H]
		\caption{{}}
		\label{fig:nondegsimplices_Ink.2}
		\begin{center}
			\begin{tikzpicture}[transform shape, align=center, 
				minimum size=3mm,
				>=stealth,
				bend angle=45,scale=0.8, every node/.style={scale=0.9}]
				\tikzset{node distance = 0.25cm and 0.25cm}
				
				\node (100)  {\footnotesize $([1,e],0)$};
				\node (000) [above left =of 100] {\footnotesize $([0,e], 0)$} edge[->] (100); 
				\node (200) [below right =of 100] {$\ddots$} ; 
				\node (300) [below right = of 200] {\footnotesize$([j,e],0)$}; 
				
				\node (310) [below = of 300] {\footnotesize$([j,\sigma],0)$};
				\node (210) [below left =of 310,xshift=0.1cm]{$\udots$};
				\node (110) [below left = of 210] {\footnotesize$([1,\sigma],0)$};
				\node (010) [below left = of 110] {\footnotesize$([0,\sigma],0)$};
				
				\node (301) [above right =of 300,yshift=1cm] {\footnotesize$([j,e],1)$};
				\node (4301) [below right =of 301] {$\ddots$};
				
				\node (401) [below right =of 4301] {\footnotesize $([k,e],1)$}; 
				
				\node (51) [right =of 401] {\footnotesize $([k+1,e],1)$};
				\node (61) [right =of 51] {\footnotesize $\cdots$};
				\node (71) [right =of 61] {\footnotesize $([n,e],1)$};
				
				\node (311) [below right =of 310, yshift=-0.7cm] {\footnotesize$([j,\sigma],1)$}; 
				\node (4311) [below left = of 401,xshift=0.5cm] {$\udots$};
				
				\path[->] (100) edge (200);
				\path[->] (200) edge (300);
				\path[red,->] (300) edge (301);
				\path[red,->] (301) edge (4301);
				\path[red,->] (4301) edge (401);
				
				\path[->] (010) edge (110);
				\path[->] (110) edge (210);
				\path[->] (210) edge (310);
				\path[red,->] (310) edge (311);
				\path[red,->] (311) edge (4311);
				\path[red,->] (4311) edge (401);
				
				\path[red,->] (401) edge (51);
				\path[red,->] (51) edge (61);
				\path[red,->] (61) edge (71);
			\end{tikzpicture}
		\end{center}
	\end{figure}
\end{lemma}

\begin{proof}
	By direct application of Yoneda's lemma and the fact that $I^{n,k}= \Hom_{G\isovposets}(-,\thk([n]_k))$ it is easy to show that the non-degenerate simplices of $I^{n,k}$ are in correspondance with maximal paths without repetitions in $\thk([n]_k)$. The consequence of the lemma then follows from \Cref{lemma: Nondegenerate sections of interval I}.   	
\end{proof}

\begin{definition}\label{Definition:isovariant anodyne extension}
	A morphism of isovariant simplicial sets $f:X \lto Y$ will be called \emph{isovariant anodyne extension} or $\interval$-anodyne extension if it belongs to the smallest saturated class of morphisms containing the set of isovariant admissible horn inclusions. Moreover, $f$ will be called \emph{isovariant fibration} if it has the right lifting property with respect to the class of isovariant anodyne extensions. \\
	An isovariant simplicial set $X$ is \emph{fibrant} (also called $\interval$-fibrant) if the canonical morphism from $X$ to the terminal isovariant simplicial set is an isovariant fibration.
\end{definition}

We will use the last two lemmas to state and proof an adapted version of \cite[Theorem 2.1.1]{gabriel2012calculus}. Consider the following three classes of morphisms in $\ssetisov$: 
\begin{enumerate}
	\item[$\bullet$] $\catcal{A} \coloneqq$ the set of admissible horn inclusions 
	\[  \Lambda^{n,k}_l \lto \Delta^{n,k}, \quad  n>0, 0\leq k\leq n+1, 0\leq l \leq n\] 
	
	\item[$\bullet$] $\catcal{B} \coloneqq$ the set of all inclusions 
	\[ \interval \partial \Delta^{n,k} \cup \{\varepsilon\}(\Delta^{n,k}) \lto \interval \Delta^{n,k}, \quad n\geq 0, \varepsilon=0,1, 0\leq k \leq n+1\] 
	\item[$\bullet$] $\catcal{C} \coloneqq$ the set of all inclusions 
	\[ \interval K \cup \{\varepsilon\}(L) \lto \interval L, \quad \varepsilon=0,1 \text{  and  } K \linj L \text{ is a normal inclusion}.\]
\end{enumerate}

\begin{lemma}\label{lemma: saturation of admissible horn inclusions is the same as saturation of cisinski's class}
	The classes $\catcal{A}$, $\catcal{B}$ and $\catcal{C}$ have the same saturation. More precisely  if we denote by $M_{\catcal{A}}$ the smallest saturated class containing $\catcal{A}$ and similarly $M_{\catcal{B}}$ and $M_{\catcal{C}}$ denote the smallest saturated classes of morphisms containing $\catcal{B}$ and $\catcal{C}$ respectively. Then 
	\[ M_{\catcal{A}} = M_{\catcal{B}} = M_{\catcal{C}}. \]
	Notice that here, the class $M_{\catcal{A}}$ is precisely the class of isovariant anodyne extensions.
\end{lemma}

\begin{proof} We will proceed in several steps. In most instances of the proof we will consider the case when $\varepsilon=1$,  the case $\varepsilon=0$ is analogous. 
	\begin{enumerate}[label={\emph{Step }\arabic*:}, wide, labelwidth=!, labelindent=0pt]
		\item Let us first show that $\catcal{B}\subseteq M_{\catcal{A}}$. This will then imply that $M_{\catcal{B}}\subseteq M_{\catcal{A}}$. The cases $k=0$ and $k=n+1$ follow directly from \cite[2.1.1]{gabriel2012calculus}.
		Let us then fix $0<k<n+1$ and consider the inclusion 
		\[ \varphi: \interval\partial\Delta^{n,k} \cup \left\{1\right\}\Delta^{n,k} \lto \interval\Delta^{n,k}. \]
		For $0\leq i \leq k-1$ we let $C_i$ be the image of the non-degenerate $(n+1,k+1)$-simplex $\gamma_i: \Delta^{n+1,k+1} \lto \interval\Delta^{n,k}$,  corresponding to the string of morphisms in $\thk([n]_k)$ given by \Cref{fig:nondegsimplices_Ink.2}.
		Let us denote such $C_i$ by
		\[C_i \coloneqq \left\langle (v_0^{\nreal},0),(v_1^{\nreal},0),...,(v_i^{\nreal},0), (v_i^{\nreal},1),...,(v_k^{\real},1),...,(v_n^{\real},1) \right\rangle \] 
		Similarly for $k\leq i \leq n$ we let $D_i$ to be the image of the non-degenerate $(n+1,k)$-simplex  $\delta_i: \Delta^{n+1,k} \lto \interval\Delta^{n,k}$ corresponding to the other string of morphisms in $\thk([n]_k)$ given in \Cref{fig:nondegsimplices_Ink.1}. Again we denote it as
		\[ D_i \coloneqq \left\langle (v_0^{\nreal},0),(v_1^{\nreal},0),...,(v_{k-1}^{\nreal},0), (v_k^{\real},0),...,(v_i^{\real},0),(v_i^{\real},1),...,(v_n^{\real},1) \right\rangle\]
		We then define 
		\[ F_i \coloneqq \begin{cases}
			C_i &\text{if  } 0\leq i \leq k-1 \\ D_i &\text{if  } k\leq i \leq n
		\end{cases}\]
		and $E_{-1}= \interval\partialbf\Delta^{n,k} \cup \left\{1\right\}\Delta^{n,k}$ and, for any $0\leq i \leq n$ we let 
		\[ E_i \coloneqq E_{i-1}\cup F_i.
		\]			
		For instance in the case where $n=k=2$, the object $E_0 = \interval\partial\Delta^{2,2} \cup \left\{1\right\}\Delta^{2,2} \cup C_0$ has a geometrical realization as in the \Cref{Figure:Geometrical realization of E0} below.
		
		\begin{figure}[H]
			\caption{Geometric realization of $E_0=\interval\partial\Delta^{2,2} \cup \left\{1\right\}\Delta^{2,2} \cup C_0$} \label{Figure:Geometrical realization of E0}
			\begin{center}
				\begin{tikzpicture}[line join = round, line cap = round,baseline=1ex,scale=0.9, transform shape, font=\scriptsize]
					\coordinate[label=left:{$([0,e],0)$}]  (000) at (0,2,1.5);
					\coordinate[label=above:{$([1,e],0)$}]  (100) at (1,2.5,0);
					\coordinate[label=above left:{$([0,e],1)$}]  (001) at (7,2,1.5);
					\coordinate[label=above:{$([1,e],1)$}]  (101) at (8,2.5,0);
					\coordinate[label=left:{$([2,e],0)$}]  (20) at (0,0,0);
					\coordinate[label=right:{$([2,e],1)$}]  (21) at (7,0,0);
					\coordinate[label=left:{$([0,\sigma],0)$}]  (010) at (-1,-2.5,0);
					\coordinate[label=above right:{$([1,\sigma],0)$}]  (110) at (0,-2,-1.5);
					\coordinate[label=right:{$([0,\sigma],1)$}]  (011) at (6,-2.5,0);
					\coordinate[label=above right:{$([1,\sigma],1)$}]  (111) at (7,-2,-1.5);
					
					\begin{scope}[decoration= {markings, mark=at position 0.5 with {\arrow{stealth}}}, xshift=0cm]
						\draw[postaction=decorate] (000)--(100);
						\draw[postaction=decorate] (100)--(101);
						\draw[postaction=decorate] (000)--(001);
						\draw[postaction=decorate] (001)--(101);
						\draw[postaction=decorate] (000)--(20);
						\draw[postaction=decorate] (001)--(21);
						\draw[postaction=decorate,lightgray] (100)--(20);
						\draw[postaction=decorate] (101)--(21);
						\draw[postaction=decorate,cyan] (20)--(21);
						
						\draw[postaction=decorate] (010)--(20);
						\draw[postaction=decorate,lightgray] (110)--(20);
						\draw[postaction=decorate] (010)--(110);
						
						\draw[postaction=decorate] (011)--(21);
						\draw[postaction=decorate,lightgray] (111)--(21);
						\draw[postaction=decorate] (011)--(111);
						
						\filldraw[fill=lightgray,fill opacity=0.1,dashed,draw=lightgray] (010)--(011)--(111)--cycle;
						
						\draw[postaction=decorate] (010)--(011);
						\draw[postaction=decorate,lightgray] (110)--(111);
					\end{scope}
					
					\filldraw[fill=lightgray,fill opacity=0.6] (001)--(21)--(101)--cycle;
					\filldraw[fill=lightgray,fill opacity=0.6] (011)--(21)--(111)--cycle;
					
					\fill (000) circle[radius=2pt];
					\fill (100) circle[radius=2pt];
					\fill (001) circle[radius=2pt];
					\fill (101) circle[radius=2pt];
					
					\fill (110) circle[radius=2pt];
					\fill (010) circle[radius=2pt];
					\fill (011) circle[radius=2pt];
					\fill (111) circle[radius=2pt];
					
					\filldraw[fill=lightgray,fill opacity=0.3] (000)--(21)--(001)--cycle;
					\filldraw[fill=lightgray,fill opacity=0.1] (000)--(101)--(001)--cycle;
					
					\filldraw[fill=lightgray,fill opacity=0.3] (010)--(21)--(011)--cycle;
					
					\fill[cyan] (20) circle[radius=2pt];
					\fill[cyan] (21) circle[radius=2pt];
					
				\end{tikzpicture}	
			\end{center}
		\end{figure}
		\noindent
		Notice that $E_n= \interval\Delta^{n,k}$. 
		In particular $\varphi$ can be then factored as a finite composition: 
		\[ \varphi: \interval\partial \Delta^{n,k} \linj E_0 \linj E_1 \linj \cdots \linj E_n = \interval\Delta^{n,k}\]
		it is then enough to show that each of the inclusions $E_{i-1}\linj E_i$ belongs to $M_{\catcal{A}}$ for this last class is closed under transfinite compositions which it will then imply that $\varphi \in M_{\catcal{A}}$. Moreover, by construction of the $E_i$'s, for any $0\leq i\leq n$ we have pushout diagrams
		
		\begin{equation}\label{eq: Pushout of E_i's}
			\begin{tikzcd}
				E_{i-1} \cap F_i \arrow[r,hookrightarrow] \arrow[d,hookrightarrow] &E_{i-1} \arrow[d] \\ F_i \arrow[r] &E_i
			\end{tikzcd}
		\end{equation}
		and since $M_{\catcal{A}}$ is also closed under pushouts, if $E_{i-1}\cap F_i \hookrightarrow F_i$ is in $M_{\catcal{A}}$ then so are the inclusions $E_{i-1}\lto E_i$. We will then show that for any $0\leq i \leq n$ the inclusions $E_{i-1}\cap F_i \hookrightarrow F_i$ are in $M_{\catcal{A}}$. Let us first consider the case when $0\leq i\leq k-1$ so that the squares $\ref{eq: Pushout of E_i's}$ are of the form: 
		
		\begin{equation}\label{eq: second Pushout of E_i's}
			\begin{tikzcd}
				E_{i-1} \cap C_i \arrow[r,hookrightarrow] \arrow[d,hookrightarrow] &E_{i-1} \arrow[d] \\ C_i \arrow[r] &E_i
			\end{tikzcd}
		\end{equation}
		Notice that we have commutative diagrams: 
		
		\begin{equation}\label{Eq:Diagram1and2Step1}
			\begin{tikzcd}
				\Delta^{n,k} \arrow[r,"d_1^0"] \arrow[d,"\simeq",swap] &\Delta^{n+1,k+1} \arrow[d,"\gamma_0"] \\
				\left\{1\right\}\Delta^{n,k} \arrow[r,hookrightarrow] & \interval \Delta^{n,k}
			\end{tikzcd} \ \  
			\begin{tikzcd}
				\Delta^{n,k} \arrow[r,"d_1^{k-1}"] \arrow[d,"\simeq",swap] &\Delta^{n+1,k+1} \arrow[d,"\gamma_{k-1}"] \\
				\left\{1\right\}\Delta^{n,k} \arrow[r,hookrightarrow] & \interval \Delta^{n,k}
			\end{tikzcd}
		\end{equation}
		%
		\begin{equation}\label{Eq:Diagram3Step1}
			\begin{tikzcd}
				\Delta^{n,k} \arrow[r,"d_1^j"] \arrow[d,"\gamma_{i-1}",swap] &\Delta^{n+1,k+1} \arrow[d,"\gamma_i"] \\
				\interval\Delta^{n-1,k-1} \arrow[r,"\interval(d_1^j)",swap] & \interval \Delta^{n,k}
			\end{tikzcd} \quad  0<j<i\leq k-1  
		\end{equation}
		
		\begin{equation}\label{Eq:Diagram4Step1}
			\begin{tikzcd}
				\Delta^{n,k} \arrow[r,"d_1^j"] \arrow[d,"\gamma_{i}",swap] &\Delta^{n+1,k+1} \arrow[d,"\gamma_i"] \\
				\interval\Delta^{n-1,k-1} \arrow[r,"\interval(d_1^{j-1})",swap] & \interval \Delta^{n,k}
			\end{tikzcd}   \quad  0<i+1<j<k+1
		\end{equation}
		
		\begin{equation}\label{Eq:Diagram5Step1}
			\begin{tikzcd}
				\Delta^{n,k} \arrow[r,"d_1^i"] \arrow[d,"d_1^i",swap] &\Delta^{n+1,k+1} \arrow[d,"\gamma_i"] \\
				\Delta^{n+1,k+1} \arrow[r,"\gamma_{i-1}",swap] & \interval \Delta^{n,k}
			\end{tikzcd}   \quad  0<i < k-1
		\end{equation}		
		We claim that that for any face map $d_1^j: \Delta^{n,k} \lto \Delta^{n+1,k+1}$ with $j\neq i+1$ the image of $\gamma_i \circ d_1^j$ is in $E_{i-1}$. Indeed if $j<i$ then diagram \ref{Eq:Diagram3Step1} implies that 
		\begin{align*}\im(\gamma_i \circ d_1^j) &= \im(\interval(d_1^j) \circ \gamma_{i-1}) \\&= \left\langle (v_0^{\nreal},0),(v_1^{\nreal},0),...,\widehat{(v_j^{\nreal},0)},(v_{j+1}^{\nreal},0),...,(v_i^{\nreal},0),(v_i^{\nreal},1),...,(v_n^{\real},1) \right\rangle \\ &\subseteq \interval \partialbf \Delta^{n,k} \subseteq E_{-1} \subseteq E_{i-1}
		\end{align*}
		Similarly if $j=i$ the diagrams \ref{Eq:Diagram1and2Step1} and \ref{Eq:Diagram5Step1} also justify the claim. In the latter case, \ie, for $0<i=j<k-1$ we then have that 
		\begin{align*}
			\im(\gamma_i \circ d_1^i) &= \im(\gamma_{i-1} \circ d_1^{i}) \\&= \left\langle (v_0^{\nreal},0),(v_1^{\nreal},0),...,{(v_{i-1}^{\nreal},0)},\widehat{(v_i^{\nreal},0)}, ,(v_{i}^{\nreal},1),...,(v_{k-1}^{\nreal},1),(v_k^{\real},1),...,(v_n^{\real},1) \right\rangle \\ &\subseteq \interval \partialbf \Delta^{n,k} \subseteq E_{-1} \subseteq E_{i-1}
		\end{align*}
		An analogous calculation using diagram \ref{Eq:Diagram4Step1} justifies the claim in the case where $j>i+1$. Finally notice that the image of $\gamma_i \circ d_1^{i+1}$ cannot be in $\interval \partial^{n,k}$ because 
		\[ \im(\gamma_i \circ d_1^{i+1}) = \left\langle (v_0^{\nreal},0),(v_1^{\nreal},0),...,(v_i^{\nreal},0),\widehat{(v_i^{\nreal},1)},(v_{i+1}^{\nreal},1),...,(v_k^{\real},1),...,(v_n^{\real},1) \right\rangle \]
		so that $\im(\gamma_{i+1} \circ d_1^{i+1})$ projects onto $\left\langle v_0^{\nreal},v_1^{\nreal},...,v_{k-1}^{\nreal},v_k^{\real},...,v_n^{\real} \right\rangle \simeq \Delta^{n,k}$ under the projection $\rho: \interval\Delta^{n,k} \lto \Delta^{n,k}$. Moreover $\gamma_i \circ d_1^{i+1}: \Delta^{n,k} \lto \interval\Delta^{n,k}$ cannot be a face of any $C_j \subseteq \interval\Delta^{n,k}$ with $j=0,1,...,i-1$ since $\im(\gamma_i \circ d_1^{i+1})$ contains the vertex $(v_i^{\nreal},0)$. We conclude then that $\im(\gamma_i \circ d_1^{i+1})$ is also not contained in $E_{i-1}$. Thus, the inclusion $E_{i-1} \cap C_i \lto C_i$ is of the form $\Lambda^{n+1,k+1}_{i+1} \lto \Delta^{n+1,k+1} \simeq C_i$. Notice that this last horn inclusion is always admissible since if $k=n$ then $i+1$ is at most $n$ and in such case $\Lambda^{n+1,n+1}_n$ is always admissible by \Cref{Proposition:admissible horns}, in any other case the horn $\Lambda^{n+1,k+1}_{i+1}$ is still admissible. 
		
		The proof that $E_{i-1}\cap F_i \linj F_i$ is also equivalent to an admissible horn inclusion when $k\leq i \leq n$ follows essentially the same idea as in the case $0\leq i \leq k-1$. We conclude that such inclusions $E_{i-1}\cap F_i \linj F_i$ are always in $M_{\Acal}$ which implies that $\catcal{B}\subseteq M_{\catcal{A}}$. 
		
		\item Let us now show that $\Acal \subseteq M_{\catcal{B}}$. As before the cases where $k=0$ and $k=n+1$ are just direct applications of the classical argument \cite[\nopp II 2.1.1]{gabriel2012calculus} so let us fix $\iota: \Lambda^{n,k}_l \linj \Delta^{n,k}$ an admissible horn inclusion with $0<k<n+1$. We will prove that $\iota$ is a retract of some map in $\catcal{B}$. The closure under retracts will then imply that $\Acal \subseteq M_{\catcal{B}}$ and thus by saturation $M_{\Acal}  \subseteq M_{\catcal{B}}$. Suppose first that $0\leq l\leq k-1$. The fact that $\iota$ is an admissible horn inclusion implies that if $l=0$ then $k$ must be different than $1$. We will then consider several cases, the proofs in each case will follow essentially the same idea: 
		\begin{enumerate}[label=\textbf{(\alph*)}, ref=(\alph*)]
			\item \label{Item:Step2:a} Let us then consider first the case where $l>0$.  Define $q: [n]_k \lto \thk[n]_k$ by sending
			\[q([i,g]) = 
			\begin{cases}
				([i,g],0) &\text{ if } 0\leq i <l \\
				([i,g],1) &\text{ if } i \geq l
			\end{cases} \]
			and $r_l: \thk[n]_k \lto [n_k]$ by 
			\[ r_l([i,g],\delta) = 
			\begin{cases}
				[i,g] &\text{ if } 0\leq i\leq l \text{  and  } \delta=0 \ \text{or} \ k\leq i \leq n \ \text{and} \ \delta=0 \\
				[l,g] &\text{ if } l < i \leq k-1 \text{  and   } \delta=0 \\ 
				[l,g] &\text{ if } i\leq l \text{  and  } \delta=1 \\ 
				[i,g] &\text{ if } l<i \text{  and  } \delta=1
			\end{cases}\]
			Then notice that 
			\[ (r_l \circ q )[i,g] = 
			\begin{cases}
				r_l([i,g],0) &\text{if } 0\leq i<l \\
				r_l([i,g],1) &\text{if } i>l \\ 
				r_l([l,g],1) &\text{if } i=l
			\end{cases} = [i,g]\]
			So that $r_l \circ q $ is precisely the identity $\idd_{[n]_k}$. These two maps are clearly isovariant and order preserving and they induce isovariant simplicial set maps $r_l:\interval\Delta^{n,k} \lto \Delta^{n,k}$ and $q:\Delta^{n,k} \lto \interval\Delta^{n,k}$ such that $r_l \circ q = \idd_{\Delta^{n,k}}$. We claim that the image of $\Lambda^{n,k}_l \subseteq \Delta^{n,k}$ by $q$ must be in the union $\interval(\partialbf \Delta^{n,k})\cup \partialbf \interval(\Delta^{n,k})$. Indeed, we recall that 
			\[\Lambda^{n,k}_l = \partialbf\Delta^{n,k}\smallsetminus \left\langle v_0^c,v_1^c,...,\widehat{v_l^c},...,v_{k-1}^c,v_k^r,...,v_n^r\right\rangle\]
			and notice that if $j<l$, then
			\begin{align*}
				q(\partialbf_1^j\Delta^{n,k}) &= q\left(\left\langle v_0^c,...,\widehat{v_j^c},...,v_l^c,...,v_{k-1}^c,v_k^r,...,v_n^r \right\rangle\right) \\
				&= \left\langle (v_0^c,0),...,\widehat{(v_j^c,0)},...,(v_{l-1}^c,0),(v_l^c,1),...,(v_k^r,1),...,(v_n^r,1) \right\rangle \\
				&\subseteq \interval(\partialbf \Delta^{n,k})\cup \partialbf \interval(\Delta^{n,k})
			\end{align*}
			similarly, if $l<j\leq k-1$ we have that
			\begin{align*}
				q(\partialbf_1^j\Delta^{n,k}) &= q\left(\left\langle v_0^c,...,{v_l^c},...,\widehat{v_j^c},...,v_{k-1}^c,v_k^r,...,v_n^r \right\rangle\right) \\
				&= \left\langle (v_0^c,0),...,{(v_{l-1}^c,0)},(v_l^c,1)...,\widehat{(v_j^c,0)},...,(v_k^r,1),...,(v_k^r,1),...,(v_n^r,1) \right\rangle \\
				&\subseteq \interval(\partialbf \Delta^{n,k})\cup \partialbf \interval(\Delta^{n,k})
			\end{align*}
			Similar calculations show that $q(\partialbf_0^j\Delta^{n,k}) \subseteq \interval(\partialbf \Delta^{n,k})\cup \partialbf \interval(\Delta^{n,k})$ for $k\leq j\leq n$. Furthermore, none of the faces of $\Lambda^{n,k}_l$ is sent into $\{0\}\Delta^{n,k}$ as $l>0$. Each of these faces are actually contained in $\interval \Lambda^{n,k}_l\cup \{1\}\Delta^{n,k}$ since $(v_l^c,1)$ is reached by each face of $\Lambda^{n,k}_l$ under the map $q$. In other words,
			\[ \im\left(q|_{\Lambda^{n,k}_l}\right)\subseteq \interval \Lambda^{n,k}_l\cup \{1\}\Delta^{n,k}. \]
			On the other hand, the image of $\interval \Lambda^{n,k}_l\cup \{1\}\Delta^{n,k}$ by $r_l$ fits into $\Lambda^{n,k}_l$ since as $l>0$ we have that: 
			
			\begin{align*}
				r_l(\{1\}\Delta^{n,k}) 
				&= r_l\left(\left\langle (v_0^c,1),(v_1^c,1),...,(v_l^c,1),...,(v_k^r,1),...,(v_n^r,1) \right\rangle\right) \\
				&\subseteq \im(d_1^0:\Delta^{n-1,k-1} \lto \Delta^{n,k}) \\
				&=\partialbf_1^0(\Delta^{n,k})\subseteq \Lambda_l^{n,k}. 
			\end{align*}
			Notice that in the last two lines we have used the fact that $l>0$. Moreover $r_l$ also sends $\interval\Lambda^{n,k}_l$ into $\Lambda^{n,k}_l$ since for any $u: \Delta^{m,l}\lto \interval\Delta^{n,k}$ injective such that $(v_l^c,0)$ is contained in the image of $u$ then, by definition of $r_l$, the vertex $v_l^c$ is contained in the image of the composite $r_l\circ u$. Furthermore, this last composition is not surjective so that $r_l$ sends each simplex of $\interval \Lambda^{n,k}_l$ into $\Lambda^{n,k}_l$ by \ref{Remark:simplices of isovariant horns}. Therefore, we have a commutative diagram:
			
			\begin{equation}\label{Equation:diagram retract}
				\begin{tikzcd}
					{\Lambda^{n,k}_l} && {\interval\Lambda^{n,k}_l\cup \{1\}\Delta^{n,k}} && {\Lambda^{n,k}_l} \\
					\\
					{\Delta^{n,k}} && {\interval\Delta^{n,k}} && {\Delta^{n,k}}
					\arrow["{\interval{1}_{\Delta^{n,k}}}", bend right=24pt, from=3-1, to=3-5,swap]
					\arrow["{\interval{1}_{\Lambda^{n,k}_l}}", bend left=24pt, from=1-1, to=1-5]
					\arrow["q", from=3-1, to=3-3]
					\arrow["{r_l}", from=3-3, to=3-5]
					\arrow["{r_l}"', from=1-3, to=1-5]
					\arrow["\iota"', hook, from=1-1, to=3-1]
					\arrow[hook, from=1-3, to=3-3]
					\arrow["\iota", hook, from=1-5, to=3-5]
					\arrow["q"', from=1-1, to=1-3]
			\end{tikzcd}\end{equation}
			As a consequence, $\iota$ is a retract of a morphism in $M_{\catcal{B}}$ and thus it belongs to such class of morphisms.
			
			\item \label{Item:Step2:b} Let us now consider the case where $k \leq l <n$. We will consider the cases $l=0$ and $l=n$ separately. The idea is more or less similar except that this time we consider $r_l: \thk[n]_k \lto [n_k]$ defined by 
			\[ r_l([i,g],\delta) = 
			\begin{cases}
				[i,g] &\text{ if } 0\leq i< l \text{  and  } \delta=0  \\
				[l,g] &\text{ if } l \leq i \leq n \text{  and   } \delta=0 \\ 
				[i,g] &\text{ if } 0\leq i\leq k-1 \text{  and  } \delta=1 \text{  or  } l<i\leq n \text{  and  } \delta=1 \\ 
				[l,g] &\text{ if } k\leq i\leq l \text{  and  } \delta=1
			\end{cases}\]
			and $q: [n]_k \lto \thk[n]_k$ is given by
			\[q([i,g]) = 
			\begin{cases}
				([i,g],0) &\text{ if } 0\leq i \leq l-1 \\
				([i,g],1) &\text{ if } l\leq i \leq n
			\end{cases} \]
			
			Again, by definition we have that $r_l \circ q = \idd_{[n]_k}$ and these two maps induce isovariant simplicial set maps $r_l: \interval\Delta^{n,k}\lto \Delta^{n,k}$ and $q: \Delta^{n,k}\lto \interval \Delta^{n,k}$ such that $r_l\circ q = \idd_{\Delta^{n,k}}$. The rest of the proof follows in a similar way as above.
			\item \label{Item:Step2:c} Let us now consider the last two remaining cases, \ie when $l=0$ and $l=n$. We will only prove the former as the argument is still analogous to that in \ref{Item:Step2:a}. Notice that since the horn is admissible $k$ must be different than $1$. In this case the fixed horn inclusion $\iota: \Lambda^{n,k}_0$ will be a retract of a morphism in the class $\catcal{B}$ with corresponding $\varepsilon=0$.  We can then define $q:[n]_k \lto \thk[n]_k$ by
			\[ q([i,g]) = \begin{cases}
				([0,g],0) &\text{  if  } i=0 \\
				([i,g],1) &\text{  else  }
			\end{cases}
			\]
			and $r_0: \thk[n]_k \lto [n]_k$ by 
			
			\[
			r_0([i,g],\delta)= \begin{cases}
				[0,g] &\text{  if  } 0\leq i\leq k-1 \text{  and  } \delta=0 \\ 
				[i,g] &\text{  else  }
			\end{cases}
			\]
			Then it is clear that $r_0 \circ q = \idd_{[n]_k}$. Just as in \ref{Item:Step2:a} and \ref{Item:Step2:b} these two maps induce isovariant simplicial set maps $q:\Delta^{n,k}\lto \interval \Delta^{n,k}$ and $r_0:\interval \Delta^{n,k}\lto \Delta^{n,k}$ such that $r_0 \circ q = \idd_{\Delta^{n,k}}$ and, by definition of $q$ we have that for any $0<j\leq k-1$ we have that 
			\begin{align*}
				q(\partialbf_1^j\Delta^{n,k}) &= q\left(\left\langle v_0^c,...,\widehat{v_j^c},...,v_{k-1}^c,v_k^r,...,v_n^r \right\rangle\right) \\
				&= \left\langle (v_0^c,1), (v_1^c,1),...,\widehat{(v_j^c,1)},(v_{k-1}^c,1),(v_k^r,1),...,(v_n^r,1) \right\rangle \\
				&\subseteq \interval(\partialbf \Delta^{n,k})\cup \{0\}(\Delta^{n,k})
			\end{align*}
			and similarly for any $k<j\leq n$ we also have that $q(\partialbf_0^j  \Delta^{n,k}) \subseteq \interval(\partialbf \Delta^{n,k})\cup \{0\}(\Delta^{n,k})$. This implies that 
			\[ \im\left(q|_{\Lambda^{n,k}_0}\right)\subseteq \interval \Lambda^{n,k}_0\cup \{0\}\Delta^{n,k}. \]
			On the other hand, 
			\begin{align*}
				r_0(\{0\}\Delta^{n,k}) 
				&= r_0\left(\left\langle (v_0^c,0),(v_1^c,0),...,(v_k^c,0),...,(v_n^r,0) \right\rangle\right) \\ 
				&= \left\langle v_0^c,v_0^c,...,v_k^r,...,v_n^r \right\rangle \\
				&\subseteq \im(d_1^1:\Delta^{n-1,k-1} \lto \Delta^{n,k}) \\
				&=\partialbf_1^1(\Delta^{n,k})\subseteq \Lambda_0^{n,k}. 
			\end{align*}
			Again, $r_0$ sends $\interval\Lambda_0^{n,k}$ into $\Lambda^{n,k}_0$ by the same reasons as in case \ref{Item:Step2:a}. Besides, the rest of the proof also follows as in \ref{Item:Step2:a} with the diagram \ref{Equation:diagram retract} having $l=0$ and $\varepsilon=0$ instead of $\varepsilon=1$.
		\end{enumerate}
		We conclude that $M_{\catcal{A}} \subseteq M_{\catcal{B}}$. 	       
		\item Let us now show that $\catcal{C}\subseteq M_{\catcal{B}}$. Let $K \lto L$ be a normal inclusion. Recall that from \cite[\nopp 8.1.32]{Cisinski2016LPCMDTH} the skeleton functors come with a filtration 
		\[\begin{tikzcd}
			{\emptyset = \text{Sk}_{-1}} & {\text{Sk}_1} & {\text{Sk}_2} & \cdots & {\text{Sk}_n} & {\text{Sk}_{n+1}} & \cdots
			\arrow[hook, from=1-1, to=1-2]
			\arrow[hook, from=1-2, to=1-3]
			\arrow[hook, from=1-3, to=1-4]
			\arrow[hook, from=1-4, to=1-5]
			\arrow[hook, from=1-5, to=1-6]
			\arrow[hook, from=1-6, to=1-7]
		\end{tikzcd}\]
		such that for any isovariant simplicial set $L$ we have that 
		\[ L \simeq \colim_{n}\sk_n(L) = \bigcup_{n\geq -1}\sk_n(L).\]
		Since $\interval$ is an exact cylinder functor and in particular it commutes with colimits we have canonical isomorphisms
		\[
		\begin{split} 
			\interval K \cup \left\{\varepsilon\right\}(L) \simeq \interval \left( K \cup \sk_{-1}(L) \right)\cup \{\varepsilon\}(L) \\ 
			\interval L \simeq  \bigcup_{n \geq -1} \left(\interval(K \cup \sk_n(L)) \cup \{\varepsilon\}(L) \right)			
		\end{split}
		\]
		Thus, to show that the map $\interval K \cup \{\varepsilon\}(L) \lto \interval L$ belongs to $M_{\catcal{B}}$ it suffices to show that for any $n\geq -1$ the maps
		\[\varphi_n: \interval (K \cup \sk_n(L)) \cup \{\varepsilon\}(L) \lto \interval (K \cup \sk_{n+1}(L)) \cup \{\varepsilon\}(L)\]
		are in $M_{\catcal{B}}$. Let $\Sigma^{n,k}$ denote the set of isomorphism classes of non-degenerate $(n,k)$-simplices of $L$ that are not contained in $K$ and $\Sigma^n = \coprod_{0\leq k \leq n+1}\Sigma^{n,k}$. By \cite[Lemme 8.1.34.]{Cisinski2016LPCMDTH} we have a pushout square: 
		\[\begin{tikzcd}[row sep=small,column sep=small]
			{\displaystyle \coprod_{\Sigma^{n+1}}(\partial\Delta^{n+1,k})_x} && {K \cup \text{Sk}_n(L)} \\
			\\
			{\displaystyle \coprod_{\Sigma^{n+1}}(\Delta^{n+1,k})_x} && {K \cup \text{Sk}_{n+1}(L)}
			\arrow[from=1-1, to=3-1]
			\arrow[from=1-3, to=3-3]
			\arrow["{x\circ -}", from=3-1, to=3-3]
			\arrow[from=1-1, to=1-3]
			\arrow["\lrcorner"{anchor=center, pos=0.125, rotate=180}, draw=none, from=3-3, to=1-1]
		\end{tikzcd}\]
		Since $\interval$ preserves small colimits from the last pushout diagram we get a new pushout square: 	 
		\[\begin{tikzcd}[row sep=small,column sep=small]
			{\displaystyle \coprod_{\Sigma^{n+1}} \interval (\partial\Delta^{n+1,k})_x} && {\interval\left(K \cup \text{Sk}_n(L)\right)} \\
			\\
			{\displaystyle \coprod_{\Sigma^{n+1}}\interval (\Delta^{n+1,k})_x} && {\interval\left(K \cup \text{Sk}_{n+1}(L)\right)}
			\arrow[from=1-1, to=3-1]
			\arrow[from=1-3, to=3-3]
			\arrow[from=3-1, to=3-3]
			\arrow[from=1-1, to=1-3]
			\arrow["\lrcorner"{anchor=center, pos=0.125, rotate=180}, draw=none, from=3-3, to=1-1]
		\end{tikzcd}\]
		Furthermore the left and right vertical arrows both factorize respectively through 
		\[
		\coprod \interval(\partial\Delta^{n+1,k})_x \cup \{\varepsilon\}(\Delta^{n+1,k})_x \text{  and  } \interval (K \cup \sk_{n}(L)) \cup \{\varepsilon\}(K\cup \sk_{n+1}(L))
		\]
		so that we get two commutative squares
		\begin{equation}\label{Equation:Double square pushout Step3}
			\begin{tikzcd}[column sep=tiny]
				{\displaystyle \coprod_{\Sigma^{n+1}} \interval(\partial\Delta^{n+1,k})_x} && {\interval\left(K \cup \text{Sk}_n(L)\right)} \\
				{\displaystyle \coprod_{\Sigma^{n+1}} \interval(\partial\Delta^{n+1,k})_x \cup \{\varepsilon \}(\Delta^{n+1,k})_x} && {\interval\left(K \cup \text{Sk}_n(L)\right) \cup \{\varepsilon\}(K\cup \text{Sk}_{n+1}(L))} \\
				{\displaystyle \coprod_{\Sigma^{n+1}} \interval(\Delta^{n+1,k})_x} && {\interval\left(K \cup \text{Sk}_{n+1}(L)\right)}
				\arrow[from=3-1, to=3-3]
				\arrow[from=1-1, to=1-3]
				\arrow[from=1-1, to=2-1]
				\arrow[from=2-1, to=3-1]
				\arrow[from=2-1, to=2-3]
				\arrow[from=1-3, to=2-3]
				\arrow[from=2-3, to=3-3]
			\end{tikzcd}
		\end{equation}
		where the outer square is a pushout. We claim that the top square in \ref{Equation:Double square pushout Step3} is also a pushout diagram. Indeed, since the functor $\{\varepsilon\}$ also preserves colimits and monomorphisms, we have two pushout squares 
		\begin{equation}\label{Equation:step3 double pushout 1}
			\begin{tikzcd}[row sep=small,column sep=tiny]
				{\displaystyle \coprod_{\Sigma^{n+1}} \{\varepsilon\}(\partial\Delta^{n+1,k})_x} && {\{\varepsilon\}\left(K \cup \text{Sk}_n(L)\right)} && {\interval(K \cup \text{Sk}_{n}(L))} \\
				\\
				{\displaystyle \coprod_{\Sigma^{n+1}}\{\varepsilon\}(\Delta^{n+1,k})_x} && {\{\varepsilon\}\left(K \cup \text{Sk}_{n+1}(L)\right)} && {\interval(K \cup \text{Sk}_{n}(L)) \cup \{\varepsilon\}\left(K \cup \text{Sk}_{n+1}(L)\right) }
				\arrow[from=1-1, to=3-1]
				\arrow[from=1-3, to=3-3,hookrightarrow]
				\arrow[from=3-1, to=3-3]
				\arrow[from=1-1, to=1-3]
				\arrow["\lrcorner"{anchor=center, pos=0.125, rotate=180}, draw=none, from=3-3, to=1-1]
				\arrow[from=1-3, to=1-5,hookrightarrow]
				\arrow[from=3-3, to=3-5]
				\arrow["\lrcorner"{anchor=center, pos=0.125, rotate=180}, draw=none, from=3-5, to=1-3]
				\arrow[from=1-5, to=3-5]
		\end{tikzcd}\end{equation}
		Moreover, the outer square in this last diagram also fits in a diagram
		\begin{equation}\label{Equation:Step3 double pushout second diagram}
			\begin{tikzcd}[column sep=tiny]
				{\displaystyle \coprod \{\varepsilon\}(\partial\Delta^{n+1,k})_x} & {\displaystyle \coprod \interval(\partial\Delta^{n+1,k})_x} & {\interval(K \cup \text{Sk}_{n}(L))} \\
				\\
				{\displaystyle \coprod \{\varepsilon\}(\Delta^{n+1,k})_x} & {\displaystyle \coprod \interval(\partial\Delta^{n+1,k})_x \cup \{\varepsilon\}(\Delta^{n+1,k})_x} & {\interval(K \cup \text{Sk}_{n}(L)) \cup \{\varepsilon\}\left(K \cup \text{Sk}_{n+1}(L)\right) }
				\arrow[hook, from=1-1, to=3-1]
				\arrow[from=1-3, to=3-3]
				\arrow[hook, from=1-1, to=1-2]
				\arrow[from=3-1, to=3-2]
				\arrow[from=1-2, to=3-2]
				\arrow["\lrcorner"{anchor=center, pos=0.125, rotate=180}, draw=none, from=3-2, to=1-1]
				\arrow[from=1-2, to=1-3]
				\arrow[from=3-2, to=3-3]
		\end{tikzcd}\end{equation}
		where the left most square is canonically a pushout square. Thus, by \ref{Equation:step3 double pushout 1} and \ref{Equation:Step3 double pushout second diagram} we conclude that the top square in \ref{Equation:Double square pushout Step3} is indeed a pushout diagram. Moreover, the bottom square in \ref{Equation:Double square pushout Step3} is then also a pushout square as the outer and top squares are.
		
		Finally, notice that the lower right vertical arrow in \ref{Equation:Double square pushout Step3} is part of a canonical pushout square 
		\begin{equation}\label{Equation:pushout2 in step 3}
			\begin{tikzcd}[]
				{\interval\left(K \cup \text{Sk}_n(L)\right) \cup \{\varepsilon\}(K\cup \text{Sk}_{n+1}(L))} & {\interval\left(K \cup \text{Sk}_n(L)\right) \cup \{\varepsilon\}(L)} \\
				\\
				{\interval\left(K \cup \text{Sk}_{n+1}(L)\right)} & {\interval\left(K \cup \text{Sk}_{n+1}(L)\right) \cup \{\varepsilon\}(L)}
				\arrow[from=1-1, to=3-1]
				\arrow[from=1-1, to=1-2]
				\arrow[from=1-2, to=3-2,"\varphi_n"]
				\arrow[from=3-1, to=3-2]
				\arrow["\lrcorner"{anchor=center, pos=0.125, rotate=180}, draw=none, from=3-2, to=1-1]
		\end{tikzcd}\end{equation}
		Combining the diagram \ref{Equation:pushout2 in step 3} and the fact that its bottom square is a pushout diagram we get that the following diagram is also a pushout: 
		\[\begin{tikzcd}
			{{\displaystyle \coprod_{ \Sigma^{n+1}} \interval(\partial\Delta^{n+1,k})_x \cup \{\varepsilon \}(\Delta^{n+1,k})_x}} & {\interval\left(K \cup \text{Sk}_n(L)\right) \cup \{\varepsilon\}(L)} \\
			\\
			{{\displaystyle \coprod_{ \Sigma^{n+1}}\interval(\Delta^{n+1,k})_x}} & {\interval\left(K \cup \text{Sk}_{n+1}(L)\right) \cup \{\varepsilon\}(L)}
			\arrow["{\varphi_n}", from=1-2, to=3-2]
			\arrow[from=3-1, to=3-2]
			\arrow[from=1-1, to=1-2]
			\arrow[from=1-1, to=3-1]
			\arrow["\lrcorner"{anchor=center, pos=0.125, rotate=180}, draw=none, from=3-2, to=1-1]
		\end{tikzcd}\]
		In this last diagram the left vertical arrow belongs to the class $M_{\catcal{B}}$. Again by saturation we conclude that $\varphi_n$ also lies in $M_{\varphi}$ as this last one is closed under pushouts. As a consequence $M_{\catcal{C}} \subseteq M_{\catcal{B}}$. The fact that $M_{\catcal{B}}\subseteq M_{\catcal{C}}$ follows trivially from their definitions. 
	\end{enumerate}

	We conclude that $ M_{\catcal{A}} = M_{\catcal{B}} = M_{\catcal{C}}$.
\end{proof}


\begin{proposition}\label{Proposition: proof of An3 using lemma}
	Let $\catcal{D}$ be the class of morphisms $X \lto Y$ in $\ssetisov$ such that $\interval X \cup \partial \interval(Y) \lto \interval Y$ is an isovariant anodyne extension. Then $\catcal{D}$ is closed under retracts, pushouts and transfinite compositions. Furthermore, $\catcal{D}$ contains the class $\catcal{B}$ from \Cref{lemma: saturation of admissible horn inclusions is the same as saturation of cisinski's class}.
\end{proposition}	

\begin{proof}
	Recall that we have a canonical morphism of functors $\partial \interval \lto \interval$ given by the collection of morphisms $X \coprod X \overset{(\partial^0_X,\partial^1_X)}{\lto}\interval X$ for each $X \in \ssetisov$. The class $\catcal{D}$ consist then of all morphism $X \lto Y$ such that the dotted morphism deduced from the commutative square: 
	\[\begin{tikzcd}
		{\partial \interval(X)} && {\partial \interval(Y)} \\
		& {\displaystyle \interval X\coprod_{\partial \interval(X)}\partial \interval(Y)} \\
		{\interval X} && {\interval Y}
		\arrow["{(\partial^0_X,\partial^1_X)}"', from=1-1, to=3-1]
		\arrow[from=3-1, to=3-3]
		\arrow["{(\partial^0_Y,\partial^1_Y)}", from=1-3, to=3-3]
		\arrow[from=1-1, to=1-3]
		\arrow[dashed, from=2-2, to=3-3]
		\arrow[from=1-3, to=2-2]
		\arrow[from=3-1, to=2-2]
		\arrow["\lrcorner"{anchor=center, pos=0.125, rotate=180}, draw=none, from=2-2, to=1-1]
	\end{tikzcd}\]
	is an isovariant anodyne extension. The fact that $\catcal{D}$ is saturated then follows directly from Corollary 1.1.8 in \cite{Cisinski2016LPCMDTH} and the fact that the class of isovariant anodyne extensions is saturated. Now let $\partial \Delta^{n,k}\lto \Delta^{n,k}$ be a boundary inclusion of isovariant simplicial sets. Notice that $\interval \Delta^{n,k}$ is a normal isovariant simplicial set as $\Delta^{n,k}$ is. Thus the canonical inclusion 
	\[\interval (\partialbf{\Delta^{n,k})} \cup \partial\interval (\Delta^{n,k}) \linj \interval(\Delta^{n,k})\]
	is also normal. To show that $\catcal{B} \subseteq \mathcal{D}$ we have to show that 
	\begin{equation}\label{Equation: class aux D to show An3}
		\varphi: \interval\left(\interval(\partialbf\Delta^{n,k})\cup \{\varepsilon\}(\Delta^{n,k})\right)\cup \partial \interval(\interval \Delta^{n,k}) \lto \interval(\interval \Delta^{n,k}) 
	\end{equation}
	is an isovariant anodyne extension. To see this, first notice that we have canonical isomorphisms 
	\[\interval(\partial \interval(X)) \simeq \partial \interval (\interval X) \simeq (\interval X)\coprod(\interval X).\]
	Therefore, using the fact that $\interval$ commutes with colimits and also that colimits commute with colimits we have that the left hand side in \ref{Equation: class aux D to show An3} writes as: 
	
	\begin{align*}
		\interval\left(\interval(\partialbf\Delta^{n,k})\cup \{\varepsilon\}(\Delta^{n,k})\right) \cup \partial \interval(\interval \Delta^{n,k}) &\simeq  \interval\left(\interval(\partialbf\Delta^{n,k})\cup \{\varepsilon\}(\Delta^{n,k})\right) \cup \interval(\partial \interval (\Delta^{n,k})) \\ 
		& \simeq \interval\left( \interval(\partialbf \Delta^{n,k}) \cup \{\varepsilon\} (\Delta^{n,k}) \cup \partial \interval(\Delta^{n,k})\right) \\ 
		& \simeq  \interval\left( \interval(\partialbf \Delta^{n,k})  \cup \partial \interval(\Delta^{n,k}) \cup \{\varepsilon\}(\Delta^{n,k}) \right) \\
		& \simeq \interval(\interval(\partialbf \Delta^{n,k}) \cup \partial \interval(\Delta^{n,k})) \cup \{\varepsilon\}(\interval \Delta^{n,k})
	\end{align*}  
	Where in this last line we have also used the fact that $\interval$ commutes with itself in such a way that the morphism $\varphi$ from \ref{Equation: class aux D to show An3} writes as: 
	\[ \varphi: \interval(\interval(\partialbf \Delta^{n,k}) \cup \partial \interval(\Delta^{n,k})) \cup \{\varepsilon\}(\interval \Delta^{n,k}) \lto \interval(\interval \Delta^{n,k}) \]
	and since $\interval (\partialbf{\Delta^{n,k})} \cup \partial\interval (\Delta^{n,k}) \linj \interval(\Delta^{n,k})$ is normal, by definition of the class $\catcal{C}$, $\varphi$ belongs to $\catcal{C} \subseteq M_{\catcal{C}}$. Thus, by Lemma \ref{lemma: saturation of admissible horn inclusions is the same as saturation of cisinski's class} the map $\varphi$ also belongs to the class $M_{\catcal{A}}$ which is precisely the class of isovariant anodyne extensions. In other words, $\varphi$ is an isovariant anodyne extension. 
\end{proof}

\begin{corollary}\label{Corollary:An3}
	The class $\mathcal{D}$ contains the class of isovariant anodyne extensions. In particular for any $X\lto Y$ isovariant anodyne the canonical morphism $\interval X \cup \partial \interval(Y) \lto \interval Y$ is isovariant anodyne.
\end{corollary}
\begin{proof}
	This follow directly from Proposition \ref{Proposition: proof of An3 using lemma} namely, the class $\catcal{D}$ is saturated and contains the class $\catcal{B}$ so that by minimality $M_{\catcal{A}}=M_\catcal{B} \subseteq \catcal{D}$.
\end{proof}

Let $B$ be a normal isovariant simplicial set and denote by $\text{An}_B$ the class of morphisms in $\ssetisov \downarrow B$ whose image under the forgetful functor $U_B$ is an isovariant anodyne extension. We will usually refer to this class $\An_B$ as the class of \emph{fiberwise isovariant anodyne extensions}. Using the results we have shown in last few sections and the following result from \cite{Cisinski2016LPCMDTH} we can show that $\An_B$ is a class of $\interval_B$-anodyne extensions. 

\begin{lemma}[{\cite[\nopp 1.3.52]{Cisinski2016LPCMDTH}}]\label{Lemma:Cisinski 1.3.52}
	Let $\catname{A}$ be a small category, $\catcal{F}$ be a class of presheaf morphisms in $\Psh(\catname{A})$ and let $X$ any presheaf over $\catname{A}$. If $\catcal{F}\downarrow X$ denotes the class of morphisms in $\Psh(\catname{A})\downarrow X$ of the form 
	\[\begin{tikzcd}
		K && L \\
		& X
		\arrow["f", from=1-1, to=1-3]
		\arrow[from=1-1, to=2-2]
		\arrow[from=1-3, to=2-2]
	\end{tikzcd}\]
	where $f \in \catcal{F}$. Then $L(R(\catcal{F}\downarrow X)) = L(R(\catcal{F}))\downarrow X$. 
\end{lemma}

\begin{proposition}\label{Proposition:Admissible horn inclusions is a class of anodyne extensions}
	The class of morphisms $\text{An}_B$ is a class of $\ \interval_B$-anodyne extensions for the category $\ssetisov\downarrow B$. That is, $\text{An}_B$ satisfies the axioms of \cite[Définition 1.3.10]{Cisinski2016LPCMDTH}
\end{proposition}

\begin{proof}
	The first axiom An0 follows directly from \Cref{{Lemma:Cisinski 1.3.52},{lemma: saturation of admissible horn inclusions is the same as saturation of cisinski's class}}: 
	\[ L(R(\catcal{A}\downarrow B)) = L(R(\catcal{A}))\downarrow B = \An_B.\]
	The axiom An1 follows again from \Cref{lemma: saturation of admissible horn inclusions is the same as saturation of cisinski's class} and the properties \enumref{Paragraph:Properties_normal_monomorphisms}{2} and \enumref{Paragraph:Properties_normal_monomorphisms}{3} of normal monomorphisms for if $(X \rightarrow B) \linj (Y \rightarrow B)$ is a monomorphism in $\ssetisov \downarrow B$, then $X$ and $Y$ are also normal as $B$ is. Finally the last axiom An2 correspond precisely to \Cref{Corollary:An3}.
	
\end{proof}

\subsection{The Model Structure}

In this section we collect the analogue definitions as in \cite[Définition 1.3.21]{Cisinski2016LPCMDTH} in the context of isovariant simplicial sets over a fixed normal object $B$. Next we shall use all the results we have shown in the last sections to justify the existance of a cofibrantly generated model structure in $\ssetisov \downarrow B$ by using \cite[Théorème 1.3.22]{Cisinski2016LPCMDTH}. 

\begin{definition}\label{Definition: isovariant weak equivalences, cofibs, fibs}
	Let $B$ be a normal isovariant simplicial set and recall that an isovariant simplicial set map $f:X \lto Y$ over $B$ is called:
	\begin{enumerate}[label=\textbf{\arabic*.}, ref=\arabic*.]
		\item Isovariant cofibration if $f$ is a monomorphism (equivalently normal monomorphism). 
		\item Trivial fibration if $f$ has the right lifting property with respect to all cofibrations (normal monomorphisms). 
		\item Isovariant $B$-fibration if $f$ has the right lifting property with respect to the class of fiberwise isovariant anodyne extensions (the class $\An_B$). 
		
		Moreover an isovariant simplicial set $X$ over $B$ will be called \emph{fibrant} if the structural map $X \lto B$ is an isovariant fibration (\Cref{Definition:isovariant anodyne extension}). 
		\item Weak equivalence if $f$ is a fiberwise homotopy equivalence, that is, for any isovariant fibration $W \lto B$ (equivalently $W\lto B$ is fibrant in $\ssetisov \downarrow B$) the induced morphism 
		\[ [Y,W]_B \lto [X,W]_B\]
		is a bijection.
	\end{enumerate}
\end{definition}

\begin{theorem}\label{Theorem: Model structure for ssisov}
	For any $B \in \ssetisov$ normal. The over category of isovariant simplicial sets $\ssetisov \downarrow B$ has a left proper cofibrantly generated model category structure where the cofibrations are the isovariant cofibrations defined above, the weak equivalences are the fiberwise homotopy equivalences and the fibrant objects are the isovariant simplicial sets such that the structural map is an isovariant fibration (see \Cref{Definition:isovariant anodyne extension}). 
\end{theorem}

\begin{proof}
	This follows from \cite[Théorème 1.3.22]{Cisinski2016LPCMDTH} applied to the category of isovariant simplicial sets $\ssetisov \downarrow B$ together with the homotopical data $(\interval_B, \An_B)$ constructed in \Cref{{Proposition:Exact cylinder of over category},{Proposition:Admissible horn inclusions is a class of anodyne extensions}}.
\end{proof}

We would like to transfer this last model structure to the category of isovariant simplicial sets by using the transfer Theorem due to Sjoerd Crans in \cite{CRANS199535}. We will use a mild variation from \cite[\S 9.1 ]{dwyer1997model}. 

\begin{para}\label{Paragraph: adjuctions and normalizations}\normalfont
	Let $B$ be a normal isovariant simplicial set and denote by $U:B \lto \ast$ the canonical isovariant simplicial set from $B$ to the terminal object $\ast$ in $\ssetisov$. We have a triple adjunction \cite[IV \S 7 Theorem 2]{MacLane1994}:	
	\begin{equation}\label{Eq:basechangeadjunction}
	\begin{tikzcd}
		{\ssetisov\downarrow B} && \ssetisov
		\arrow[""{name=0, anchor=center, inner sep=0}, "{U_{!}}", bend right={-22pt}, from=1-1, to=1-3]
		\arrow[""{name=1, anchor=center, inner sep=0}, "{U^{\ast}}"'{pos=0.3}, from=1-3, to=1-1]
		\arrow[""{name=2, anchor=center, inner sep=0}, "{U_{\ast}}"', bend right={22pt}, from=1-1, to=1-3]
		\arrow["\dashv"{anchor=center, rotate=-90}, draw=none, from=0, to=1]
		\arrow["\dashv"{anchor=center, rotate=-90}, draw=none, from=1, to=2]
	\end{tikzcd}
	\end{equation}
%
%
	where $U^{\ast}$ denotes the change-of-base functor. In this particular case it is simply given, on objects, by $X\longmapsto X \times B$. The structural map is the evident projection. The functor $U_{!}$ is given by composition with $U:B\lto \ast$. In this particular case $U_{!}$ corresponds to the usual forgetful functor. 
	
	Finally, notice that the category $\ssetisov$ has inner hom objects. These are given as in any category of presheaves (see \cite[I. \S 6 Proposition 1]{MacLane1994}) by: 
	
	\[
	\Map(X,Y):  \begin{array}{ll}
		G\Deltacat^{\op} &\longrightarrow \Set \\ {[n]_k} &\longmapsto \Hom_{\ssetisov}(X \times \Delta^{n,k},Y)
	\end{array}
	\]
	If $p_X: X \lto B$ is an object of $\ssetisov \downarrow B$ then the right adjoint $U_{\ast}$ is then given, on objects, by the pullback: 
	
	\[\begin{tikzcd}
		{U_{\ast}(X)} & {\Map(B,X)} \\
		\ast & {\Map(B,B)}
		\arrow[dashed, from=1-1, to=1-2]
		\arrow["{p_X \circ -}", from=1-2, to=2-2]
		\arrow[dashed, from=1-1, to=2-1]
		\arrow["{«\idd_B»}"', from=2-1, to=2-2]
		\arrow["\lrcorner"{anchor=center, pos=0.125}, draw=none, from=1-1, to=2-2]
	\end{tikzcd}\]
	The triple adjunction $U_{!}\dashv U^{\ast} \dashv U_{\ast}$ induces an adjunction
	\begin{equation}\label{Eq:adjunction fiberwise to itself}
	\adjto{U^{\ast} \circ U_{!}}{\ssetisov \downarrow B}{\ssetisov \downarrow B}{U^{\ast} \circ U_{\ast}}
	\end{equation}
\end{para}

	\begin{proposition}\label{Prop:fiberwise adjunction is quillen}
		The adjunction from \eqref{Eq:adjunction fiberwise to itself} is a Quillen adjunction with respect to the model structure given by \emph{\Cref{Theorem: Model structure for ssisov}}.
	\end{proposition}
	
	\begin{proof}
		Using \cite[Corollary A.2]{Dugger2001}, it suffices to show that the left adjoint $U^{\ast}\circ U_{!}$ preserves both cofibrations between cofibrant objects and trivial cofibrations. Notice that a cofibration between cofibrant objects $f: (X,p_X)\lto (Y,p_Y)$ amounts to say that $f:X \lto Y$ is a normal monomorphism and both $X$ and $Y$ are normal. A direct application of \enumref{Paragraph:Properties_normal_monomorphisms}{1} and \enumref{Paragraph:Properties_normal_monomorphisms}{2} implies that $X \times B$ and $Y\times B$ are both normal and the morphism $(U^{\ast} \circ U)(f)\coloneq (f, \idd_B)$:
		\[\begin{tikzcd}
			{X \times B} && {Y \times B} \\
			& B
			\arrow["{\pr_2}"', from=1-1, to=2-2]
			\arrow["{\pr_2}", from=1-3, to=2-2]
			\arrow["{(U^{\ast} \circ U)(f) }", from=1-1, to=1-3]
		\end{tikzcd}\]
		is a normal monomorphism. This implies that $(U^{\ast} \circ U)(f)$ is a cofibration in $\ssetisov \downarrow B$. Additionally, if $f$ is a trivial cofibration, by using \cite[Proposition 2.4.15]{Cisinski2019} it it easy to show that $f \times \idd_B$ is also a fiberwise homotopy equivalence.
	\end{proof}	
	
We would like to find a suitable normal simplicial set $B$ in such a way that the adjunction $(U_{!} \dashv U^{\ast})$ satisfies the conditions of the right-transfer theorem \cite[\S 9.1 ]{dwyer1997model} and thus $\ssetisov$ will admit a right-transfered cofibrantly generated model category structure. A necessary condition for the above to be true is that the morphism $B \lto \ast$ is a weak equivalence. This is exactly the same situation as in \cite[Proposition 3.12]{Cisinski2011}.

Recall that the class of normal monomorphisms in $\ssetisov$ form a saturated class and moreover it is the smallest saturated class containing the boundary inclusions. The usual small object argument then implies that any isovariant simplicial set map $X \lto Y$ factors as a normal monomorphism $X \lto Z$ followed by a morphism $Z \lto Y$ having the right lifting property with respect to all normal monomorphisms. In particular if we consider the canonical normal monomorphism $\emptyset \lto X$ from the initial object of $\ssetisov$ to any $X$ in $\ssetisov$ there exists $\widehat{X}$ normal and $\widehat{X} \lto X$ such that the diagram
\[\begin{tikzcd}
	\emptyset && X \\
	& {\widehat{X}}
	\arrow[tail, from=1-1, to=2-2]
	\arrow[from=1-1, to=1-3]
	\arrow[two heads, from=2-2, to=1-3]
\end{tikzcd}\]
commutes and $\widehat{X}\lto X$ has the right lifting property with respect to normal monomorphisms (equivalently it is a trivial fibration).
Such $\widehat{X}$ will be called a \emph{normalization} of $X$. These normalizations turn out to be unique up to $\interval$-homotopy equivalence.

\begin{lemma}[{\cite[Lemma 8.1.5]{Moerdijk2010}}]\label{Lemma:normalizations are unique up to homotopy}
	Let $\widehat{X}$ and $\widetilde{X}$ be two normalizations of an isovariant simplicial set $X$. Then there exists an isovariant homotopy equivalence $\widehat{X}\lto \widetilde{X}$.
\end{lemma}

\begin{proposition}\label{Prop:quillenequivalencetoitself}
	Let $E^{\infty}$ be a normalization of the terminal isovariant simplicial set $\ast$ and $\begin{tikzcd}
		U: E^{\infty} \arrow[r,two heads] &\ast
	\end{tikzcd}$ the normalization morphism. Then the adjunction from \eqref{Eq:adjunction fiberwise to itself}: \[\adjto{U^{\ast} \circ U_{!}}{\ssetisov \downarrow E^{\infty}}{\ssetisov \downarrow E^{\infty}}{U^{\ast} \circ U_{\ast}} \]
	is a Quillen equivalence. 
\end{proposition}

\begin{proof}
	Let us denote by $L \coloneqq U^{\ast} \circ U_{!}$ and $R \coloneqq U^{\ast} \circ U_{\ast}$. By using \cite[Corollary 1.3.16]{Hovey1999}, it suffices to check that:
	\begin{enumerate}[label=\stlabel{Prop:quillenequivalencetoitself}, ref={\arabic*}]
		\item \label{Prop:quillenequivalencetoitself.1} The left adjoint $L$ reflects weak equivalences between cofibrant objects and,
		\item \label{Prop:quillenequivalencetoitself.2} For every fibrant object $Y$, the morphism $LQR(Y) \lto Y$ is a weak equivalence. Where $Q$ denotes the cofibrant replacement functor on $\ssetisov \downarrow B$.
	\end{enumerate}
	   Nevertheless, notice that any object $X\lto E^{\infty}$ is cofibrant for $E^{\infty}$ is normal and \enumref{Paragraph:Properties_normal_monomorphisms}{2} implies that $X$ is normal as well. As a consequence the canonical morphism $\emptyset \lto X$ in $\ssetisov \downarrow E^{\infty}$ is a cofibration. Therefore, \enumref{Prop:quillenequivalencetoitself}{2} reduces to check that for any fibrant object $Y$, the morphism $LR(Y) \lto Y$ is a weak equivalence. Let us first prove \enumref{Prop:quillenequivalencetoitself}{1}. Suppose $f:X \lto Y$ is a morphism in $\ssetisov \downarrow E^{\infty}$ such that $L(f): X \times E^{\infty} \lto Y \times E^{\infty}$ is a weak equivalence in $\ssetisov\downarrow E^{\infty}$. Since $X$ and $Y$ are normal, the morphisms $(\idd_X, p_X):X\lto X \times E^{\infty}$ and $(\idd_Y, p_Y): Y \lto Y\times E^{\infty}$ are isovariant anodyne extensions. This follows in exactly the same way as in \cite[Lemma 3.9]{Cisinski2011} because $(\idd_X, p_X)$ and $(\idd_Y, p_Y)$ are, respectively, sections of the projection maps $X \times E^{\infty} \lto X$ and $Y\times E^{\infty} \lto Y$ and these last projections are trivial fibrations by definition of $E^{\infty}$. Moreover, since isovariant anodyne extensions are weak equivalences (see \cite[Proposition 1.3.31]{Cisinski2016LPCMDTH}) we have that $(\idd,p_X)$ and $(\idd,p_Y)$ are both weak equivalences. Finally we have a commutative diagram 
	   \[\begin{tikzcd}
	   	X & Y \\
	   	{X \times E^{\infty}} & {Y \times E^{\infty}}
	   	\arrow["f", from=1-1, to=1-2]
	   	\arrow["{(\idd,p_X)}"', from=1-1, to=2-1]
	   	\arrow["{L(f)}"', from=2-1, to=2-2]
	   	\arrow["{(\idd,p_Y)}", from=1-2, to=2-2]
	   	\arrow[bend right={-30pt}, from=2-2, to=1-2]
	   	\arrow[bend right={30pt}, from=2-1, to=1-1]
	   \end{tikzcd}\]
	   Notice that this diagram is commutative even in the category $\ssetisov \downarrow E^{\infty}$. By the 2-out-of-3 property for weak equivalences we conclude that $f$ is a weak equivalence in $\ssetisov \downarrow E^{\infty}$. 
	   
	   Suppose now that $p_Y: Y\lto E^{\infty}$ is fibrant. Recall that this means that $p_Y$ has the right lifting property with respect to isovariant anodyne extensions. Notice that this implies that the projection $U_{\ast}(Y)\times E^{\infty} \lto E^{\infty}$ also has the right lifting property with respect to isovariant anodyne extensions so that $U_{\ast}(Y)\times E^{\infty} \lto E^{\infty}$ is fibrant. Moreover, the canonical morphism $U_{\ast}(Y) \times E^{\infty} \lto Y$ in $\ssetisov \downarrow E^{\infty}$ is a weak equivalence and the morphism $LR(Y)\lto Y$ writes as the composite: 
	   
	\[\begin{tikzcd}
		{U_{\ast}(Y)\times E^{\infty}\times E^{\infty}} & {U_{\ast}(Y)\times E^{\infty}} & Y
		\arrow["\sim", two heads, from=1-1, to=1-2]
		\arrow["\sim", from=1-2, to=1-3]
	\end{tikzcd}\]
and again by \cite[Lemma 3.9]{Cisinski2011} the map ${U_{\ast}(Y)\times E^{\infty}\times E^{\infty}} \lto U_{\ast}(Y) \times E^{\infty}$ is a weak equivalence. As a consequence, $LR(Y) \lto Y$ is a weak equivalence.
\end{proof}

\begin{corollary}\label{Corollary:modelstructureabsolute}
	The adjunction $\adjto{U_{!}}{\ssetisov \downarrow E^{\infty}}{\ssetisov}{U^{\ast}}$ satisfies the hypothesis of \emph{\cite[\S 9.1]{dwyer1997model}}. In particular $\ssetisov$ is endowed with a cofibrantly generated model category structure in which the cofibrations are the normal monomorphisms, the fibrant objects are the $\interval$-fibrant objects, and the fibrations between fibrant objects are the isovariant fibrations. Weak equivalences are given by the isovariant homotopy equivalences and in particular for any $X$ normal and $Y$ fibrant we have that 
	\[[X,Y]_{\emph{isov}} = \Hom_{\catname{Ho}(\ssetisov)}(X,Y).\]
\end{corollary}

	\DeclareFieldFormat{labelnumberwidth}{#1}
	\printbibliography[keyword=alph, heading=references]
	\DeclareFieldFormat{labelnumberwidth}{{#1\adddot\midsentence}}
	\printbibliography[heading=references, notkeyword=alph]

@string{C = {Carbon}}

@string{F = {Fuel}}

@string{I = {Ind. Eng. Chem. Res.}}

@book {Hovey1999,
	AUTHOR = {Hovey, Mark},
	TITLE = {Model categories},
	SERIES = {Mathematical Surveys and Monographs},
	VOLUME = {63},
	PUBLISHER = {American Mathematical Society, Providence, RI},
	YEAR = {1999},
	PAGES = {xii+209},
	ISBN = {0-8218-1359-5},
	MRCLASS = {55U35 (18D15 18G30 18G55)},
	MRNUMBER = {1650134},
	MRREVIEWER = {Teimuraz\ Pirashvili},
}

@article {Dugger2001,
	AUTHOR = {Dugger, Daniel},
	TITLE = {Replacing model categories with simplicial ones},
	JOURNAL = {Trans. Amer. Math. Soc.},
	FJOURNAL = {Transactions of the American Mathematical Society},
	VOLUME = {353},
	YEAR = {2001},
	NUMBER = {12},
	PAGES = {5003--5027},
	ISSN = {0002-9947,1088-6850},
	MRCLASS = {55U35 (18G30 18G55)},
	MRNUMBER = {1852091},
	MRREVIEWER = {David\ A.\ Blanc},
	DOI = {10.1090/S0002-9947-01-02661-7},
	URL = {https://doi.org/10.1090/S0002-9947-01-02661-7},
}

@book {MacLane1994,
	AUTHOR = {Mac Lane, Saunders and Moerdijk, Ieke},
	TITLE = {Sheaves in geometry and logic},
	SERIES = {Universitext},
	NOTE = {A first introduction to topos theory,
	Corrected reprint of the 1992 edition},
	PUBLISHER = {Springer-Verlag, New York},
	YEAR = {1994},
	PAGES = {xii+629},
	ISBN = {0-387-97710-4},
	MRCLASS = {03G30 (18B25 54B40)},
	MRNUMBER = {1300636},
	MRREVIEWER = {M.\ Makkai},
}

@book {borceux1994handbook,
	AUTHOR = {Borceux, Francis},
	TITLE = {Handbook of categorical algebra. 1},
	SERIES = {Encyclopedia of Mathematics and its Applications},
	VOLUME = {50},
	NOTE = {Basic category theory},
	PUBLISHER = {Cambridge University Press, Cambridge},
	YEAR = {1994},
	PAGES = {xvi+345},
	ISBN = {0-521-44178-1},
	MRCLASS = {18-02 (18Axx)},
	MRNUMBER = {1291599},
	MRREVIEWER = {Martin\ Hyland}
}

@incollection {Browder1973,
	AUTHOR = {Browder, William and Quinn, Frank},
	TITLE = {A surgery theory for {$G$}-manifolds and stratified sets},
	BOOKTITLE = {Manifolds---{T}okyo 1973 ({P}roc. {I}nternat. {C}onf.,
	{T}okyo, 1973)},
	PAGES = {27--36},
	PUBLISHER = {Published for the Mathematical Society of Japan by University
	of Tokyo Press, Tokyo},
	YEAR = {1975},
	MRCLASS = {57D65},
	MRNUMBER = {375348},
	MRREVIEWER = {Edgar\ H.\ Brown, Jr.}
}

@article {Cisinski2016LPCMDTH,
	AUTHOR = {Cisinski, Denis-Charles},
	TITLE = {Les pr\'{e}faisceaux comme mod\`eles des types d'homotopie},
	JOURNAL = {Ast\'{e}risque},
	FJOURNAL = {Ast\'{e}risque},
	NUMBER = {308},
	YEAR = {2006},
	PAGES = {xxiv+390},
	ISSN = {0303-1179,2492-5926},
	ISBN = {978-2-85629-225-9},
	MRCLASS = {55-02 (18F20 18G50 55P60 55U35)},
	MRNUMBER = {2294028},
	MRREVIEWER = {Philippe\ Gaucher},
}

@book {Cisinski2019,
	AUTHOR = {Cisinski, Denis-Charles},
	TITLE = {Higher categories and homotopical algebra},
	SERIES = {Cambridge Studies in Advanced Mathematics},
	VOLUME = {180},
	PUBLISHER = {Cambridge University Press, Cambridge},
	YEAR = {2019},
	PAGES = {xviii+430},
	ISBN = {978-1-108-47320-0},
	MRCLASS = {18-02 (18D05 18F20 18G55 55U10 55U35)},
	MRNUMBER = {3931682},
	MRREVIEWER = {Charles\ Rezk},
	DOI = {10.1017/9781108588737},
	URL = {https://doi.org/10.1017/9781108588737}
}

@article {Cisinski2011,
	AUTHOR = {Cisinski, Denis-Charles and Moerdijk, Ieke},
	TITLE = {Dendroidal sets as models for homotopy operads},
	JOURNAL = {J. Topol.},
	FJOURNAL = {Journal of Topology},
	VOLUME = {4},
	YEAR = {2011},
	NUMBER = {2},
	PAGES = {257--299},
	ISSN = {1753-8416,1753-8424},
	MRCLASS = {55P48 (18D10 18G30 18G55 55U10 55U40)},
	MRNUMBER = {2805991},
	MRREVIEWER = {Julia\ Bergner},
	DOI = {10.1112/jtopol/jtq039},
	URL = {https://doi.org/10.1112/jtopol/jtq039}
}

@article {CRANS199535,
	AUTHOR = {Crans, Sjoerd E.},
	TITLE = {Quillen closed model structures for sheaves},
	JOURNAL = {J. Pure Appl. Algebra},
	FJOURNAL = {Journal of Pure and Applied Algebra},
	VOLUME = {101},
	YEAR = {1995},
	NUMBER = {1},
	PAGES = {35--57},
	ISSN = {0022-4049,1873-1376},
	MRCLASS = {18D15 (18F20)},
	MRNUMBER = {1346427},
	DOI = {10.1016/0022-4049(94)00033-F},
	URL = {https://doi.org/10.1016/0022-4049(94)00033-F}
}

@book {SchultzDovermann1990,
	AUTHOR = {Dovermann, Karl Heinz and Schultz, Reinhard},
	TITLE = {Equivariant surgery theories and their periodicity properties},
	SERIES = {Lecture Notes in Mathematics},
	VOLUME = {1443},
	PUBLISHER = {Springer-Verlag, Berlin},
	YEAR = {1990},
	PAGES = {vi+227},
	ISBN = {3-540-53042-8},
	MRCLASS = {57R67 (57S17)},
	MRNUMBER = {1077825},
	MRREVIEWER = {Amir\ H.\ Assadi},
	DOI = {10.1007/BFb0092354},
	URL = {https://doi.org/10.1007/BFb0092354},
}

@incollection {Schultz1992,
	AUTHOR = {Schultz, Reinhard},
	TITLE = {Isovariant homotopy theory and differentiable group actions},
	BOOKTITLE = {Algebra and topology 1992 ({T}aej\u{o}n)},
	PAGES = {81--148},
	PUBLISHER = {Korea Adv. Inst. Sci. Tech., Taej\u{o}n},
	YEAR = {1992},
	MRCLASS = {57S17 (55P91 57R67 57R91)},
	MRNUMBER = {1212981},
	MRREVIEWER = {James\ Frederic\ Davis},
}

@unpublished{dwyer1997model,
	author = {William G. Dwyer, Philip S. Hirschhorn, Daniel M. Kan},
	title = {Model Categories and More General Abstract Homotopy Theory},
	date = {1997},
}

@book {Moerdijk2010,
	AUTHOR = {Moerdijk, Ieke and To\"{e}n, Bertrand},
	TITLE = {Simplicial methods for operads and algebraic geometry},
	SERIES = {Advanced Courses in Mathematics. CRM Barcelona},
	EDITOR = {Casacuberta, Carles and Kock, Joachim},
	PUBLISHER = {Birkh\"{a}user/Springer Basel AG, Basel},
	YEAR = {2010},
	PAGES = {x+186},
	ISBN = {978-3-0348-0051-8},
	MRCLASS = {18G30 (14A20 14D20 18D50 18G55)},
	MRNUMBER = {2797154},
	MRREVIEWER = {Beno\^{i}t\ Fresse},
	DOI = {10.1007/978-3-0348-0052-5},
	URL = {https://doi.org/10.1007/978-3-0348-0052-5},
}

@article {douteau2020simplicial,
	AUTHOR = {Douteau, Sylvain},
	TITLE = {A simplicial approach to stratified homotopy theory},
	JOURNAL = {Trans. Amer. Math. Soc.},
	FJOURNAL = {Transactions of the American Mathematical Society},
	VOLUME = {374},
	YEAR = {2021},
	NUMBER = {2},
	PAGES = {955--1006},
	ISSN = {0002-9947,1088-6850},
	MRCLASS = {55U35 (18N40 18N50 57N80)},
	MRNUMBER = {4196384},
	MRREVIEWER = {Viktoriya\ Ozornova},
	DOI = {10.1090/tran/8264},
	URL = {https://doi.org/10.1090/tran/8264}
}

@misc{stacks-project,
	shorthand    = {Stacks},
	author       = {The {Stacks Project Authors}},
	title        = {\textit{Stacks Project}},
	howpublished = {\url{https://stacks.math.columbia.edu}},
	year         = {2018},
}

@article {Yeakel2019IET,
	AUTHOR = {Yeakel, Sarah},
	TITLE = {An isovariant {E}lmendorf's theorem},
	JOURNAL = {Doc. Math.},
	FJOURNAL = {Documenta Mathematica},
	VOLUME = {27},
	YEAR = {2022},
	PAGES = {613--628},
	ISSN = {1431-0635,1431-0643},
	MRCLASS = {55P91},
	MRNUMBER = {4432523},
}

@article {MH,
	AUTHOR = {Murphy-Hernandez, Frank},
	TITLE = {The category of {G}-posets},
	JOURNAL = {Algebra Universalis},
	FJOURNAL = {Algebra Universalis},
	VOLUME = {81},
	YEAR = {2020},
	NUMBER = {3},
	PAGES = {Paper No. 38, 21},
	ISSN = {0002-5240,1420-8911},
	MRCLASS = {06A11 (16B50)},
	MRNUMBER = {4120403},
	DOI = {10.1007/s00012-020-00669-3},
	URL = {https://doi.org/10.1007/s00012-020-00669-3},
}

@book {gabriel2012calculus,
	AUTHOR = {Gabriel, P. and Zisman, M.},
	TITLE = {Calculus of fractions and homotopy theory},
	SERIES = {Ergebnisse der Mathematik und ihrer Grenzgebiete [Results in
	Mathematics and Related Areas]},
	VOLUME = {Band 35},
	PUBLISHER = {Springer-Verlag New York, Inc., New York},
	YEAR = {1967},
	PAGES = {x+168},
	MRCLASS = {55.40 (18.00)},
	MRNUMBER = {210125},
	MRREVIEWER = {A.\ K.\ Bousfield}
}

@book {leinster2014basic,
	AUTHOR = {Leinster, Tom},
	TITLE = {Basic category theory},
	SERIES = {Cambridge Studies in Advanced Mathematics},
	VOLUME = {143},
	PUBLISHER = {Cambridge University Press, Cambridge},
	YEAR = {2014},
	PAGES = {viii+183},
	ISBN = {978-1-107-04424-1},
	MRCLASS = {18-01},
	MRNUMBER = {3307165},
	MRREVIEWER = {Philippe\ Gaucher},
	DOI = {10.1017/CBO9781107360068},
	URL = {https://doi.org/10.1017/CBO9781107360068},
}

@book {riehl_2014,
	AUTHOR = {Riehl, Emily},
	TITLE = {Categorical homotopy theory},
	SERIES = {New Mathematical Monographs},
	VOLUME = {24},
	PUBLISHER = {Cambridge University Press, Cambridge},
	YEAR = {2014},
	PAGES = {xviii+352},
	ISBN = {978-1-107-04845-4},
	MRCLASS = {18G55 (18D20 55U35)},
	MRNUMBER = {3221774},
	MRREVIEWER = {David\ A.\ Blanc},
	DOI = {10.1017/CBO9781107261457},
	URL = {https://doi.org/10.1017/CBO9781107261457},
}
	
\end{document}